\documentclass[12pt,a4paper]{article}
\usepackage{vmargin}
\usepackage{amsfonts,graphicx,amsmath,amssymb,amsthm,hyperref,color,nicefrac,enumerate,comment,mathrsfs}
\usepackage[latin1]{inputenc}
\usepackage[nocompress]{cite}
\usepackage{url}
\usepackage{bbm}
\usepackage{stmaryrd}
\usepackage{enumerate}
\newcommand{\R}{\mathbb{R}}

\newcommand{\N}{\mathbb{N}}
\newcommand{\E}{\mathbb{E}}

\newcommand{\smallsum}{\textstyle\sum}
\renewcommand{\P}{\mathbb{P}}

\newcommand\numberthis{\addtocounter{equation}{1}\tag{\theequation}}

\newcommand{\vertiii}[1]{{\left\vert\kern-0.25ex\left\vert\kern-0.25ex\left\vert {#1} 
		\right\vert\kern-0.25ex\right\vert\kern-0.25ex\right\vert}}
\newcommand{\vertiiibig}[1]{{\big\vert\kern-0.25ex\big\vert\kern-0.25ex\big\vert {#1} 
		\big\vert\kern-0.25ex\big\vert\kern-0.25ex\big\vert}}
\newcommand{\vertiiiStandard}[1]{{\vert\kern-0.25ex\vert\kern-0.25ex\vert {#1} 
		\vert\kern-0.25ex\vert\kern-0.25ex\vert}}

\newcommand{\ANNs}{\mathbf{N}}
\newcommand{\activation}{a}
\newcommand{\activationDim}[1]{\mathfrak{M}_{\activation,#1}}
\newcommand{\functionANN}{\mathcal{R}_{\activation}}
\newcommand{\paramANN}{\mathcal{P}}

\newcommand{\lengthANN}{\mathcal{L}}
\newcommand{\inDimANN}{\mathcal{I}}
\newcommand{\compANN}[2]{{#1 \bullet  #2}}

\newcommand{\paraANN}[1]{\mathbf{P}_{#1}}

\newcommand{\outDimANN}{\mathcal{O}}

\newcommand{\dims}{\mathcal{D}}
\newcommand{\hiddenLength}{\mathcal{H}}
\newcommand{\hiddenDimId}{\mathfrak{i}}

\newcommand{\parallelizationSpecial}{\mathbf{P}}

\newcommand{\idMatrix}{\operatorname{I}}

\newcommand{\qandq}{\qquad\text{and}\qquad}

\newcommand{\pa}[1]{\left({#1}\right)}

\newcommand{\SubsetANNs}{\mathfrak{N}}
\newcommand{\ExponDim}{\mathfrak{d}}
\newcommand{\ExponError}{\mathfrak{e}}
\newcommand{\ExponN}{\mathfrak{n}}
\newcommand{\Constant}{\mathfrak{C}}

\newcommand{\matrixANN}{\mathfrak{W}}
\newcommand{\vectorANN}{\mathfrak{B}}
\newcommand{\idRelu}{\mathfrak{I}}
\newcommand{\sumANN}{\mathfrak{S}}
\newcommand{\extensionANN}{\mathfrak{T}}

\newcommand{\dimANNlevel}{\mathbb{D}}

\numberwithin{equation}{section}
\newtheorem{lemma}{Lemma}[section]

\newtheorem{cor}[lemma]{Corollary}
\newtheorem{definition}[lemma]{Definition}

\newtheorem{theorem}[lemma]{Theorem}

\newtheorem{prop}[lemma]{Proposition}

\begin{document}
	
\title{Deep neural network approximations\\
	 for Monte Carlo algorithms}

\author{ Philipp Grohs$^1$, Arnulf Jentzen$^2$, and   Diyora Salimova$^3$
	\bigskip
	\\
	\small{$^1$Faculty of Mathematics and
		Research Platform Data Science,}\\
	\small{University of Vienna, Austria, e-mail:   philipp.grohs@univie.ac.at}
	\smallskip
	\\
	\small{$^2$Seminar for Applied Mathematics, Department of Mathematics,}\\
	\small{ETH Zurich, Switzerland, e-mail:   arnulf.jentzen@sam.math.ethz.ch}
	\smallskip
	\\
	\small{$^3$Seminar for Applied Mathematics, Department of Mathematics,}\\
	\small{ETH Zurich, Switzerland, e-mail:  diyora.salimova@sam.math.ethz.ch}}

\maketitle

\begin{abstract}
In the past few years deep artificial neural networks (DNNs) have been  successfully  employed in   a large number of computational problems including, e.g.,
language processing, 
image recognition, 
fraud detection, and computational advertisement.
Recently, it has also been proposed in the scientific literature 
to reformulate 
partial differential equations (PDEs) as stochastic learning problems and to employ DNNs together with 
stochastic gradient descent methods 
to approximate the solutions of such PDEs. There are also a few mathematical convergence results in the scientific literature  which show that  DNNs can approximate solutions of certain PDEs without the curse of dimensionality  in the sense that the number of real parameters employed to describe the DNN grows at most polynomially both in the PDE dimension $d \in \N$  and the reciprocal of the prescribed  approximation accuracy $\varepsilon > 0$. 
One key argument in most of these results is,  first, to employ a Monte Carlo approximation scheme which can approximate the solution of the PDE under consideration at a fixed space-time point without the curse of dimensionality and, thereafter, to prove then that  DNNs are flexible enough to mimic the behaviour of the employed approximation scheme. Having this in mind, one could aim for a general abstract result which shows under suitable assumptions that if a certain function can be approximated by any kind of (Monte Carlo) approximation scheme  without   the curse of dimensionality, then the function can also be approximated with DNNs without   the curse of dimensionality. 
It is a  key contribution of this article to make a first step towards this direction.
In particular,
the main result of this paper, roughly speaking, shows that if a function can be approximated by means of some suitable discrete approximation scheme without  the curse of dimensionality  and if there exist  DNNs which satisfy certain regularity properties and which approximate this discrete approximation scheme without  the curse of dimensionality,  then the function itself can  also be approximated with  DNNs without the curse of dimensionality.
Moreover, 
for the number of real parameters used to describe such approximating  DNNs we provide an explicit upper bound for the optimal exponent of the dimension $d \in \N$ of the function under consideration  as well as an explicit lower bound for the optimal exponent of the prescribed approximation accuracy $\varepsilon >0$.
As an application of this result we derive that
solutions of suitable Kolmogorov PDEs
can  be approximated with DNNs without the curse of dimensionality. 
\end{abstract}

\tableofcontents

\section{Introduction}

In the past few years deep artificial neural networks (DNNs) have been  successfully  employed in   a large number of computational problems including, e.g.,
 language processing (cf., e.g., \cite{dahl2012context,graves2013speech,HuConvolutional2014,Kalchbrenner14aconvolutional,hinton2012deep,wu2016stimulated}), 
 image recognition  (cf., e.g., \cite{huang2017densely,krizhevsky2012imagenet,simonyan2014very,Taigman2014,wangfacerecognition2015}), 
fraud detection
(cf., e.g., \cite{CHOUIEKH2018,Roy2018}), and computational advertisement
(cf., e.g., \cite{Wang2017Ad,Zhai2016}).
Recently, it has also been proposed in 
\cite{weinan2017deep,Han2018PNAS} to reformulate 
partial differential equations (PDEs) as stochastic learning problems and to employ DNNs together with 
stochastic gradient descent methods 
to approximate the solutions of such PDEs (cf., e.g., also \cite{uchiyama1993solving,MeadeFernandez1994,Lagaris1998ArtificialNN,LiLuo2003}). We refer, e.g., to \cite{Raissi2018DeepHP,PhamWarin2019,Magill2018NeuralNT,LyeMishraRay2019,LongLuMaDong2018,JacquierOumgari2019,HurePhamWarin2019,HanLong2018,GoudenegeMolent2019,FujiiTakahashi2019,BeckBeckerCheridito2019,Berg2018AUD,ChanMikaelWarin2019, Kolmogorov,BeckJentzenE2019,BeckerCheridito2019,BeckerCheriditoJentzen2019, weinan2018deep,Farahmand2017DeepRL,   henry2017deep,Sirignano2018,Dockhorn2019} for further developments and extensions of such deep learning based numerical approximation methods for PDEs. 
In particular, the references \cite{Magill2018NeuralNT,Kolmogorov,Berg2018AUD,weinan2018deep,JacquierOumgari2019} deal with linear PDEs (and the stochastic differential equations (SDEs) related to them, respectively),
the references \cite{HurePhamWarin2019,BeckBeckerCheridito2019,ChanMikaelWarin2019,Farahmand2017DeepRL,FujiiTakahashi2019,henry2017deep,Dockhorn2019} deal with  semilinear PDEs (and the backward stochastic differential equations (BSDEs) related to them, respectively),
  the references \cite{LongLuMaDong2018,BeckJentzenE2019,PhamWarin2019,Raissi2018DeepHP} deal  with  fully nonlinear PDEs (and the second-order backward stochastic differential equations (2BSDEs) related to them,
  respectively),
the references \cite{LyeMishraRay2019,Sirignano2018,HanLong2018} deal  with  certain specific subclasses of fully nonlinear PDEs (and the 2BSDEs related to them, 
respectively),
 and
the references \cite{BeckerCheridito2019,BeckerCheriditoJentzen2019,Sirignano2018,GoudenegeMolent2019} deal with free boundary PDEs (and the optimal stopping/option pricing problems  related to them (see, e.g.,~\cite[Chapter~1]{BensoussanLions1982}), respectively).
In the scientific literature there are also a few rigorous mathematical convergence results for such deep learning based numerical approximation methods for PDEs. For example, the references \cite{HanLong2018,Sirignano2018} provide mathematical convergence results for such deep learning based numerical approximation methods for PDEs without any information on the convergence speed and, for instance, the references
 \cite{BernerGrohsJentzen2018, ElbraechterSchwab2018,GrohsWurstemberger2018,HutzenthalerJentzenKruse2019,JentzenSalimovaWelti2018,KutyniokPeterseb2019,ReisingerZhang2019,GrohsHornungJentzen2019} provide mathematical convergence results of such deep learning based numerical approximation methods for PDEs with dimension-independent convergence rates and error constants which are only polynomially dependent on the dimension. In particular, the latter references show that  DNNs can approximate solutions of certain PDEs without the curse of dimensionality (cf.~\cite{Bellman1957}) in the sense that the number of real parameters employed to describe the DNN grows at most polynomially both in the PDE dimension $d \in \N$  and the reciprocal of the prescribed  approximation accuracy $\varepsilon > 0$ (cf., e.g., \cite[Chapter~1]{Novak2008} and \cite[Chapter~9]{Novak2010}). One key argument in most of these articles is,  first, to employ a Monte Carlo approximation scheme which can approximate the solution of the PDE under consideration at a fixed space-time point without the curse of dimensionality and, thereafter, to prove then that  DNNs are flexible enough to mimic the behaviour of the employed approximation scheme (cf., e.g., \cite[Section~2 and (i)--(iii) in Section~1]{JentzenSalimovaWelti2018} and \cite{GrohsWurstemberger2018}). Having this in mind, one could aim for a general abstract result which shows under suitable assumptions that if a certain function can be approximated by any kind of (Monte Carlo) approximation scheme  without   the curse of dimensionality, then the function can also be approximated with DNNs without   the curse of dimensionality. 
 
It is a  key contribution of this article to make a first step towards this direction.
In particular,
the main result of this paper, Theorem~\ref{thm:main} below, roughly speaking, shows that if a function can be approximated by means of some suitable discrete approximation scheme without  the curse of dimensionality (cf.~\eqref{eq:ass:u} in Theorem~\ref{thm:main} below) and if there exist  DNNs which satisfy certain regularity properties and which approximate this discrete approximation scheme without  the curse of dimensionality,  then the function itself can  also be approximated with  DNNs without the curse of dimensionality.
Moreover, 
for the number of real parameters used to describe such approximating  DNNs we provide in Theorem~\ref{thm:main} below an explicit upper bound for the optimal exponent of the dimension $d \in \N$ of the function under consideration  as well as an explicit lower bound for the optimal exponent of the prescribed approximation accuracy $\varepsilon >0$ (see \eqref{eq:prop:statement} in Theorem~\ref{thm:main} below).

In our applications of 
 Theorem~\ref{thm:main}   we employ  Theorem~\ref{thm:main} to study in 
 Theorem~\ref{thm:dnn:kolmogorov} below DNN approximations for PDEs.
Theorem~\ref{thm:dnn:kolmogorov} can be considered as a special case of Theorem~\ref{thm:main} with the function to be approximated to be equal to the solution of a suitable Kolmogorov PDE (cf.~\eqref{eq:PDE} below) at the final time $T \in (0, \infty)$ and the approximating scheme to be equal to the Monte Carlo Euler scheme.
In particular, Theorem~\ref{thm:dnn:kolmogorov} shows that 
 solutions of suitable Kolmogorov PDEs
can  be approximated with DNNs without the curse of dimensionality. 
For the number of real parameters used to describe such approximating  DNNs Theorem~\ref{thm:dnn:kolmogorov} also provides an explicit upper bound for the optimal exponent of the dimension $d \in \N$ of the PDE under consideration  as well as  an explicit lower bound for the optimal exponent of the prescribed approximation accuracy $\varepsilon >0$ (see \eqref{eq:kolmogorov:statement} below). 
In order to illustrate the findings of Theorem~\ref{thm:dnn:kolmogorov} below, we
now present in Theorem~\ref{thm:intro} below a special case of  Theorem~\ref{thm:dnn:kolmogorov}.

\begin{theorem}
	\label{thm:intro}
	Let
	$ \varphi_{0,d} \colon \R^d \to \R $, $ d \in \N $,
	and
	$ \varphi_{ 1, d } \colon \R^d \to \R^d $,
	$ d \in \N $,
	be functions,
	let  $\left\| \cdot \right\| \colon (\cup_{d \in \N} \R^d) \to [0, \infty)$ satisfy for all $d \in \N$, $x = (x_1, x_2, \ldots, x_d) \in \R^d$ that $\|x\| = ( \smallsum_{i=1}^d |x_i|^2)^{\nicefrac{1}{2}}$,
		let $A_d \in C(\R^d, \R^d)$, $d \in \N$, satisfy   for all $d \in \N$, $x = (x_1, x_2, \ldots, x_d) \in \R^d$ that
	$A_d(x) = (\max\{x_1, 0\}, \max\{x_2, 0\}, \ldots, \max\{x_d, 0\})$, let 	$\ANNs
	=
	\cup_{L \in \N}
	\cup_{ (l_0,l_1,\ldots, l_L) \in \N^{L+1} }
	(
	\times_{k = 1}^L (\R^{l_k \times l_{k-1}} \times \R^{l_k})
	)$, let $P \colon \ANNs \to \N$  and $R \colon \ANNs \to \cup_{k,l\in\N}\,C(\R^k,\R^l)$ 
	 satisfy
	for all $ L\in\N$, $l_0,l_1,\ldots, l_L \in \N$, 
	$
	\Phi  
	=
	((W_1, B_1),(W_2, B_2),\allowbreak \ldots, (W_L,\allowbreak B_L))
	\in  \allowbreak
	( \times_{k = 1}^L\allowbreak(\R^{l_k \times l_{k-1}} \times \R^{l_k}))
	$,
	$x_0 \in \R^{l_0}, x_1 \in \R^{l_1}, \ldots, x_{L-1} \in \R^{l_{L-1}}$ 
	with $\forall \, k \in \N \cap (0,L) \colon x_k =A_{l_k}(W_k x_{k-1} + B_k)$  
	that
	$P (\Phi)
	=
	\sum_{k = 1}^L l_k(l_{k-1} + 1) 
	$,
$R(\Phi) \in C(\R^{l_0},\R^{l_L})$, and
$(R(\Phi)) (x_0) = W_L x_{L-1} + B_L$,
	let
	$ T, \kappa \in (0, \infty)$, 
	$\ExponDim_1 \in [\nicefrac{1}{2}, \infty)$, $\ExponDim_3 \in [4, \infty)$,
	$\ExponError, \ExponDim_2, \ExponDim_4, \ExponDim_5, \ExponDim_6  \in [0, \infty)$, $\theta \in [1, \infty)$, 
		$
	( \phi^{ m, d }_{ \varepsilon } )_{ 
		(m, d, \varepsilon) \in \{ 0, 1 \} \times \N \times (0,1] 
	} 
	\subseteq \ANNs
	$, 
	assume for all
	$ d \in \N $, 
	$ \varepsilon \in (0,1] $, 
	$ m \in \{0, 1\}$,
	$ 
	x, y \in \R^d
	$
	that
	$
	R( \phi^{ 0, d }_{ \varepsilon } )
	\in 
	C( \R^d, \R )
	$,
	$
	R( \phi^{ 1, d }_{ \varepsilon } )
	\in
	C( \R^d, \R^d )
	$, 
		$ 
	P( \phi^{ m, d }_{ \varepsilon } ) 
	\leq \kappa d^{ 2^{(-m)} \ExponDim_3 } \varepsilon^{ - 2^{(-m)}  \ExponError }$,
	$ |( R (\phi^{ 0, d }_{ \varepsilon }) )(x) - ( R (\phi^{ 0, d }_{ \varepsilon }) )(y)| \leq \kappa d^{\ExponDim_6} (1   + \|x\|^{\theta} + \|y \|^{\theta})\|x-y\|$, 
	$
	\|
	( R (\phi^{ 1, d }_{ \varepsilon }) )(x)    
	\|	
	\leq 
	\kappa ( d^{ \ExponDim_1 + \ExponDim_2 } + \| x \| )
	$, $|
	\varphi_{ 0, d }( x )| \leq \kappa d^{ \ExponDim_6 }
	( d^{ \theta(\ExponDim_1 + \ExponDim_2) } + \| x \|^{ \theta } )$,
	$
	\| 
	\varphi_{ 1, d }( x ) 
	- 
	\varphi_{ 1, d }( y )
	\|
	\leq 
	\kappa 
	\| x - y \| 
	$, 	
	and
	\begin{equation}
	\label{eq:intro:hypo}
	\| 
	\varphi_{ m, d }(x) 
	- 
	( R (\phi^{ m, d }_{ \varepsilon }) )(x)
	\|
	\leq 
	\varepsilon  \kappa d^{\ExponDim_{(5 -m)}} (d^{\theta(\ExponDim_1 + \ExponDim_2)}+ \|x\|^{\theta}) 
	,
	\end{equation}
	and for every $ d \in \N $ let
	$ u_d \colon [0,T] \times \R^{d} \to \R $
	be an 
	at most polynomially growing viscosity solution of
	\begin{equation}
	\label{eq:PDE_intro}
	\begin{split}
	( \tfrac{ \partial }{\partial t} u_d )( t, x ) 
	& = 
	( \tfrac{ \partial }{\partial x} u_d )( t, x )
	\,
	\varphi_{ 1, d }( x )
	+
	\textstyle
	\sum\limits_{ i = 1 }^d
	\displaystyle
	( \tfrac{ \partial^2 }{ \partial x_i^2  } u_d )( t, x )
	\end{split}
	\end{equation}
	with $ u_d( 0, x ) = \varphi_{ 0, d }( x ) $
	for $ ( t, x ) \in (0,T) \times \R^d $.
	Then for every $p \in (0, \infty)$
	there exist
	$
	c \in \R
	$ and
	$
	( 
	\Psi_{ d, \varepsilon } 
	)_{ (d , \varepsilon)  \in \N \times (0,1] } \subseteq \ANNs
	$
	such that
	for all 
	$
	d \in \N 
	$,
	$
	\varepsilon \in (0,1] 
	$
	it holds that
	$
R( \Psi_{ d, \varepsilon } )
	\in C( \R^{ d }, \R )
	$, $[
	\int_{ [0, 1]^d }
	|
	u_d(T, x) - ( R (\Psi_{ d, \varepsilon }) )( x )
	|^p
	\,
	dx
	]^{ \nicefrac{ 1 }{ p } }
	\leq
	\varepsilon$, 
	and
	\begin{align}
	\label{eq:intro:par}
	P( \Psi_{ d, \varepsilon } ) \leq c \varepsilon^{-(\ExponError +6)}  d^{6[\ExponDim_6 + ( \ExponDim_1 + \ExponDim_2)(\theta+1)] +  \ExponDim_3  +  \ExponError \max\{\ExponDim_5 + \theta ( \ExponDim_1 + \ExponDim_2), \ExponDim_4 + \ExponDim_6 + 2\theta ( \ExponDim_1 + \ExponDim_2)\}}.
	\end{align}
\end{theorem}

Theorem~\ref{thm:intro} is an immediate consequence of Corollary~\ref{cor:laplacian:lebesgue}  in Section~\ref{sec:Kolmogorov} below. Corollary~\ref{cor:laplacian:lebesgue}, in turn, is a special case of Theorem~\ref{thm:dnn:kolmogorov}.
Let us add some comments regarding the mathematical objects appearing 
in Theorem~\ref{thm:intro}. 
The set $ \ANNs $ in  Theorem~\ref{thm:intro}  above 
is a set of tuples of pairs of real matrices and real vectors and this set
represents the set of all DNNs (see also Definition~\ref{Def:ANN} below). 
The functions $A_d \in C(\R^d, \R^d)$, $d \in \N$, in Theorem~\ref{thm:intro}  represent multidimensional rectifier functions. 
Theorem~\ref{thm:intro} 
is thus an approximation result for rectified DNNs. 
Moreover, 
for every DNN $ \Phi \in \ANNs $
in Theorem~\ref{thm:intro} above 
$ P( \Phi ) \in \N $ represents the number of real parameters 
which are used to describe the DNN $ \Phi $ (see also Definition~\ref{Def:ANN} below). 
In particular, for every DNN $ \Phi \in \ANNs $
in Theorem~\ref{thm:intro} one can think of 
$ P( \Phi ) \in \N $ 
as a number proportional to the amount of memory storage 
 needed to store the DNN $\Phi$.
Furthermore, the function 
$ R \colon \ANNs \to \cup_{ k, l \in \N } C( \R^k, \R^l ) $  
from the set $ \ANNs $ of ``all DNNs" to the union 
$ \cup_{ k, l \in \N } C( \R^k, \R^l ) $ 
of continuous functions  describes the realization functions
associated to the DNNs (see also Definition~\ref{Definition:ANNrealization} below). 
The real number $ T > 0 $ 
in Theorem~\ref{thm:intro} describes the time horizon under consideration and
the real numbers $ \kappa, \ExponError, \theta, \ExponDim_1, \ExponDim_2, \ldots, \ExponDim_6 \in \R$ in Theorem~\ref{thm:intro} 
are constants used to formulate the assumptions in Theorem~\ref{thm:intro}. 
The key assumption in Theorem~\ref{thm:intro} is the hypothesis 
that both the possibly nonlinear initial value functions $ \varphi_{ 0, d } \colon \R^d \to \R $,
$ d \in \N $, 
and the possibly nonlinear drift coefficient functions 
$ \varphi_{ 1, d } \colon \R^d \to \R^d $, $ d \in \N $,
of the PDEs in \eqref{eq:PDE_intro} can be approximated  by means of DNNs 
without the curse of dimensionality 
(see \eqref{eq:intro:hypo} above for details). Results related to Theorem~\ref{thm:dnn:kolmogorov} have been established in 
\cite[Theorem~3.14]{GrohsWurstemberger2018}, \cite[Theorem~1.1]{JentzenSalimovaWelti2018},    \cite[Theorem~4.1]{HutzenthalerJentzenKruse2019}, and
\cite[Corollary~2.2]{ReisingerZhang2019}. 
Theorem~3.14 in \cite{GrohsWurstemberger2018} 
proves a similar statement to \eqref{eq:intro:par}
for a different class of PDEs than \eqref{eq:PDE_intro}, that is, Theorem~3.14 in \cite{GrohsWurstemberger2018}  deals with Black-Scholes PDEs with affine linear coefficient functions while in \eqref{eq:PDE_intro} the diffusion coefficient is constant and the drift coefficient may be nonlinear. 
Theorem~1.1 in
\cite{JentzenSalimovaWelti2018}   shows the existence of  constants  and  exponents of   $d \in \N$ 
and  $\varepsilon >0$ 
such that \eqref{eq:intro:par} holds but does not provide any explicit form for these exponents.
Theorem~4.1 in \cite{HutzenthalerJentzenKruse2019}  studies a different class of PDEs than \eqref{eq:PDE_intro} (the diffusion coefficient is chosen so that the second order term is the Laplacian and the drift coefficient is chosen to be zero but there is a nonlinearity depending on the PDE solution in the PDE in Theorem~4.1 in \cite{HutzenthalerJentzenKruse2019}) and provides 
an explicit  exponent for $\varepsilon >0$
and
the existence of  constants  and exponents of $d \in \N$  such that \eqref{eq:intro:par} holds. 
Corollary~2.2 in \cite{ReisingerZhang2019} studies a more general class of Kolmogorov PDEs than \eqref{eq:PDE_intro}  and  shows the existence of constants  and exponents of  $d \in \N$ 
and  $\varepsilon >0$ 
such that \eqref{eq:intro:par} holds.
Theorem~\ref{thm:dnn:kolmogorov} above extends these results by providing explicit exponents for 
$d \in \N$ and $\varepsilon > 0$ in terms of the used assumptions such that \eqref{eq:intro:par} holds and, in addition, Theorem~\ref{thm:dnn:kolmogorov} can  be considered as a special case of the general  DNN approximation result in Theorem~\ref{thm:main} with the functions to be approximated to be equal to the solutions of  the PDEs in \eqref{eq:PDE_intro} at the final time $T \in (0, \infty)$ and the approximating scheme to be equal to the Monte Carlo Euler scheme.

The remainder of this article is organized as follows.  In Section~\ref{sec:main} we present Theorem~\ref{thm:main}, which is the main result of this paper.   The proof of Theorem~\ref{thm:main}  employs  the elementary result in Lemma~\ref{lem:apriori}. Lemma~\ref{lem:apriori} establishes suitable a priori bounds for random variables and follows from the well-known  discrete Gronwall-type inequality in
Lemma~\ref{lem:discrete:Gronwall}
below. In Section~\ref{sec:calculus} we develop in Lemma~\ref{lem:sum:ANN} and Lemma~\ref{lem:Composition_Sum}  a few elementary results on representation flexibilities of DNNs.
The proofs  of Lemma~\ref{lem:sum:ANN} and Lemma~\ref{lem:Composition_Sum} use  results on a certain artificial neural network (ANN) calculus which we recall and extend in Subsections~\ref{subsec:ANNs}--\ref{subsec:sums}.
In Section~\ref{sec:Kolmogorov} in Theorem~\ref{thm:dnn:kolmogorov} 
we employ Lemma~\ref{lem:sum:ANN} and Lemma~\ref{lem:Composition_Sum} to establish
 the existence of  DNNs which approximate  solutions of suitable Kolmogorov PDEs without the curse of dimensionality. 
In our proof of Theorem~\ref{thm:dnn:kolmogorov}  we also employ error estimates for the Monte Carlo Euler method
which we present  in Proposition~\ref{prop:monte_carlo_euler} in Section~\ref{sec:Kolmogorov}. The proof of Proposition~\ref{prop:monte_carlo_euler}, in turn, makes use of the elementary error estimate results  in Lemmas~\ref{lem:con_monte_carlo_euler}--\ref{lem:error:monte_carlo} below.

\section[Deep artificial neural network (DNN) approximations]{Deep artificial neural network (DNN) approximations}
\label{sec:main}

In this section we show in Theorem~\ref{thm:main} below that, 
roughly speaking, if a function can be approximated by means of some suitable discrete approximation scheme without  the curse of dimensionality  and if there exist  DNNs which satisfy certain regularity properties and which approximate this discrete approximation scheme without  the curse of dimensionality,  then the function itself can  also be approximated with  DNNs without the curse of dimensionality.
In our proof of Theorem~\ref{thm:main}  we employ the elementary a priori estimates for expectations of certain random variables in Lemma~\ref{lem:apriori} below. Lemma~\ref{lem:apriori}, in turn, follows from the well-known discrete Gronwall-type inequality in
Lemma~\ref{lem:discrete:Gronwall}
 below.

\subsection{A priori bounds for random variables}

\begin{lemma}
	\label{lem:discrete:Gronwall}
	Let $\alpha \in [0, \infty)$, $ \beta \in [0, \infty]$ and let $ x \colon \N_0 \to \R$  satisfy for all $n \in \N$ that $x_n \leq \alpha x_{n-1} + \beta$.
	Then it holds for all $n \in \N$  that
	\begin{align}
	\label{eq:gronwall:result}
	x_n \leq \alpha^n x_0 + \beta (1 + \alpha + \ldots + \alpha^{n-1}) \leq \alpha^n x_0 + \beta e^{\alpha}.
	\end{align}
\end{lemma}
\begin{proof}[Proof of Lemma~\ref{lem:discrete:Gronwall}]
	We prove \eqref{eq:gronwall:result} by induction on $n \in \N$. For the base case $n=1$ note that the hypothesis that $\forall \, k \in \N \colon x_k \leq \alpha x_{k-1} + \beta$ ensures that
	\begin{align}
	x_1 \leq \alpha x_0 +\beta = \alpha^1 x_0 + \beta \leq \alpha^1 x_0 + \beta e^\alpha.
	\end{align}
	This establishes \eqref{eq:gronwall:result} in the base case $n=1$. For the induction step $\N \ni (n-1)  \to n \in \N \cap [2, \infty)$ observe that the hypothesis that $\forall \, k \in \N \colon x_k \leq \alpha x_{k-1} + \beta$  implies that for all $n \in \N \cap [2, \infty)$ with $x_{n-1} \leq \alpha^{n-1} x_0 + \beta (1 + \alpha + \ldots + \alpha^{n-2})$ it holds that
	\begin{align}
	\begin{split}
	x_{n} &\leq \alpha x_{n-1} + \beta \leq \alpha^{n} x_0 + \alpha \beta (1 + \alpha + \ldots + \alpha^{n-2}) + \beta\\
	& = \alpha^{n} x_0 + \beta (1 + \alpha + \ldots + \alpha^{n-1})  \leq \alpha^{n} x_0 + \beta e^{\alpha}.
	\end{split}
	\end{align}
	Induction thus establishes \eqref{eq:gronwall:result}. This completes the proof of Lemma~\ref{lem:discrete:Gronwall}.
\end{proof}

\begin{lemma}
	\label{lem:apriori}
	Let $N \in \N$, $p \in [1, \infty)$, $\alpha, \beta, \gamma \in [0, \infty)$
	and let $X_n \colon \Omega \to \R$, $n \in \{0, 1, \ldots, N\}$, 	and  $Z_n \colon \Omega \to \R$, $n \in \{0, 1, \ldots, N-1\}$,  be random variables which satisfy for all $n \in \{1, 2, \ldots, N\}$ that
	\begin{align}
	\label{eq:apriori:ass}
	|X_n | \leq \alpha |X_{n-1}| + \beta  \big[\gamma + |Z_{n-1}| \big].
	\end{align}
	Then it holds that
	\begin{align}
	\begin{split}
	\left( \E\! \left[ |X_N |^p \right]\right)^{\nicefrac{1}{p}} \leq \alpha^N \! \left( \E\! \left[ |X_0 |^p \right]\right)^{\nicefrac{1}{p}} + e^{\alpha} \beta \! \left[ \gamma + \sup\nolimits_{i \in \{0, 1, \ldots, N -1 \}} \left( \E\! \left[ |Z_{i} |^p \right]\right)^{\nicefrac{1}{p}}\right].
	\end{split}
	\end{align}
\end{lemma}
\begin{proof}[Proof of Lemma~\ref{lem:apriori}]
	First, note that \eqref{eq:apriori:ass} implies for all $n \in \{1, 2, \ldots, N\}$  that
	\begin{align}
	\begin{split}
	\left( \E\! \left[ |X_n |^p \right]\right)^{\nicefrac{1}{p}}  &\leq \alpha \! \left( \E\! \left[ |X_{n-1} |^p \right]\right)^{\nicefrac{1}{p}} +\beta \! \left[ \gamma+  \left( \E\! \left[ |Z_{n-1} |^p \right]\right)^{\nicefrac{1}{p}} \right]\\
	& \leq  \alpha \! \left( \E\! \left[ |X_{n-1} |^p \right]\right)^{\nicefrac{1}{p}} +\beta \! \left[ \gamma+  \sup\nolimits_{i \in \{0, 1, \ldots, N -1 \}} \left( \E\! \left[ |Z_{i} |^p \right]\right)^{\nicefrac{1}{p}} \right].
	\end{split}
	\end{align}
	Lemma~\ref{lem:discrete:Gronwall} (with $\alpha = \alpha$, $\beta = \beta \, [ \gamma+  \sup\nolimits_{i \in \{0, 1, \ldots, N -1 \}} ( \E [ |Z_{i} |^p ] )^{\nicefrac{1}{p}} ]$ in the notation of Lemma~\ref{lem:discrete:Gronwall}) hence establishes for all $n \in \{1, 2, \ldots, N\}$ that
	\begin{align}
	\begin{split}
	\left( \E\! \left[ |X_n |^p \right]\right)^{\nicefrac{1}{p}} \leq \alpha^n \! \left( \E\! \left[ |X_0 |^p \right]\right)^{\nicefrac{1}{p}} + e^{\alpha} \beta \! \left[ \gamma + \sup\nolimits_{i \in \{0, 1, \ldots, N -1 \}} \left( \E\! \left[ |Z_{i} |^p \right]\right)^{\nicefrac{1}{p}}\right].
	\end{split}
	\end{align}
	The proof of Lemma~\ref{lem:apriori} is thus completed.
\end{proof}

\subsection{A DNN approximation result for Monte Carlo algorithms}

\begin{theorem}
	\label{thm:main}
		Let $(\Omega, \mathcal{F}, \P)$ be a probability space,
	let $  \ExponN_0 \in (0, \infty)$, $\ExponN_1, \ExponN_2, \ExponError, \allowbreak \ExponDim_0, \ExponDim_1, \allowbreak\ldots\allowbreak, \ExponDim_6 \in [0, \infty)$, $\Constant, p, \theta \in [1, \infty)$, 
 $(M_{N})_{N \in \N} \subseteq \N$,  
	let $Z^{N, d, m}_n \colon \Omega \to \R^{d} $, $n \in \{0, 1, \ldots, N-1\}$, $m \in \{1, 2, \ldots, M_{N}\}$, $d, N \in \N$, be random variables,
	let $f_{N, d}  \in C( \R^{d} \times \R^{d}, \R^{d})$,  $d, N \in \N$, and
	 $Y^{N, d, x}_n = (Y^{N, d, m, x}_n)_{m \in \{1, 2, \ldots, M_{N}\}} \colon \Omega \to \R^{M_N d}$, $n \in \{0, 1, \ldots, N\}$, $x \in \R^d$, $d, N \in \N$,   satisfy  for all $N,  d \in \N$,  $m \in \{1, 2, \ldots, M_{N}\}$,   $x \in \R^d$, $n \in \{1, 2, \ldots, N\}$, $\omega \in \Omega$ that 	$Y^{N, d, m, x}_{0}(\omega)  = x$ and
	\begin{align}
	\begin{split}
	Y^{N, d, m, x}_{n}(\omega) &= f_{N, d} \big(Z^{N, d, m}_{n-1}(\omega), Y^{N, d, m, x}_{n-1}(\omega)\big),
	\end{split}
	\end{align}
let $\left\| \cdot \right\| \colon (\cup_{d \in \N} \R^d) \to [0, \infty)$ satisfy for all $d \in \N$, $x = (x_1, x_2, \ldots, x_d) \in \R^d$ that $\|x\| = ( \smallsum_{i=1}^d |x_i|^2)^{\nicefrac{1}{2}}$,
for every $d \in \N$ let $ \nu_d \colon \mathcal{B}(\R^d) \to [0,1]$ be a probability measure on $\R^d$, 
	let $g_{N, d} \in C( \R^{Nd}, \R)$, $ d, N \in \N$, and $u_d \in C(\R^d, \R)$, $d \in \N$,   satisfy for all $ N, d \in \N$, $m \in \{1, 2, \ldots, M_N\}$, $n \in \{0, 1, \ldots, N-1\}$ that 
	\begin{gather}
	\label{eq:ass:u}
	\left(\E \! \left[ \int_{\R^d} \big|u_d(x) - g_{M_N,d} (Y^{N, d, x}_N) \big|^p \, \nu_d (dx) \right] \right)^{\!\nicefrac{1}{p}} \leq \Constant  d^{\ExponDim_0} N^{-\ExponN_0},\\
	\label{eq:ass:Z:apriori}
	\left( \E\! \left[ \|Z^{N, d, m}_{n} \|^{2 p \theta} \right]\right)^{\nicefrac{1}{(2 p \theta)}} \leq  \Constant d^{\ExponDim_1}, \quad \text{and} \quad
\left[	\int_{\R^d} \|x\|^{2p \theta} \, \nu_d (dx) \right]^{\nicefrac{1}{(2 p \theta)}}  \leq \Constant d^{\ExponDim_1 + \ExponDim_2}, 
	\end{gather}
		let $\ANNs$ be a  set,
	let 
	$
	\paramANN
	\colon \ANNs \to \N
	$,
	$\mathcal{D} \colon \ANNs \to \cup_{L=2}^\infty\, \N^{L}$,
	and
	$
	\mathcal{R} \colon 
	\ANNs
	\to 
	\cup_{ k, l \in \N } C( \R^k, \R^l )
	$
	be functions,   
	let $\SubsetANNs_{d, \varepsilon} \subseteq \ANNs$, 
	$ \varepsilon \in (0, 1]$, $d \in \N$, let
 $(\mathbf{f}^{N, d}_{\varepsilon, z})_{(N, d, \varepsilon, z) \in \N^2 \times (0, 1] \times \R^d } \subseteq \ANNs$,
	$(\mathbf{g}^{N, d}_{\varepsilon})_{(N, d, \varepsilon) \in \N^2 \times (0, 1] } \subseteq \ANNs$,
	 	$(\idRelu_{d})_{d \in \N} \subseteq \ANNs$,
	assume  for all  $N, d \in \N$, $\varepsilon \in (0, 1]$, $x, y, z \in \R^d$ that
	$ \SubsetANNs_{d, \varepsilon} \subseteq \{\Phi \in \ANNs \colon \mathcal{R}( \Phi) \in C(\R^d, \R^d) \}$,
	$\idRelu_d \in \SubsetANNs_{d, \varepsilon}$, 
		$(\mathcal{R} (\idRelu_d))(x) = x$, $\paramANN(\idRelu_d) \leq \Constant d^{\ExponDim_3} $, $  \mathcal{R}( \mathbf{f}^{N, d}_{\varepsilon, z}) \in C(\R^d, \R^d)$,
	 $(\R^d \ni \mathfrak{z} \mapsto  ( \mathcal{R} (\mathbf{f}^{N, d}_{\varepsilon, \mathfrak{z}}))(x) \in \R^d)$ is $\mathcal{B}(\R^d) \slash \mathcal{B}(\R^d)$-measurable, and
	\begin{gather}
	\label{eq:ass:dist:phi}
	\|f_{N, d}(z, x) - ( \mathcal{R} (\mathbf{f}^{N, d}_{\varepsilon, z}))(x)  \| \leq \varepsilon \Constant d^{\ExponDim_4} (d^{\theta(\ExponDim_1 + \ExponDim_2)}+ \|x\|^{\theta}), \\
\label{eq:ass:linear}
\| ( \mathcal{R} (\mathbf{f}^{N, d}_{\varepsilon, z}))(x)  \| \leq 
\big(1 + \tfrac{\Constant}{N}\big) \|x\| + \Constant d^{\ExponDim_2}( d^{\ExponDim_1} + \|z\|),\\
	\label{eq:ass:phi:linear}
	\|f_{N, d}(z, x) - f_{N, d}(z, y) \| \leq \Constant \|x - y \|,
	\end{gather}
	assume  for every  $N, d \in \N$, $\varepsilon \in (0, 1]$,  $\Phi \in \SubsetANNs_{d, \varepsilon}$ that there exist
	$(\phi_z)_{z \in \R^d} \subseteq \SubsetANNs_{d, \varepsilon}$ such that
	for all $x, z, \mathfrak{z} \in \R^d$ it holds that
	$ (\mathcal{R} (\phi_z)) (x) = ( \mathcal{R} (\mathbf{f}^{N, d}_{\varepsilon, z}))((\mathcal{R} (\Phi))(x))  $, 
	$\paramANN(\phi_z) \leq \paramANN(\Phi) + \Constant N^{\ExponN_1} d^{\ExponDim_3} \varepsilon^{-\ExponError}$,
	and
	$ \mathcal{D} (\phi_z) = \mathcal{D} (\phi_{\mathfrak{z}})$,
	assume for all $N, d \in \N$, $\varepsilon \in (0, 1]$,  $x = (x_i)_{i \in \{1, 2, \ldots, N\}} \in \R^{Nd}$, 
	$y = (y_i)_{i \in \{1, 2, \ldots, N\}} \in \R^{Nd}$ 
	that $  \mathcal{R} (\mathbf{g}^{N, d}_{\varepsilon}) \in C(\R^{Nd}, \R)$ and
	\begin{gather}
	\label{eq:perturbation:psi}
	|g_{N,d}(x) - ( \mathcal{R} (\mathbf{g}^{N, d}_{\varepsilon}) )(x) | \leq \varepsilon \Constant d^{\ExponDim_5} \left[ d^{\theta(\ExponDim_1 + \ExponDim_2)}+ \tfrac{1}{N} \textstyle \sum\limits_{i=1}^N \displaystyle \|x_i \|^{\theta} \right],\\
	\label{eq:increments:psi}
 |( \mathcal{R} (\mathbf{g}^{N, d}_{\varepsilon}) )(x) - ( \mathcal{R} (\mathbf{g}^{N, d}_{\varepsilon}) )(y)| \leq  \frac{\Constant d^{\ExponDim_6}}{N} \left[ \textstyle \sum\limits_{i=1}^N \displaystyle (d^{\theta(\ExponDim_1 + \ExponDim_2)} + \|x_i\|^{\theta} + \|y_i \|^{\theta})\|x_i- y_i\| \right],
 	\end{gather}
 	and assume for every  $N, d \in \N$, $\varepsilon \in (0, 1]$,  $\Phi_1, \Phi_2, \ldots, \Phi_{M_N} \in \SubsetANNs_{d, \varepsilon}$ with $\mathcal{D}(\Phi_1) = \mathcal{D}(\Phi_2) = \ldots = \mathcal{D}(\Phi_{M_N})$ that there exists 
 	$\varphi \in \ANNs$ such that for all $x \in \R^d$ it holds that
 	 $  \mathcal{R} (\varphi) \in C(\R^d, \R)$,
 	$( \mathcal{R} (\varphi))(x) =  ( \mathcal{R} (\mathbf{g}^{ M_N, d }_{ \varepsilon }) )( (\mathcal{R} (\Phi_1))(x), (\mathcal{R} (\Phi_2))(x),$ $\ldots, (\mathcal{R} (\Phi_{M_N}))(x))$, and $\paramANN(\varphi) \leq \Constant N^{\ExponN_2} ( N^{\ExponN_1 +1} d^{\ExponDim_3} \varepsilon^{-\ExponError} + \paramANN(\Phi_1))$.
	Then 
	there exist
		$
	c \in \R
	$ and
	$
	( 
	\Psi_{ d, \varepsilon } 
	)_{ (d , \varepsilon)  \in \N \times (0,1] } \subseteq \ANNs
	$
	such that
	for all 
	$
	d \in \N 
	$,
	$
	\varepsilon \in (0,1] 
	$
	it holds that
	$
	\mathcal{R}( \Psi_{ d, \varepsilon } )
	\in C( \R^{ d }, \R )
	$, $[
	\int_{ \R^d }
	|
	u_d(x) - ( \mathcal{R} (\Psi_{ d, \varepsilon }) )( x )
	|^p
	\,
	\nu_d(dx)
	]^{ \nicefrac{ 1 }{ p } }
	\leq
	\varepsilon$, 
	and
	\begin{align}
	\label{eq:prop:statement}
\paramANN( \Psi_{ d, \varepsilon } ) \leq c d^{\frac{\ExponDim_0( \ExponN_1+\ExponN_2 +1)}{\ExponN_0} +  \ExponDim_3  +  \ExponError \max\{\ExponDim_5 + \theta (\ExponDim_1 + \ExponDim_2), \ExponDim_4 + \ExponDim_6 + 2\theta (\ExponDim_1 + \ExponDim_2)\}}  \varepsilon^{-\frac{( \ExponN_1+ \ExponN_2 +1)}{\ExponN_0} -\ExponError}.
\end{align}
\end{theorem}
\begin{proof}[Proof of Theorem~\ref{thm:main}]
	Throughout this proof let $\gamma = 46  e^{\Constant}   \Constant^2  ( 4 e^{\Constant+1}  \Constant^3   )^{ 2\theta}$, let $\delta = \max\{\ExponDim_5 + \theta (\ExponDim_1 + \ExponDim_2), \ExponDim_4 + \ExponDim_6 + 2\theta (\ExponDim_1 + \ExponDim_2)\}$, let $X^{N, d, x, \varepsilon}_n = (X^{N, d, m, x, \varepsilon}_n)_{m \in \{1, 2, \ldots, M_{N}\}} \colon \Omega \to \R^{M_N d}$, $n \in \{0, 1, \ldots, N\}$, $\varepsilon \in (0, 1]$, $x \in \R^d$, $d, N \in \N$, be the random variables  which satisfy  for all $N,  d \in \N$,  $m \in \{1, 2, \ldots, M_{N}\}$,   $x \in \R^d$,  $\varepsilon \in (0, 1]$, $n \in \{1, 2, \ldots, N\}$, $\omega 
	\in \Omega$ that 	$X^{N, d, m, x, \varepsilon}_{0}(\omega)  = x$ and
	\begin{align}
	\label{eq:iteration:app}
	\begin{split}
	X^{N, d, m, x, \varepsilon}_{n}(\omega) &= \Big( \mathcal{R} \Big(\mathbf{f}^{N, d}_{\varepsilon, Z^{N, d, m}_{n-1}(\omega)}\Big) \Big)  \big( X^{N, d, m, x, \varepsilon}_{n-1}(\omega)\big),
	\end{split}
	\end{align}
and	let $(\mathcal{N}_{d, \varepsilon})_{(d, \varepsilon) \in \N \times (0, 1]} \subseteq \N$ and $(\mathcal{E}_{d, \varepsilon})_{(d, \varepsilon) \in \N \times (0, 1]} \subseteq (0, 1]$ satisfy for all $\varepsilon \in (0, 1]$, $d \in \N$ that
	\begin{align}
	\label{eq:N:varepsilon}
	\mathcal{N}_{d, \varepsilon} = \min\! \left( \N \cap \big[\big(\tfrac{2\Constant d^{\ExponDim_0}}{\varepsilon}\big)^{\nicefrac{1}{\ExponN_0}}, \infty\big) \right) \qandq \mathcal{E}_{d, \varepsilon} = \tfrac{\varepsilon}{\gamma d^{  \delta}} .
	\end{align}
Note that for all  $N,  d \in \N$,  $\varepsilon \in (0, 1]$,  $n \in \{0, 1, 2, \ldots, N\}$ it holds that
\begin{align}
\big(\R^d \ni x \mapsto X^{N, d, x, \varepsilon}_n \in \R^{M_N d}\big) \in C(\R^d, \R^{M_N d}).
\end{align}
This implies that for all $N, d \in \N$, $\varepsilon \in (0, 1]$ it holds that 
	\begin{align}
	\label{eq:Lp:triangle}
	\begin{split}
	&\left(\E \! \left[ \int_{\R^d} \big|u_d(x) - ( \mathcal{R} (\mathbf{g}^{M_N, d}_{\varepsilon}) )(X^{N, d, x, \varepsilon}_N) \big|^p \, \nu_d (dx) \right] \right)^{\!\nicefrac{1}{p}} \\
	& \leq 	\left(\E \! \left[ \int_{\R^d} \big|u_d(x) - g_{M_N,d} (Y^{N, d, x}_N) \big|^p \, \nu_d (dx) \right] \right)^{\!\nicefrac{1}{p}} \\
	& + 	\left(\E \! \left[ \int_{\R^d} \big| g_{M_N,d} (Y^{N, d, x}_N) - (\mathcal{R} (\mathbf{g}^{M_N, d}_{\varepsilon} ))(Y^{N, d, x}_N) \big|^p \, \nu_d (dx) \right] \right)^{\!\nicefrac{1}{p}}\\
	& + 	\left(\E \! \left[ \int_{\R^d} \big|  (\mathcal{R} (\mathbf{g}^{M_N, d}_{\varepsilon}) )(Y^{N, d, x}_N) - ( \mathcal{R} (\mathbf{g}^{M_N, d}_{\varepsilon}) )(X^{N, d, x, \varepsilon}_N) \big|^p \, \nu_d (dx) \right] \right)^{\!\nicefrac{1}{p}}.
	\end{split}
	\end{align}
	Next observe that \eqref{eq:perturbation:psi} ensures for all $N, d \in \N$, $\varepsilon \in (0, 1]$ that 
	\begin{align}
	\label{eq:psi:Lp}
	\begin{split}
	&\left(\E \! \left[ \int_{\R^d} \big| g_{M_N,d} (Y^{N, d, x}_N) - (\mathcal{R} (\mathbf{g}^{M_N, d}_{\varepsilon}) )(Y^{N, d, x}_N) \big|^p \, \nu_d (dx) \right] \right)^{\!\nicefrac{1}{p}} \\
	&\leq  \varepsilon \Constant d^{\ExponDim_5} \left(\E \! \left[ \int_{\R^d} \Big| d^{\theta(\ExponDim_1 + \ExponDim_2)} + \tfrac{1}{M_N} \textstyle \sum\limits_{m=1}^{M_N} \displaystyle \| Y^{N, d, m, x}_N \|^{\theta} \Big|^p \, \nu_d (dx) \right] \right)^{\!\nicefrac{1}{p}}\\
	& \leq  \varepsilon \Constant d^{\ExponDim_5} \left[ d^{\theta(\ExponDim_1 + \ExponDim_2)}+ \tfrac{1}{M_N} \textstyle \sum\limits_{m=1}^{M_N} \displaystyle  \left(\E \! \left[ \int_{\R^d}  \| Y^{N, d, m, x}_N \|^{p\theta} \, \nu_d (dx) \right] \right)^{\!\nicefrac{1}{p}} \right].
	\end{split}
	\end{align}
In addition, note that 	\eqref{eq:ass:dist:phi} and \eqref{eq:ass:linear} assure that for all $N, d \in \N$, $\varepsilon \in (0, 1]$,  $x, z \in \R^d$ it holds that 
\begin{align}
\begin{split}
\|f_{N, d}(z, x) \| & \leq \|f_{N, d}(z, x) - ( \mathcal{R} (\mathbf{f}^{N, d}_{\varepsilon, z}))(x)  \| + \|( \mathcal{R} (\mathbf{f}^{N, d}_{\varepsilon, z}))(x)  \|\\
& \leq \varepsilon \Constant d^{\ExponDim_4} (d^{\theta(\ExponDim_1 + \ExponDim_2)}+ \|x\|^{\theta}) + \big(1 + \tfrac{\Constant}{N}\big) \|x\| + \Constant d^{\ExponDim_2}( d^{\ExponDim_1} + \|z\|).
\end{split}
\end{align}
This proves that for all $N, d \in \N$, $x, z \in \R^d$ it holds that  
\begin{align}
\|f_{N, d}(z, x) \|  \leq \big(1 + \tfrac{\Constant}{N}\big) \|x\| + \Constant d^{\ExponDim_2} ( d^{\ExponDim_1} + \|z\|).
\end{align}
Hence, we obtain that for all $N, d \in \N$, $m \in \{1, 2, \ldots, M_N\}$, $x \in \R^d$,  $ n \in \{1, 2, \ldots, N\}$ it holds that
\begin{align}
\label{eq:Y:bound}
\begin{split}
\| Y^{N, d, m, x}_n \| & = \big\| f_{N, d} \big(Z^{N, d, m}_{n-1}, Y^{N, d, m, x}_{n-1}\big)  \big\| \\
& \leq \big(1 + \tfrac{\Constant}{N}\big) \|Y^{N, d, m, x}_{n-1} \| + \Constant d^{\ExponDim_2} \big[ d^{\ExponDim_1} + \|Z^{N, d, m}_{n-1}\|\big].
\end{split}
\end{align}
Moreover, note that \eqref{eq:ass:linear} assures that for all $N, d \in \N$, $m \in \{1, 2, \ldots, M_N\}$, $x \in \R^d$, $\varepsilon \in (0, 1]$,  $ n \in \{1, 2, \ldots, N\}$ it holds that 
\begin{align}
\begin{split}
\| X^{N, d, m, x, \varepsilon}_{n}  \| & = \Big\| \Big( \mathcal{R} \Big(\mathbf{f}^{N, d}_{\varepsilon, Z^{N, d, m}_{n-1}} \Big) \Big)  \big( X^{N, d, m, x, \varepsilon}_{n-1}\big)  \Big\| \\
& \leq \big(1 + \tfrac{\Constant}{N}\big)\|X^{N, d, m, x, \varepsilon}_{n-1} \| + \Constant d^{\ExponDim_2} \big[ d^{\ExponDim_1} + \|Z^{N, d, m}_{n-1}\|\big].
\end{split}
\end{align}
Lemma~\ref{lem:apriori} (with $N = n$,  $p = 2p\theta$, $\alpha$ = $(1 + \frac{\Constant}{N})$, $\beta = \Constant d^{\ExponDim_2}$, 
$\gamma = d^{\ExponDim_1}$,
 $Z_i = \| Z^{N, d, m}_{i} \|$ for $N, d \in \N$, $ n \in \{1, 2, \ldots, N\}$, $m \in \{1, 2, \ldots, M_N\}$, $i \in \{0, 1, \ldots, n-1\}$  in the notation of Lemma~\ref{lem:apriori}),
 \eqref{eq:Y:bound}, and \eqref{eq:ass:Z:apriori} therefore demonstrate that for all $N, d \in \N$, $m \in \{1, 2, \ldots, M_N\}$, $x \in \R^d$,  $\varepsilon \in (0, 1]$,  $ n \in \{1, 2, \ldots, N\}$  it holds that 
\begin{align}
\begin{split}
& \max\left\{ \left(	\E \big[ \| Y^{N, d, m, x}_n \|^{2 p \theta} \big] \right)^{\nicefrac{1}{(2 p \theta)}},  \left(	\E \big[ \|X^{N, d, m, x, \varepsilon}_{n} \|^{2 p \theta} \big] \right)^{\nicefrac{1}{(2 p \theta)}} \right\} \\
&\leq \big(1 + \tfrac{\Constant}{N}\big)^n \|x\|
+ e^{(1 + \frac{\Constant}{N})} \Constant d^{\ExponDim_2} \left[ d^{\ExponDim_1}+ \sup\nolimits_{i \in \{0, 1, \ldots, n -1 \}} \left( \E\! \left[ \|Z^{N, d, m}_{i} \|^{2 p \theta} \right]\right)^{\!\nicefrac{1}{(2 p \theta)}} \right]\\
& \leq e^{\Constant} \|x\| + e^{\Constant+1} \Constant d^{\ExponDim_2} \big[ d^{\ExponDim_1} + \Constant d^{\ExponDim_1} \big] \leq e^{\Constant} \|x\| +  2 e^{\Constant+1}  \Constant^2 d^{\ExponDim_1 + \ExponDim_2}\\
& \leq 2 e^{\Constant+1}  \Constant^2 \big[\|x\| + d^{\ExponDim_1 + \ExponDim_2} \big].
\end{split}
\end{align}
This and the fact that $\forall \, a, b \in \R \colon |a+b|^{\theta} \leq 2^{\theta-1}(|a|^{\theta}+|b|^{\theta})$ prove  for all $N, d \in \N$, $m \in \{1, 2, \ldots, M_N\}$, $x \in \R^d$,  $\varepsilon \in (0, 1]$,  $ n \in \{1, 2, \ldots, N\}$  that 
\begin{align}
\label{eq:lp_norms:space}
\begin{split}
	  &\max\left\{  \E \big[ \| Y^{N, d, m, x}_n \|^{2 p \theta} \big], \E \big[ \|X^{N, d, m, x, \varepsilon}_{n} \|^{2 p \theta} \big] \right\} \\
	  &\leq \left(2 e^{\Constant+1}  \Constant^2 \big[\|x\| + d^{\ExponDim_1 + \ExponDim_2} \big] \right)^{2 p \theta} = \left( 2 e^{\Constant+1}  \Constant^2 \right)^{2 p \theta} \big[\|x\| + d^{\ExponDim_1 + \ExponDim_2} \big]^{2 p \theta}\\ &\leq 2^{2 p (\theta-1) }  \left( 2 e^{\Constant+1}  \Constant^2  \right)^{2 p \theta}\big[\|x\|^{\theta} + d^{\theta(\ExponDim_1 + \ExponDim_2)} \big]^{2p}\\
	& \leq  ( 4 e^{\Constant+1}  \Constant^2   )^{2 p \theta} \big[\|x\|^{\theta} + d^{\theta(\ExponDim_1 + \ExponDim_2)} \big]^{2p}.
\end{split}
\end{align}
This and \eqref{eq:ass:Z:apriori} establish that for all $N, d \in \N$, $m \in \{1, 2, \ldots, M_N\}$, $\varepsilon \in (0, 1]$  it holds that 	
\begin{align*}
\label{eq:lp_norms:integral}
& \max\left\{ \left(\E \! \left[ \int_{\R^d}  \| Y^{N, d, m, x}_N \|^{2p\theta} \, \nu_d (dx) \right] \right)^{\!\nicefrac{1}{(2p)}}, \left(\E \! \left[ \int_{\R^d}  \| X^{N, d, m, x, \varepsilon}_{N}  \|^{2p\theta} \, \nu_d (dx) \right] \right)^{\!\nicefrac{1}{(2p)}} \right\}\\
& \leq  ( 4 e^{\Constant+1}  \Constant^2   )^{ \theta} \left( \int_{\R^d}   \big[\|x\|^{\theta} + d^{\theta(\ExponDim_1 + \ExponDim_2)} \big]^{2p}\, \nu_d (dx)  \right)^{\!\nicefrac{1}{(2p)}}\\
& \leq  ( 4 e^{\Constant+1}  \Constant^2   )^{ \theta} \left[ \left( \int_{\R^d}    \|x\|^{ 2p \theta} \, \nu_d (dx)  \right)^{\!\nicefrac{1}{(2p)}} + d^{\theta(\ExponDim_1 + \ExponDim_2)} \right]\\
& \leq  ( 4 e^{\Constant+1}  \Constant^2   )^{ \theta}   \big[ \Constant^{\theta} d^{\theta (\ExponDim_1 + \ExponDim_2)} + d^{\theta(\ExponDim_1 + \ExponDim_2)}\big] \leq   2  ( 4 e^{\Constant+1}  \Constant^3   )^{ \theta}  d^{\theta (\ExponDim_1 + \ExponDim_2)}  . \numberthis
\end{align*}
Hence, we obtain that for all $N, d \in \N$, $m \in \{1, 2, \ldots, M_N\}$ it holds that 	
\begin{align}
\begin{split}
\left(\E \! \left[ \int_{\R^d}  \| Y^{N, d, m, x}_N \|^{p\theta} \, \nu_d (dx) \right] \right)^{\!\nicefrac{1}{p}}  &\leq \left(\E \! \left[ \int_{\R^d}  \| Y^{N, d, m, x}_N \|^{2p\theta} \, \nu_d (dx) \right] \right)^{\!\nicefrac{1}{(2p)}}\\
& \leq 2  ( 4 e^{\Constant+1}  \Constant^3   )^{ \theta}  d^{\theta(\ExponDim_1 + \ExponDim_2)}.
\end{split}
\end{align}
Combining this and \eqref{eq:psi:Lp} demonstrates that for all $N, d \in \N$, $\varepsilon \in (0, 1]$ it holds
that 
\begin{align}
\label{eq:bound:2nd_term}
\begin{split}
& \left(\E \! \left[ \int_{\R^d} \big| g_{M_N,d} (Y^{N, d, x}_N) - (\mathcal{R} (\mathbf{g}^{M_N, d}_{\varepsilon}) )(Y^{N, d, x}_N) \big|^p \, \nu_d (dx) \right] \right)^{\!\nicefrac{1}{p}} \\
& \leq \varepsilon \Constant d^{\ExponDim_5} \big[ d^{\theta(\ExponDim_1 + \ExponDim_2)} +2  ( 4 e^{\Constant+1}  \Constant^3   )^{ \theta}  d^{\theta(\ExponDim_1 + \ExponDim_2)} \big] \leq 3 \varepsilon \Constant  ( 4 e^{\Constant+1}  \Constant^3   )^{ \theta} d^{\ExponDim_5 + \theta (\ExponDim_1 + \ExponDim_2)}.
\end{split}
\end{align}
In addition, observe that \eqref{eq:increments:psi} ensures that for all $N, d \in \N$, $x \in \R^d$, $\varepsilon \in (0, 1]$ it holds that  
\begin{align*}
&  \big|  (\mathcal{R} (\mathbf{g}^{M_N, d}_{\varepsilon}) )(Y^{N, d, x}_N) - ( \mathcal{R} (\mathbf{g}^{M_N, d}_{\varepsilon}) )(X^{N, d, x, \varepsilon}_N) \big| \numberthis \\
& \leq  \frac{\Constant d^{\ExponDim_6}}{M_N} \bigg[ \textstyle \sum\limits_{m=1}^{M_N} \displaystyle \big(d^{\theta(\ExponDim_1 + \ExponDim_2)} + \|Y^{N, d, m, x}_N\|^{\theta} + \|X^{N, d, m, x, \varepsilon}_N \|^{\theta}\big)  \| Y^{N, d, m, x}_N - X^{N, d, m, x, \varepsilon}_N\| \bigg].
\end{align*}
This ensures  for all $N, d \in \N$, $\varepsilon \in (0, 1]$  that
\begin{align}
\begin{split}
& \left(\E \! \left[ \int_{\R^d} \big|  (\mathcal{R} (\mathbf{g}^{M_N, d}_{\varepsilon}) )(Y^{N, d, x}_N) - ( \mathcal{R} (\mathbf{g}^{M_N, d}_{\varepsilon}) )(X^{N, d, x, \varepsilon}_N) \big|^p \, \nu_d (dx) \right] \right)^{\!\nicefrac{1}{p}}\\
& \leq \frac{\Constant d^{\ExponDim_6}}{ M_N }  \sum_{m=1}^{M_N} \bigg(\E \bigg[ \int_{\R^d} \big(d^{\theta(\ExponDim_1 + \ExponDim_2)} + \|Y^{N, d, m, x}_N\|^{\theta} + \|X^{N, d, m, x, \varepsilon}_N \|^{\theta}\big)^p \\
& \qquad  \qquad \qquad \qquad \qquad \cdot \| Y^{N, d, m, x}_N -X^{N, d, m, x,  \varepsilon}_N\|^p  \, \nu_d (dx) \bigg] \bigg)^{\!\nicefrac{1}{p}}.
\end{split}
\end{align}
H\"older's inequality hence assures  for all $N, d \in \N$, $\varepsilon \in (0, 1]$  that
\begin{align*}
\label{eq:bound:3rd_term}
& \left(\E \! \left[ \int_{\R^d} \big|  (\mathcal{R} (\mathbf{g}^{M_N, d}_{\varepsilon}) )(Y^{N, d, x}_N) - ( \mathcal{R}( \mathbf{g}^{M_N, d}_{\varepsilon}) )(X^{N, d, x, \varepsilon}_N) \big|^p \, \nu_d (dx) \right] \right)^{\!\nicefrac{1}{p}}\\
& \leq \frac{\Constant d^{\ExponDim_6}}{ M_N }  \sum_{m=1}^{M_N} \bigg(\E \bigg[ \int_{\R^d} \big(d^{\theta(\ExponDim_1 + \ExponDim_2)} + \|Y^{N, d, m, x}_N\|^{\theta} + \|X^{N, d, m, x, \varepsilon}_N \|^{\theta}\big)^{2p} \, \nu_d (dx) \bigg] \bigg)^{\!\nicefrac{1}{(2p)}} \\
& \qquad  \qquad \qquad \cdot \bigg(\E \bigg[ \int_{\R^d} \| Y^{N, d, m, x}_N - X^{N, d, m, x, \varepsilon}_N\|^{2p}  \, \nu_d (dx) \bigg] \bigg)^{\!\nicefrac{1}{(2p)}}. \numberthis
\end{align*}
Moreover, note that \eqref{eq:lp_norms:integral} implies that for all $N, d \in \N$, $m \in \{1, 2, \ldots, M_N\}$, $\varepsilon \in (0, 1]$  it holds that
\begin{align}
\label{eq:bound:3rd_term:1}
\begin{split}
&\left(\E \bigg[ \int_{\R^d} \big(d^{\theta(\ExponDim_1 + \ExponDim_2)} + \|Y^{N, d, m, x}_N\|^{\theta} + \|X^{N, d, m, x, \varepsilon}_N \|^{\theta}\big)^{2p} \, \nu_d (dx) \bigg] \right)^{\!\nicefrac{1}{(2p)}}\\
& \leq d^{\theta(\ExponDim_1 + \ExponDim_2)}+ \bigg(\E \bigg[ \int_{\R^d} \|Y^{N, d, m, x}_N\|^{2p\theta}  \, \nu_d (dx) \bigg] \bigg)^{\!\nicefrac{1}{(2p)}} \\
& + \bigg(\E \bigg[ \int_{\R^d}  \|X^{N, d, m, x, \varepsilon}_N \|^{2p \theta} \, \nu_d (dx) \bigg] \bigg)^{\!\nicefrac{1}{(2p)}}\\
&\leq  d^{\theta(\ExponDim_1 + \ExponDim_2)} + 4  ( 4 e^{\Constant+1}  \Constant^3   )^{ \theta}  d^{\theta (\ExponDim_1 + \ExponDim_2)}  \leq 5  ( 4 e^{\Constant+1}  \Constant^3   )^{ \theta}  d^{\theta (\ExponDim_1 + \ExponDim_2)}. 
\end{split}
\end{align}
Next observe that \eqref{eq:ass:phi:linear} and \eqref{eq:ass:dist:phi} prove that  for all $N, d \in \N$, $m \in \{1, 2, \ldots, M_N\}$, $x \in \R^d$, $\varepsilon \in (0, 1]$, $n \in \{1, 2, \ldots, N\}$ it holds that
\begin{align}
\begin{split}
&\| Y^{N, d, m, x}_n - X^{N, d, m, x, \varepsilon}_n\| \\
&= \left\| f_{N, d} \big(Z^{N, d, m}_{n-1}, Y^{N, d, m, x}_{n-1}\big) -  \Big( \mathcal{R} \Big(\mathbf{f}^{N, d}_{\varepsilon, Z^{N, d, m}_{n-1}} \Big) \Big)  \big( X^{N, d, m, x, \varepsilon}_{n-1}\big) \right\|\\
& \leq \left\| f_{N, d} \big(Z^{N, d, m}_{n-1}, Y^{N, d, m, x}_{n-1}\big) - f_{N, d} \big(Z^{N, d, m}_{n-1},  X^{N, d, m, x, \varepsilon}_{n-1} \big) \right\|\\
& + \left\| f_{N, d} \big(Z^{N, d, m}_{n-1},  X^{N, d, m, x, \varepsilon}_{n-1} \big) -  \Big( \mathcal{R} \Big(\mathbf{f}^{N, d}_{\varepsilon, Z^{N, d, m}_{n-1}} \Big) \Big)  \big( X^{N, d, m, x, \varepsilon}_{n-1}\big) \right\|\\
& \leq \Constant  \| Y^{N, d, m, x}_{n-1} - X^{N, d, m, x, \varepsilon}_{n-1}\| + \varepsilon \Constant d^{\ExponDim_4} \left( d^{\theta(\ExponDim_1 + \ExponDim_2)} + \| X^{N, d, m, x, \varepsilon}_{n-1} \|^{\theta} \right).
\end{split}
\end{align}
Lemma~\ref{lem:apriori} (with $N = N$, $p = 2p$, $\alpha$ = $\Constant$, $\beta =  \varepsilon \Constant d^{\ExponDim_4}$, $\gamma = d^{\theta(\ExponDim_1 + \ExponDim_2)}$, $Z_n = \|X^{N, d, m, x, \varepsilon}_{n} \|^{\theta} $, $X_n =  \| Y^{N, d, m, x}_n - X^{N, d, m, x, \varepsilon}_n \|$ for $N, d \in \N$, $m \in \{1, 2, \ldots, M_N\}$, $x \in \R^d$, $\varepsilon \in (0, 1]$, $n \in \{1, 2, \ldots, N\}$  in the notation of Lemma~\ref{lem:apriori}) and \eqref{eq:lp_norms:space} hence ensure for all  $N, d \in \N$, $m \in \{1, 2, \ldots, M_N\}$, $x \in \R^d$, $\varepsilon \in (0, 1]$ that 
\begin{align}
\begin{split}
& \left(\E \big[ \| Y^{N, d, m, x}_N - X^{N, d, m, x, \varepsilon}_N\|^{2p} \big]\right)^{\!\nicefrac{1}{(2p)}}\\
& \leq e^{\Constant}  \varepsilon \Constant d^{\ExponDim_4} \left[ d^{\theta(\ExponDim_1 + \ExponDim_2)} + \sup\nolimits_{i \in \{0, 1, \ldots, N-1\}} \left(\E \big[ \|  X^{N, d, m, x, \varepsilon}_{i}\|^{2p \theta} \big]\right)^{\!\nicefrac{1}{(2p)}} \right]\\
& \leq \varepsilon e^{\Constant}   \Constant d^{\ExponDim_4} \left[ d^{\theta(\ExponDim_1 + \ExponDim_2)} + \left(( 4 e^{\Constant+1}  \Constant^2   )^{2 p \theta} \big[\|x\|^{\theta} + d^{\theta(\ExponDim_1 + \ExponDim_2)} \big]^{2p} \right)^{\!\nicefrac{1}{(2p)}} \right]\\
& = \varepsilon  e^{\Constant}   \Constant d^{\ExponDim_4}  \left[ d^{\theta(\ExponDim_1 + \ExponDim_2)} + ( 4 e^{\Constant+1}  \Constant^2   )^{ \theta} \big[\|x\|^{\theta} + d^{\theta(\ExponDim_1 + \ExponDim_2)} \big]   \right]\\
& \leq 2 \varepsilon  e^{\Constant}   \Constant d^{\ExponDim_4} ( 4 e^{\Constant+1}  \Constant^2   )^{ \theta} \big[\|x\|^{\theta} + d^{\theta(\ExponDim_1 + \ExponDim_2)} \big].
\end{split}
\end{align}
This and \eqref{eq:ass:Z:apriori} demonstrate that for all $N, d \in \N$, $m \in \{1, 2, \ldots, M_N\}$, $\varepsilon \in (0, 1]$ it holds that
\begin{align}
\begin{split}
&  \bigg(\E \bigg[ \int_{\R^d} \| Y^{N, d, m, x}_N - X^{N, d, m, x, \varepsilon}_N\|^{2p}  \, \nu_d (dx) \bigg] \bigg)^{\!\nicefrac{1}{(2p)}}\\
& \leq  2 \varepsilon  e^{\Constant}   \Constant d^{\ExponDim_4} ( 4 e^{\Constant+1}  \Constant^2   )^{ \theta} \left[ \int_{\R^d} \big( \|x\|^{\theta} +d^{\theta(\ExponDim_1 + \ExponDim_2)} \big)^{2p}\, \nu_d (dx) \right]^{\nicefrac{1}{(2p)}}\\
& \leq 2 \varepsilon  e^{\Constant}   \Constant d^{\ExponDim_4} ( 4 e^{\Constant+1}  \Constant^2   )^{ \theta} \left(   \left[ \int_{\R^d}  \|x\|^{2p \theta}  \, \nu_d (dx) \right]^{\nicefrac{1}{(2p)}} + d^{\theta(\ExponDim_1 + \ExponDim_2)}\right)\\
& \leq   2 \varepsilon  e^{\Constant}   \Constant d^{\ExponDim_4} ( 4 e^{\Constant+1}  \Constant^2   )^{ \theta}  \big( \Constant^{\theta} d^{\theta (\ExponDim_1 + \ExponDim_2)} + d^{\theta(\ExponDim_1 + \ExponDim_2)}  \big)\\
& \leq 4 \varepsilon  e^{\Constant}   \Constant  ( 4 e^{\Constant+1}  \Constant^3   )^{ \theta}  d^{\ExponDim_4 + \theta (\ExponDim_1 + \ExponDim_2)}.
\end{split}
\end{align}
Combining this with \eqref{eq:bound:3rd_term} and \eqref{eq:bound:3rd_term:1} establishes that for all $N, d \in \N$, $\varepsilon \in (0, 1]$  it holds that
\begin{align}
\begin{split}
& \left(\E \! \left[ \int_{\R^d} \big|  (\mathcal{R} (\mathbf{g}^{M_N, d}_{\varepsilon}) )(Y^{N, d, x}_N) - ( \mathcal{R} (\mathbf{g}^{M_N, d}_{\varepsilon}) )(X^{N, d, x, \varepsilon}_N) \big|^p \, \nu_d (dx) \right] \right)^{\!\nicefrac{1}{p}} \\
 & \leq \Constant d^{\ExponDim_6} \cdot 5  ( 4 e^{\Constant+1}  \Constant^3   )^{ \theta}  d^{\theta (\ExponDim_1 + \ExponDim_2)} \cdot 4 \varepsilon  e^{\Constant}   \Constant  ( 4 e^{\Constant+1}  \Constant^3   )^{ \theta}  d^{\ExponDim_4 + \theta (\ExponDim_1 + \ExponDim_2)}\\
 & \leq  20 \varepsilon  e^{\Constant}   \Constant^2  ( 4 e^{\Constant+1}  \Constant^3   )^{ 2\theta} d^{\ExponDim_4 + \ExponDim_6 + 2\theta (\ExponDim_1 + \ExponDim_2)}.
\end{split}
\end{align}
This, \eqref{eq:ass:u}, \eqref{eq:Lp:triangle}, and \eqref{eq:bound:2nd_term} prove for all $N, d \in \N$, $\varepsilon \in (0, 1]$ that
\begin{align*}
	&\left(\E \! \left[ \int_{\R^d} \big|u_d(x) - ( \mathcal{R} (\mathbf{g}^{M_N, d}_{\varepsilon}) )(X^{N, d, x, \varepsilon}_N) \big|^p \, \nu_d (dx) \right] \right)^{\!\nicefrac{1}{p}} \\
& \leq  \Constant d^{\ExponDim_0} N^{-\ExponN_0} +   3 \varepsilon \Constant  ( 4 e^{\Constant+1}  \Constant^3   )^{ \theta} d^{\ExponDim_5 + \theta (\ExponDim_1 + \ExponDim_2)}  + 20 \varepsilon  e^{\Constant}   \Constant^2  ( 4 e^{\Constant+1}  \Constant^3   )^{ 2\theta} d^{\ExponDim_4 + \ExponDim_6+ 2\theta (\ExponDim_1 + \ExponDim_2)}\\
& \leq \Constant d^{\ExponDim_0} N^{-\ExponN_0} +  23 \varepsilon  e^{\Constant}   \Constant^2  ( 4 e^{\Constant+1}  \Constant^3   )^{ 2\theta} d^{\delta}\\
& =  \Constant d^{\ExponDim_0} N^{-\ExponN_0} +  \tfrac{\varepsilon \gamma d^{\delta}}{2} . \numberthis
\end{align*}
Combining this and \eqref{eq:N:varepsilon} assures that for all $d \in \N$, $\varepsilon \in (0, 1]$ it holds that
\begin{align}
\begin{split}
& \left(\E \! \left[ \int_{\R^d} \Big|u_d(x) - \Big( \mathcal{R} \Big(\mathbf{g}^{ M_{\mathcal{N}_{d, \varepsilon}}, d }_{ \mathcal{E}_{d, \varepsilon} }\Big) \Big) \Big(X^{\mathcal{N}_{d, \varepsilon}, d, x, \mathcal{E}_{d, \varepsilon}}_{\mathcal{N}_{d, \varepsilon}} \Big) \Big|^p \, \nu_d (dx) \right] \right)^{\!\nicefrac{1}{p}} \leq  \frac{\varepsilon}{2} + \frac{\varepsilon}{2} = \varepsilon.
\end{split}
\end{align}
This and, e.g., \cite[Corollary~2.4]{JentzenSalimovaWelti2018} establish that there exists $\mathfrak{w} = (\mathfrak{w}_{d, \varepsilon})_{(d, \varepsilon) \in \N \times (0, 1]} \colon  \N \times (0, 1] \to \Omega$ which satisfies 
 for all $d \in \N$, $\varepsilon \in (0, 1]$  that
\begin{align}
\label{eq:error:last}
\begin{split}
& \left[ \int_{\R^d} \Big|u_d(x) - \Big( \mathcal{R} \Big( \mathbf{g}^{ M_{\mathcal{N}_{d, \varepsilon}}, d }_{ \mathcal{E}_{d, \varepsilon} } \Big) \Big)\Big(X^{\mathcal{N}_{d, \varepsilon}, d, x, \mathcal{E}_{d, \varepsilon}}_{\mathcal{N}_{d, \varepsilon}}(\mathfrak{w}_{d, \varepsilon})\Big) \Big|^p \, \nu_d (dx) \right]^{\nicefrac{1}{p}} \leq  \varepsilon.
\end{split}
\end{align}
Next note that for all $N,  d \in \N$,  $m \in \{1, 2, \ldots, M_{N}\}$,   $x \in \R^d$,  $\varepsilon \in (0, 1]$,  $\omega \in \Omega$ it holds that 	$X^{N, d, m, x, \varepsilon}_{0}(\omega)  = (\mathcal{R}( \idRelu_d))(x)$.
The assumption that for all $ d\in \N$, $\varepsilon \in (0, 1]$ it holds that 	$\idRelu_d \in \SubsetANNs_{d, \varepsilon}$ and   \eqref{eq:iteration:app} hence ensure 
that
 there exist $(\Phi^{N, d, m, \varepsilon, \omega}_{n})_{m \in \{1, 2, \ldots, M_{N}\}} \subseteq \SubsetANNs_{d, \varepsilon}$, $\omega \in \Omega$, $n \in \{0, 1, 2, \ldots, N\}$, $\varepsilon \in (0, 1]$, $d, N \in \N$, 
 which satisfy  for all $N,  d \in \N$,      $\varepsilon \in (0, 1]$, $n \in \{0, 1, 2, \ldots, N\}$, $\omega \in \Omega$ , $m \in \{1, 2, \ldots, M_{N}\}$, $x \in \R^d$  that  $\paramANN(\Phi^{N, d, m, \varepsilon, \omega}_{n}) \leq \paramANN(\idRelu_d)+ n \Constant N^{\ExponN_1} d^{\ExponDim_3} \varepsilon^{-\ExponError}$, $\mathcal{D}(\Phi^{N, d, m, \varepsilon, \omega}_{n}) = \mathcal{D}(\Phi^{N, d, 1, \varepsilon, \omega}_{n})$, and
\begin{align}
(\mathcal{R}( \Phi^{N, d, m, \varepsilon, \omega}_{n}))(x) = \Big( \mathcal{R} \Big(\mathbf{f}^{N, d}_{\varepsilon, Z^{N, d, m}_{n-1}(\omega)} \Big) \Big)  \big( X^{N, d, m, x, \varepsilon}_{n-1}(\omega)\big)= X^{N, d, m, x, \varepsilon}_{n}(\omega).
\end{align}
The assumption that for all $d \in \N$ it holds that $\paramANN(\idRelu_d) \leq  \Constant d^{\ExponDim_3} $ therefore implies  for all $N,  d \in \N$,  $m \in \{1, 2, \ldots, M_{N}\}$,    $\varepsilon \in (0, 1]$, $\omega \in \Omega$ that
\begin{align}
\paramANN(\Phi^{N, d, m, \varepsilon, \omega}_{N}) \leq  \Constant d^{\ExponDim_3}+ N \Constant N^{\ExponN_1} d^{\ExponDim_3} \varepsilon^{-\ExponError} \leq \Constant d^{\ExponDim_3}+ \Constant N^{\ExponN_1+1} d^{\ExponDim_3} \varepsilon^{-\ExponError}.
\end{align}
Therefore, we obtain that there exist $\Psi^{N, d, \varepsilon, \omega} \in \ANNs$, $\omega \in \Omega$, $\varepsilon \in (0, 1]$, $d, N \in \N$,
which satisfy for all  $N, d \in \N$, $\varepsilon \in (0, 1]$, $\omega \in \Omega$, $x \in \R^d$  that $\mathcal{R}(\Psi^{N, d, \varepsilon, \omega}) \in C(\R^d, \R)$, $\paramANN (\Psi^{N, d, \varepsilon, \omega}) \leq \Constant N^{\ExponN_2} (N^{\ExponN_1+1} \allowbreak d^{\ExponDim_3} \varepsilon^{-\ExponError} + \Constant d^{\ExponDim_3}+  \Constant N^{\ExponN_1+1} d^{\ExponDim_3} \varepsilon^{-\ExponError} ) $, and
\begin{align*}
\label{eq:DNN:last}
&(\mathcal{R} (\Psi^{N, d, \varepsilon, \omega}))(x)\\
 &= (\mathcal{R} (\mathbf{g}^{M_N, d}_{\varepsilon})\big((\mathcal{R} (\Phi^{N, d, 1, \varepsilon, \omega}_{N}))(x), (\mathcal{R} (\Phi^{N, d, 2, \varepsilon, \omega}_{N}))(x), \ldots, (\mathcal{R} (\Phi^{N, d, M_N, \varepsilon, \omega}_{N}))(x)\big)\\
& = (\mathcal{R} (\mathbf{g}^{M_N, d}_{\varepsilon})) \big( X^{N, d, 1, x, \varepsilon}_{N}(\omega), X^{N, d, 2, x, \varepsilon}_{N}(\omega), \ldots, X^{N, d, M_N, x, \varepsilon}_{N}(\omega)\big)\\
& = (\mathcal{R} (\mathbf{g}^{M_N, d}_{\varepsilon}))(X^{N, d, x, \varepsilon}_{N}(\omega)). \numberthis
\end{align*}
Hence, we obtain that   for all $N, d \in \N$, $\varepsilon \in (0, 1]$, $\omega \in \Omega$  it holds that
\begin{align}
\begin{split}
\paramANN (\Psi^{N, d, \varepsilon, \omega}) &\leq \Constant^2 N^{\ExponN_2} d^{\ExponDim_3} \varepsilon^{-\ExponError}(2N^{\ExponN_1+1} +1  )\\
& \leq 3 \Constant^2 N^{\ExponN_1+ \ExponN_2+1} d^{\ExponDim_3} \varepsilon^{-\ExponError}.
\end{split}
\end{align}
This and \eqref{eq:N:varepsilon}  demonstrate  that for all $d \in \N$, $\varepsilon \in (0, 1]$ it holds that
\begin{align}
\begin{split}
&\paramANN (\Psi^{\mathcal{N}_{d, \varepsilon}, d, \mathcal{E}_{d, \varepsilon}, \mathfrak{w}_{d, \varepsilon}})\\
& \leq 3 \Constant^2 2^{ \ExponN_1+\ExponN_2 +1} \big(\tfrac{2\Constant d^{\ExponDim_0}}{\varepsilon}\big)^{\frac{( \ExponN_1+\ExponN_2 +1)}{\ExponN_0}}  d^{ \ExponDim_3} \varepsilon^{-\ExponError} \gamma^{ \ExponError} d^{ \ExponError \delta}\\
& \leq  3 \Constant^2 2^{\ExponN_1+\ExponN_2 + 1} (2 \Constant)^{\frac{ \ExponN_1+\ExponN_2 +1}{\ExponN_0}}
\gamma^{ \ExponError} d^{\frac{\ExponDim_0( \ExponN_1+\ExponN_2 +1)}{\ExponN_0} +  \ExponDim_3 +   \ExponError \delta}  \varepsilon^{-\frac{( \ExponN_1+ \ExponN_2 +1)}{\ExponN_0} -\ExponError}.  
\end{split}
\end{align}
Combining this, \eqref{eq:error:last}, and \eqref{eq:DNN:last} establishes \eqref{eq:prop:statement}. The proof of Theorem~\ref{thm:main} is thus completed.
\end{proof}

\section{Artificial neural network (ANN) calculus}
\label{sec:calculus}

In this section we establish in Lemma~\ref{lem:sum:ANN} and Lemma~\ref{lem:Composition_Sum} below a few elementary results on representation flexibilities of ANNs. In our proofs of Lemma~\ref{lem:sum:ANN} and Lemma~\ref{lem:Composition_Sum} we use results from a certain ANN calculus which we recall and extend in Subsections~\ref{subsec:ANNs}--\ref{subsec:sums}.  In particular,
 Definition~\ref{Def:ANN} below is \cite[Definitions~2.1]{GrohsHornungJentzen2019},
 Definition~\ref{Def:multidim_version} below is \cite[Definitions~2.2]{GrohsHornungJentzen2019}, 
 Definition~\ref{Definition:ANNrealization} below is \cite[Definitions~2.3]{GrohsHornungJentzen2019}, 
  Definition~\ref{Definition:ANNcomposition} below is \cite[Definitions~2.5]{GrohsHornungJentzen2019}, 
  and
   Definition~\ref{Definition:simpleParallelization} below is \cite[Definitions~2.17]{GrohsHornungJentzen2019}.

\subsection{ANNs}
\label{subsec:ANNs}

\begin{definition}[ANNs]
	\label{Def:ANN}
	We denote by $\ANNs$ the set given by 
	\begin{equation}
	\begin{split}
	\ANNs
	&=
	\cup_{L \in \N}
	\cup_{ (l_0,l_1,\ldots, l_L) \in \N^{L+1} }
	\left(
	\times_{k = 1}^L (\R^{l_k \times l_{k-1}} \times \R^{l_k})
	\right)
	\end{split}
	\end{equation}
	and we denote by 	$
	\paramANN, 
	\lengthANN,  \inDimANN, \outDimANN \colon \ANNs \to \N
	$, $\hiddenLength \colon \ANNs \to \N_0$, 
	and
	$\dims\colon\ANNs\to  \cup_{L=2}^\infty\, \N^{L}$
	the functions which satisfy
	for all $ L\in\N$, $l_0,l_1,\ldots, l_L \in \N$, 
	$
	\Phi 
	\in  \allowbreak
	( \times_{k = 1}^L(\R^{l_k \times l_{k-1}} \times \R^{l_k}))$
	that
	$\paramANN(\Phi)
	=
	\sum_{k = 1}^L l_k(l_{k-1} + 1) 
	$, $\lengthANN(\Phi)=L$,  $\inDimANN(\Phi)=l_0$,  $\outDimANN(\Phi)=l_L$, $\hiddenLength(\Phi)=L-1$, and $\dims(\Phi)= (l_0,l_1,\ldots, l_L)$.
\end{definition}

\subsection{Realizations of ANNs}
\label{subsec:realizations}

\begin{definition}[Multidimensional versions]
	\label{Def:multidim_version}
	Let $d \in \N$ and let $\psi \colon \R \to \R$ be a function.
	Then we denote by $\mathfrak{M}_{\psi, d} \colon \R^d \to \R^d$ the function which satisfies for all $ x = ( x_1, \dots, x_{d} ) \in \R^{d} $ that
	\begin{equation}\label{multidim_version:Equation}
	\mathfrak{M}_{\psi, d}( x ) 
	=
	\left(
	\psi(x_1)
	,
	\ldots
	,
	\psi(x_d)
	\right).
	\end{equation}
\end{definition}

\begin{definition}[Realizations associated to ANNs]
	\label{Definition:ANNrealization}
	Let $a\in C(\R,\R)$.
	Then we denote by 
	$
	\functionANN \colon \ANNs \to \cup_{k,l\in\N}\,C(\R^k,\R^l)
	$
	the function which satisfies
	for all  $ L\in\N$, $l_0,l_1,\ldots, l_L \in \N$, 
	$
	\Phi 
	=
	((W_1, B_1),(W_2, B_2),\allowbreak \ldots, (W_L,\allowbreak B_L))
	\in  \allowbreak
	( \times_{k = 1}^L\allowbreak(\R^{l_k \times l_{k-1}} \times \R^{l_k}))
	$,
	$x_0 \in \R^{l_0}, x_1 \in \R^{l_1}, \ldots, x_{L-1} \in \R^{l_{L-1}}$ 
	with $\forall \, k \in \N \cap (0,L) \colon x_k =\activationDim{l_k}(W_k x_{k-1} + B_k)$  
	that
	\begin{equation}
	\label{setting_NN:ass2}
	\functionANN(\Phi) \in C(\R^{l_0},\R^{l_L})\qandq
	( \functionANN(\Phi) ) (x_0) = W_L x_{L-1} + B_L
	\end{equation}
	(cf.~Definition~\ref{Def:ANN} and Definition~\ref{Def:multidim_version}).
\end{definition}

\subsection{Compositions of ANNs}
\label{subsec:compositions}

\begin{definition}[Compositions of ANNs]
	\label{Definition:ANNcomposition}
	We denote by $\compANN{(\cdot)}{(\cdot)}\colon \{(\Phi_1,\Phi_2)\allowbreak\in\ANNs\times \ANNs\colon \inDimANN(\Phi_1)=\outDimANN(\Phi_2)\}\to\ANNs$ the function which satisfies for all 
	$ L,\mathscr{L}\in\N$, $l_0,l_1,\ldots, l_L,\allowbreak \mathfrak{l}_0,\mathfrak{l}_1,\ldots, \mathfrak{l}_\mathscr{L} \in \N$, 
	$
	\Phi_1
	=
	((W_1, B_1),(W_2, B_2),\allowbreak \ldots, (W_L,\allowbreak B_L))
	\in  \allowbreak
	( \times_{k = 1}^L(\R^{l_k \times l_{k-1}} \times \R^{l_k}))
	$,
	$
	\Phi_2
	=
	((\mathscr{W}_1, \mathscr{B}_1),\allowbreak(\mathscr{W}_2, \mathscr{B}_2),\allowbreak \ldots, (\mathscr{W}_\mathscr{L},\allowbreak \mathscr{B}_\mathscr{L}))
	\in  \allowbreak
	( \times_{k = 1}^\mathscr{L}\allowbreak(\R^{\mathfrak{l}_k \times \mathfrak{l}_{k-1}} \times \R^{\mathfrak{l}_k}))
	$ 
	with $l_0=\inDimANN(\Phi_1)=\outDimANN(\Phi_2)=\mathfrak{l}_{\mathscr{L}}$
	that
	\begin{equation}\label{ANNoperations:Composition}
	\begin{split}
	&\compANN{\Phi_1}{\Phi_2}=\\&
	\begin{cases} 
	\begin{array}{r}
	\big((\mathscr{W}_1, \mathscr{B}_1),(\mathscr{W}_2, \mathscr{B}_2),\ldots, (\mathscr{W}_{\mathscr{L}-1},\allowbreak \mathscr{B}_{\mathscr{L}-1}),
	(W_1 \mathscr{W}_{\mathscr{L}}, W_1 \mathscr{B}_{\mathscr{L}}+B_{1}),\\ (W_2, B_2), (W_3, B_3),\ldots,(W_{L},\allowbreak B_{L})\big)
	\end{array}
	&: L>1<\mathscr{L} \\[3ex]
	\big( (W_1 \mathscr{W}_{1}, W_1 \mathscr{B}_1+B_{1}), (W_2, B_2), (W_3, B_3),\ldots,(W_{L},\allowbreak B_{L}) \big)
	&: L>1=\mathscr{L}\\[1ex]
	\big((\mathscr{W}_1, \mathscr{B}_1),(\mathscr{W}_2, \mathscr{B}_2),\allowbreak \ldots, (\mathscr{W}_{\mathscr{L}-1},\allowbreak \mathscr{B}_{\mathscr{L}-1}),(W_1 \mathscr{W}_{\mathscr{L}}, W_1 \mathscr{B}_{\mathscr{L}}+B_{1}) \big)
	&: L=1<\mathscr{L}  \\[1ex]
	(W_1 \mathscr{W}_{1}, W_1 \mathscr{B}_1+B_{1}) 
	&: L=1=\mathscr{L} 
	\end{cases}
	\end{split}
	\end{equation}
	(cf.\ Definition~\ref{Def:ANN}).
\end{definition}

\subsection{Parallelizations of ANNs with the same length}
\label{subsec:parallel}

\begin{definition}[Parallelizations of ANNs with the same length]
	\label{Definition:simpleParallelization}
	Let $n\in\N$. Then we denote by 
	\begin{equation}
	\parallelizationSpecial_{n}\colon \big\{(\Phi_1,\Phi_2,\dots, \Phi_n)\in\ANNs^n\colon \lengthANN(\Phi_1)= \lengthANN(\Phi_2)=\ldots =\lengthANN(\Phi_n) \big\}\to \ANNs
	\end{equation}
	the function which satisfies for all  $L\in\N$,
	$(l_{1,0},l_{1,1},\dots, l_{1,L}), (l_{2,0},l_{2,1},\dots, l_{2,L}),\dots,\allowbreak (l_{n,0},\allowbreak l_{n,1},\allowbreak\dots, l_{n,L})\in\N^{L+1}$, 
	$\Phi_1=((W_{1,1}, B_{1,1}),(W_{1,2}, B_{1,2}),\allowbreak \ldots, (W_{1,L},\allowbreak B_{1,L}))\in ( \times_{k = 1}^L\allowbreak(\R^{l_{1,k} \times l_{1,k-1}} \times \R^{l_{1,k}}))$, 
	$\Phi_2=((W_{2,1}, B_{2,1}),(W_{2,2}, B_{2,2}),\allowbreak \ldots, (W_{2,L},\allowbreak B_{2,L}))\in ( \times_{k = 1}^L\allowbreak(\R^{l_{2,k} \times l_{2,k-1}} \times \R^{l_{2,k}}))$,
	\dots,  
	$\Phi_n=((W_{n,1}, B_{n,1}),(W_{n,2}, B_{n,2}),\allowbreak \ldots, (W_{n,L},\allowbreak B_{n,L}))\in ( \times_{k = 1}^L\allowbreak(\R^{l_{n,k} \times l_{n,k-1}} \times \R^{l_{n,k}}))$
	%
	%
	that
	\begin{equation}\label{parallelisationSameLengthDef}
	\begin{split}
	\parallelizationSpecial_{n}(\Phi_1,\Phi_2,\dots,\Phi_n)&=
	\left(
	\pa{\begin{pmatrix}
		W_{1,1}& 0& 0& \cdots& 0\\
		0& W_{2,1}& 0&\cdots& 0\\
		0& 0& W_{3,1}&\cdots& 0\\
		\vdots& \vdots&\vdots& \ddots& \vdots\\
		0& 0& 0&\cdots& W_{n,1}
		\end{pmatrix} ,\begin{pmatrix}B_{1,1}\\B_{2,1}\\B_{3,1}\\\vdots\\ B_{n,1}\end{pmatrix}},\right.
	\\&\quad
	\pa{\begin{pmatrix}
		W_{1,2}& 0& 0& \cdots& 0\\
		0& W_{2,2}& 0&\cdots& 0\\
		0& 0& W_{3,2}&\cdots& 0\\
		\vdots& \vdots&\vdots& \ddots& \vdots\\
		0& 0& 0&\cdots& W_{n,2}
		\end{pmatrix} ,\begin{pmatrix}B_{1,2}\\B_{2,2}\\B_{3,2}\\\vdots\\ B_{n,2}\end{pmatrix}}
	,\dots,
	\\&\quad\left.
	\pa{\begin{pmatrix}
		W_{1,L}& 0& 0& \cdots& 0\\
		0& W_{2,L}& 0&\cdots& 0\\
		0& 0& W_{3,L}&\cdots& 0\\
		\vdots& \vdots&\vdots& \ddots& \vdots\\
		0& 0& 0&\cdots& W_{n,L}
		\end{pmatrix} ,\begin{pmatrix}B_{1,L}\\B_{2,L}\\B_{3,L}\\\vdots\\ B_{n,L}\end{pmatrix}}\right)
	\end{split}
	\end{equation}
	(cf.\ Definition~\ref{Def:ANN}).
\end{definition}

\subsection{Linear transformations of ANNs}
\label{subsec:linear}

\begin{definition}[Identity matrix] \label{Definition:identityMatrix}
	Let $n\in\N$. Then we denote by $\idMatrix_{n}\in \R^{n\times n}$ the identity matrix in $\R^{n\times n}$.
\end{definition}

\begin{definition}[ANNs with a vector input]
\label{Definition:ANN:vector}
Let $n \in \N$, $B \in \R^n$. Then we denote by $\vectorANN_B \in (\R^{n \times n} \times \R^n)$ the pair given by $\vectorANN_B = (\idMatrix_n, B)$ (cf.~Definition~\ref{Definition:identityMatrix}).
\end{definition}

\begin{lemma}
	\label{lem:ANN:vector}
Let $n \in \N$, $B \in \R^n$. Then
\begin{enumerate}[(i)]
	\item\label{item:ANN:vector:1}  it holds that $\vectorANN_B \in \ANNs$,
	\item\label{item:ANN:vector:2} it holds that $\dims(\vectorANN_{B}) = (n, n) \in \N^2$,
	\item\label{item:ANN:vector:3.1} it holds for all $a \in C(\R, \R)$ that $\functionANN (\vectorANN_{B}) \in C(\R^n, \R^n)$, and
		\item\label{item:ANN:vector:3.2} it holds for all $a \in C(\R, \R)$, $x \in \R^n$ that 
	\begin{equation}
	(\functionANN (\vectorANN_{B})) (x) = x + B
	\end{equation}
\end{enumerate}
(cf.~Definition~\ref{Def:ANN}, Definition~\ref{Definition:ANNrealization}, and 
Definition~\ref{Definition:ANN:vector}).
\end{lemma}
\begin{proof}[Proof of Lemma~\ref{lem:ANN:vector}]
Note that the fact that $\vectorANN_B \in (\R^{n \times n} \times \R^n)$ ensures that $\vectorANN_B \in \ANNs$ and $	\dims(\vectorANN_{B}) = (n, n) \in \N^2$. This establishes items~\eqref{item:ANN:vector:1}--\eqref{item:ANN:vector:2}.
The fact that $ \vectorANN_B = (\idMatrix_n, B) $ (cf.~Definition~\ref{Definition:identityMatrix})  and \eqref{setting_NN:ass2} therefore prove that for all $a \in C(\R, \R)$, $x \in \R^n$ it holds that $\functionANN (\vectorANN_{B}) \in C(\R^n, \R^n)$ and
\begin{equation}
(\functionANN (\vectorANN_{B})) (x) = x + B.
\end{equation}
This establishes items~\eqref{item:ANN:vector:3.1}--\eqref{item:ANN:vector:3.2}. 
The proof of Lemma~\ref{lem:ANN:vector} is thus completed.
\end{proof}

\begin{lemma}
	\label{lem:ANN:vector:comp}
	Let $\Phi \in \ANNs$ (cf.~Definition~\ref{Def:ANN}). Then
	\begin{enumerate}[(i)]
		\item\label{item:ANN:vector:comp:1}
		it holds for all $B \in \R^{\outDimANN(\Phi)} $ that 		$\dims (\compANN{\vectorANN_B}{\Phi}) = \dims (\Phi)$,
		\item\label{item:ANN:vector:comp:2}
			it holds for all $B \in \R^{\outDimANN(\Phi)}$, $a \in C(\R, \R)$  that $\functionANN(\compANN{\vectorANN_B}{\Phi}) \in C(\R^{\inDimANN(\Phi)}, \R^{\outDimANN(\Phi)}) $,
		\item\label{item:ANN:vector:comp:3}
		 it holds for all $B \in \R^{\outDimANN(\Phi)} $, $a \in C(\R, \R)$, $x \in \R^{\inDimANN(\Phi)}$  that
		\begin{equation}
		(\functionANN(\compANN{\vectorANN_B}{\Phi}))(x) = (\functionANN(\Phi))(x) + B,
		\end{equation}
		\item\label{item:ANN:vector:comp:4}
		it holds for all $B \in \R^{\inDimANN(\Phi)} $ that
		$\dims(\compANN{\Phi}{\vectorANN_B}) = \dims(\Phi)$,
		\item\label{item:ANN:vector:comp:5}
		it holds for all $B \in \R^{\inDimANN(\Phi)} $, $a \in C(\R, \R)$ that $\functionANN(\compANN{\Phi}{\vectorANN_B}) \in C(\R^{\inDimANN(\Phi)}, \R^{\outDimANN(\Phi)}) $, and
		\item\label{item:ANN:vector:comp:6}
 it holds for all $B \in \R^{\inDimANN(\Phi)} $, $a \in C(\R, \R)$, $x \in \R^{\outDimANN(\Phi)}$    that
		\begin{equation}
		(\functionANN(\compANN{\Phi}{\vectorANN_B}))(x) = (\functionANN(\Phi))(x+B)
		\end{equation}
	\end{enumerate}
	(cf.~Definition~\ref{Definition:ANNrealization},
	Definition~\ref{Definition:ANNcomposition}, and 
	Definition~\ref{Definition:ANN:vector}).
\end{lemma}
\begin{proof}[Proof of Lemma~\ref{lem:ANN:vector:comp}]
Note that Lemma~\ref{lem:ANN:vector} demonstrates that for all $n \in \N$, $B \in \R^n$, $a \in C(\R, \R)$, $x \in \R^n$ it holds that $\dims(\vectorANN_{B}) = (n, n)$, $\functionANN (\vectorANN_{B}) \in C(\R^n, \R^n)$, and
\begin{equation}
(\functionANN (\vectorANN_{B})) (x) = x + B.
\end{equation}
Combining this and, e.g., \cite[Proposition~2.6]{GrohsHornungJentzen2019} establishes items \eqref{item:ANN:vector:comp:1}--\eqref{item:ANN:vector:comp:6}.
	The proof of Lemma~\ref{lem:ANN:vector:comp} is thus completed.
\end{proof}

\begin{definition}[ANNs with a matrix input]
	\label{Definition:ANN:matrix}
Let $m, n \in \N$, $W \in \R^{m \times n}$. Then 
we denote by $\matrixANN_{W} \in (\R^{m \times n} \times \R^{m})$ the pair 
given by $\matrixANN_{W} = (W, 0)$.
\end{definition}

\begin{lemma}
\label{lem:ANN:matrix}
Let $m, n \in \N$, $W \in \R^{m \times n}$. Then
\begin{enumerate}[(i)]
	\item\label{item:ANN:matrix:1} it holds that $\matrixANN_W \in \ANNs$,
	\item\label{item:ANN:matrix:2} it holds that $\dims(\matrixANN_{W}) = (n, m) \in \N^2$,
	\item\label{item:ANN:matrix:3} it holds for all $a \in C(\R, \R)$ that $\functionANN(\matrixANN_{W}) \in C(\R^n, \R^m)$, and
		\item\label{item:ANN:matrix:4} it holds for all $a \in C(\R, \R)$, $x \in \R^n$ that 
	\begin{equation}
	(\functionANN(\matrixANN_{W})) (x) = Wx
	\end{equation}
\end{enumerate}
(cf.~Definition~\ref{Def:ANN}, Definition~\ref{Definition:ANNrealization}, and 
Definition~\ref{Definition:ANN:matrix}).
\end{lemma}
\begin{proof}[Proof of Lemma~\ref{lem:ANN:matrix}]
Note that the fact that $\matrixANN_{W} \in (\R^{m \times n} \times \R^{m})$ ensures that $\matrixANN_W \in \ANNs$ and $	\dims(\matrixANN_{W}) = (n, m) \in \N^2$. This establishes items~\eqref{item:ANN:matrix:1}--\eqref{item:ANN:matrix:2}.
Next observe that the fact that $\matrixANN_W = (W, 0)  \in (\R^{m \times n} \times \R^{m})$ and \eqref{setting_NN:ass2} prove that for all $a \in C(\R, \R)$, $x \in \R^n$ it holds that $\functionANN(\matrixANN_{W}) \in C(\R^n, \R^m)$ and
\begin{equation}
(\functionANN(\matrixANN_{W})) (x) = Wx.
\end{equation}
This establishes items~\eqref{item:ANN:matrix:3}--\eqref{item:ANN:matrix:4}.
The proof of Lemma~\ref{lem:ANN:matrix} is thus completed.
\end{proof}

\begin{lemma}
	\label{lem:ANN:matrix:comp}
	Let  $a \in C(\R, \R)$, $\Phi \in \ANNs$ (cf.~Definition~\ref{Def:ANN}). Then
	\begin{enumerate}[(i)]
		\item\label{item:ANN:matrix:comp:1} it holds 
		for all $m \in \N$, $W \in \R^{m \times \outDimANN (\Phi)}$ that $\functionANN(\compANN{\matrixANN_W}{\Phi}) \in C(\R^{\inDimANN(\Phi)}, \R^m) $,
		\item\label{item:ANN:matrix:comp:2}  it holds 
		for all $m \in \N$, $W \in \R^{m \times \outDimANN (\Phi)}$, $x \in \R^{\inDimANN(\Phi)}$ that
			\begin{equation}
		(\functionANN(\compANN{\matrixANN_W}{\Phi}))(x) = W \big((\functionANN(\Phi))(x)\big),
		\end{equation}
		\item\label{item:ANN:matrix:comp:3} 
		it holds for all
		$n \in \N$, $W \in \R^{ \inDimANN (\Phi) \times n}$ that $\functionANN(\compANN{\Phi}{\matrixANN_W}) \in C(\R^n, \R^{\outDimANN(\Phi)}) $,  and
		\item\label{item:ANN:matrix:comp:4} it holds for all 	$n \in \N$, $W \in \R^{ \inDimANN (\Phi) \times n}$, $x \in \R^n$ that  
		\begin{equation}
		(\functionANN(\compANN{\Phi}{\matrixANN_W}))(x) = (\functionANN(\Phi))(Wx)
		\end{equation}
	\end{enumerate}
	(cf.~Definition~\ref{Definition:ANNrealization},
	Definition~\ref{Definition:ANNcomposition}, and 
	Definition~\ref{Definition:ANN:matrix}).
\end{lemma}
\begin{proof}[Proof of Lemma~\ref{lem:ANN:matrix:comp}]
Note that Lemma~\ref{lem:ANN:matrix} demonstrates that for all  $m, n \in \N$, $W \in \R^{m \times n}$, $x \in \R^n$ it holds that $\functionANN(\matrixANN_{W}) \in C(\R^n, \R^m)$ and
	\begin{equation}
	(\functionANN(\matrixANN_{W})) (x) = Wx.
	\end{equation}
 Combining this and, e.g., \cite[Proposition~2.6]{GrohsHornungJentzen2019} establishes items \eqref{item:ANN:matrix:comp:1}--\eqref{item:ANN:matrix:comp:4}.
	The proof of Lemma~\ref{lem:ANN:matrix:comp} is thus completed.
\end{proof}

\begin{definition}[Scalar multiplications of ANNs]
	\label{Definition:ANNscalar}
We denote by $\left(\cdot\right) \circledast \left(\cdot\right) \colon \R \times \ANNs \to \ANNs$ the function which satisfies for all $\lambda \in \R$, $\Phi \in \ANNs$ that
\begin{equation}
\lambda \circledast  \Phi = \compANN{\matrixANN_{\lambda \idMatrix_{\outDimANN(\Phi)}}}{\Phi}
\end{equation}
(cf.\ Definition~\ref{Def:ANN}, Definition~\ref{Definition:ANNcomposition}, Definition~\ref{Definition:identityMatrix},
and
Definition~\ref{Definition:ANN:matrix}).
\end{definition}

\begin{lemma}
\label{lem:ANNscalar}
Let $\lambda \in \R$, $\Phi \in \ANNs$ (cf.~Definition~\ref{Def:ANN}). Then
\begin{enumerate}[(i)]
	\item\label{item:ANN:scalar:1} it holds that $\dims(\lambda \circledast \Phi ) = \dims(\Phi)$,
	\item\label{item:ANN:scalar:2} it holds for all $a \in C(\R, \R)$ that $\functionANN(\lambda \circledast \Phi) \in C(\R^{\inDimANN(\Phi)}, \R^{\outDimANN(\Phi)})$,
	 and 
	 \item\label{item:ANN:scalar:3} it holds for all $a \in C(\R, \R)$, $x \in \R^{\inDimANN(\Phi)}$ that
	\begin{equation}
	(\functionANN(\lambda \circledast \Phi))(x) = \lambda \big( (\functionANN(\Phi))(x) \big)
	\end{equation}
\end{enumerate}
(cf.~Definition~\ref{Definition:ANNrealization} and Definition~\ref{Definition:ANNscalar}).
\end{lemma}
\begin{proof}[Proof of Lemma~\ref{lem:ANNscalar}]
Throughout this proof let $L \in \N$, $l_0, l_1, \ldots, l_L \in \N$ satisfy that $ L = \lengthANN(\Phi)$ and
$ (l_0, l_1, \ldots, l_L) = \dims(\Phi)$. Note that item~\eqref{item:ANN:matrix:2} in Lemma~\ref{lem:ANN:matrix} proves that
\begin{equation}
\dims(\matrixANN_{\lambda \idMatrix_{\outDimANN(\Phi)}}) =  (\outDimANN(\Phi), \outDimANN(\Phi))
\end{equation}
(cf.~Definition~\ref{Definition:identityMatrix} and Definition~\ref{Definition:ANN:matrix}).
Combining this and, e.g., \cite[item~(i) in Proposition~2.6]{GrohsHornungJentzen2019} assures that
\begin{equation}
\dims(\lambda \circledast \Phi ) = \dims (\compANN{\matrixANN_{\lambda \idMatrix_{\outDimANN(\Phi)}}}{\Phi}) = (l_0, l_1, \ldots, l_{L-1}, \outDimANN(\Phi)) = \dims (\Phi).
\end{equation}
This establishes item~\eqref{item:ANN:scalar:1}. 
Moreover, observe that  items~\eqref{item:ANN:matrix:comp:1}--\eqref{item:ANN:matrix:comp:2} in Lemma~\ref{lem:ANN:matrix:comp} demonstrate that for all $a \in C(\R, \R)$, $x \in \R^{\inDimANN(\Phi)}$ it holds that $\functionANN(\lambda \circledast \Phi) \in C(\R^{\inDimANN(\Phi)}, \R^{\outDimANN(\Phi)})$ and
\begin{align}
\begin{split}
(\functionANN(\lambda \circledast \Phi))(x) &= (\functionANN (\compANN{\matrixANN_{\lambda \idMatrix_{\outDimANN(\Phi)}}}{\Phi}))(x)\\
& = \lambda  \idMatrix_{\outDimANN(\Phi)} \! \big( (\functionANN (\Phi))(x) \big) = \lambda \big( (\functionANN(\Phi))(x) \big).
\end{split}
\end{align} 
This establishes items~\eqref{item:ANN:scalar:2}--\eqref{item:ANN:scalar:3}.
The proof of Lemma~\ref{lem:ANNscalar} is thus completed.
\end{proof}

\subsection{Representations of the identities with rectifier functions}
\label{subsec:identity}

\begin{definition}
	\label{Definition:ReLu:identity}
	We denote by $\idRelu = (\idRelu_d)_{d \in \N} \colon \N \to \ANNs$ the function which satisfies for all $d \in \N$ that 
	\begin{equation}
	\label{eq:def:id:1}
	\idRelu_1 = \left( \left(\begin{pmatrix}
	1\\
	-1
	\end{pmatrix}, \begin{pmatrix}
	0\\
	0
	\end{pmatrix} \right),
	\Big(	\begin{pmatrix}
	1& -1
	\end{pmatrix}, 
	0 \Big)
	 \right)  \in \big((\R^{2 \times 1} \times \R^{2}) \times (\R^{1 \times 2} \times \R^1) \big)
	\end{equation} 
	and 
	\begin{equation}
	\idRelu_d = \paraANN{d} (\idRelu_1, \idRelu_1, \ldots, \idRelu_1)
	\end{equation}
	(cf.~Definition~\ref{Def:ANN} and
	Definition~\ref{Definition:simpleParallelization}).
\end{definition}

\begin{lemma}
\label{lem:Relu:identity}
Let $d \in \N$, $a \in C(\R, \R)$ satisfy for all $x \in \R$ that $a(x) = \max\{x, 0\}$. Then
\begin{enumerate}[(i)]
	\item
	\label{item:lem:Relu:dims} 
	it holds that $\dims(\idRelu_d) = (d, 2d, d) \in \N^3$,
		\item
	\label{item:lem:Relu:cont} 
	it holds that $ \functionANN(\idRelu_d) \in C(\R^d, \R^d)$, and
	\item
	\label{item:lem:Relu:real}
	 it holds for all $x \in \R^d$ that 
	 \begin{equation}
(\functionANN(\idRelu_d))(x) = x
\end{equation} 
\end{enumerate}
(cf.~Definition~\ref{Def:ANN}, Definition~\ref{Definition:ANNrealization}, and Definition~\ref{Definition:ReLu:identity}).
\end{lemma}
\begin{proof}[Proof of Lemma~\ref{lem:Relu:identity}]
Throughout this proof let $L =2$, $l_0 = 1$, $l_1 = 2$, $l_2 =1$. Note that \eqref{eq:def:id:1} ensures that
\begin{equation}
\dims (\idRelu_1) = (1, 2, 1) = (l_0, l_1, l_2).
\end{equation}
This and, e.g., \cite[Lemma~2.18]{GrohsHornungJentzen2019} prove that
\begin{equation}
\begin{split}
&\paraANN{d} (\idRelu_1, \idRelu_1, \ldots, \idRelu_1) \\
&\in \big(\! \times_{k = 1}^L \big( \R^{(d l_k) \times (d l_{k-1})} \times \R^{(d l_k)}\big) \big) = \big(\big(\R^{(2d) \times d} \times \R^{2d}\big) \times \big(\R^{d \times (2d)} \times \R^d\big) \big)
\end{split}
\end{equation}
(cf.~Definition~\ref{Definition:simpleParallelization}).
Hence, we obtain that $\dims(\idRelu_d) = (d, 2d, d) \in \N^3$. This establishes item~\eqref{item:lem:Relu:dims}.
Next note that \eqref{eq:def:id:1} assures that for all $x \in \R$ it holds that
\begin{equation}
(\functionANN(\idRelu_1))(x) = a(x) - a(-x) = \max\{x, 0\} - \max\{-x, 0\} = x.
\end{equation}
Combining this and, e.g., \cite[Proposition~2.19]{GrohsHornungJentzen2019}
demonstrates that for all $ x = (x_1, x_2, \ldots, x_d) \in \R^d$ it holds that $\functionANN(\idRelu_d) \in C(\R^d, \R^d)$ and
\begin{equation}
\begin{split}
(\functionANN(\idRelu_d))(x) &= \big( \functionANN \big(\paraANN{d} (\idRelu_1, \idRelu_1, \ldots, \idRelu_1)\big)\big)(x_1, x_2, \ldots, x_d)\\
& =  \big( (\functionANN(\idRelu_1))(x_1), (\functionANN(\idRelu_1))(x_2), \ldots, (\functionANN(\idRelu_1))(x_d)\big)\\
& = (x_1, x_2, \ldots, x_d) = x.
\end{split}
\end{equation}
This establishes 
 items~\eqref{item:lem:Relu:cont}--\eqref{item:lem:Relu:real}.
The proof of Lemma~\ref{lem:Relu:identity} is thus completed.
\end{proof}

\subsection{Sums of ANNs with the same length}
\label{subsec:sums}

\begin{definition}
	\label{Definition:ANN:sum}
	Let $m, n \in \N$. Then we denote by $\sumANN_{m, n} \in (\R^{m \times (nm)} \times \R^m)$ the pair given by
	\begin{equation}
	\sumANN_{m, n} = \matrixANN_{(\idMatrix_m \,\,\,  \idMatrix_m \,\,\, \ldots \,\,\, \idMatrix_m) }
	\end{equation}
	(cf.~Definition~\ref{Definition:identityMatrix} and Definition~\ref{Definition:ANN:matrix}).
\end{definition}

\begin{lemma}
	\label{lem:def:ANNsum}
	Let $m, n \in \N$. Then
	\begin{enumerate}[(i)]
		\item\label{item:ANNsum:1} it holds that $\sumANN_{m, n} \in \ANNs$,
		\item\label{item:ANNsum:2} it holds that $\dims (\sumANN_{m, n}) = (nm, m) \in \N^2$,
		\item\label{item:ANNsum:3} it holds for all $a \in C(\R, \R)$ that $\functionANN(\sumANN_{m, n}) \in C(\R^{nm}, \R^m)$, and
		\item\label{item:ANNsum:4} it holds for all $a \in C(\R, \R)$, $x_1, x_2, \ldots, x_n \in \R^{m}$ that 
		\begin{equation}
		(\functionANN(\sumANN_{m, n})) (x_1, x_2, \ldots, x_n) = \smallsum_{k=1}^n x_k
		\end{equation}
	\end{enumerate}
	(cf.~Definition~\ref{Def:ANN}, Definition~\ref{Definition:ANNrealization}, and 
	Definition~\ref{Definition:ANN:sum}).
\end{lemma}
\begin{proof}[Proof of Lemma~\ref{lem:def:ANNsum}]
	Note that the fact that $\sumANN_{m, n} \in (\R^{m \times (nm)} \times \R^m)$ ensures that $\sumANN_{m, n} \in \ANNs$ and $\dims (\sumANN_{m, n}) = (nm, m) \in \N^2$. This establishes items~\eqref{item:ANNsum:1}--\eqref{item:ANNsum:2}. Next observe that items~\eqref{item:ANN:matrix:3}--\eqref{item:ANN:matrix:4} in Lemma~\ref{lem:ANN:matrix} prove that for all $a \in C(\R, \R)$, $x_1, x_2, \ldots, x_n \in \R^{m}$  it holds that $\functionANN(\sumANN_{m, n}) \in C(\R^{nm}, \R^m)$ and
	\begin{equation}
	\begin{split}
	(\functionANN(\sumANN_{m, n})) (x_1, x_2, \ldots, x_n) &=
	\big( \functionANN \big( \matrixANN_{(\idMatrix_m \,\,\,  \idMatrix_m \,\,\, \ldots \,\,\, \idMatrix_m) } \big)\big)(x_1, x_2, \ldots, x_n)\\
	&= (\idMatrix_m \,\,\,  \idMatrix_m \,\,\, \ldots \,\,\, \idMatrix_m) (x_1, x_2, \ldots, x_n) =  \smallsum_{k=1}^n x_k
	\end{split}
	\end{equation}
	(cf.~Definition~\ref{Definition:identityMatrix} and Definition~\ref{Definition:ANN:matrix}).
	This establishes items~\eqref{item:ANNsum:3}--\eqref{item:ANNsum:4}.
	The proof of Lemma~\ref{lem:def:ANNsum} is thus completed.
\end{proof}

\begin{lemma}
	\label{lem:def:ANNsum:comp:left}
	Let $m, n \in \N$, $a \in C(\R, \R)$,  $\Phi  \in \{\Psi \in \ANNs \colon \outDimANN(\Psi) = nm\}$ (cf.~Definition~\ref{Def:ANN}). Then
	\begin{enumerate}[(i)]
		\item\label{item:ANNsum:comp:left:1} it holds that  $\functionANN(\compANN{\sumANN_{m, n}}{\Phi}) \in C(\R^{\inDimANN(\Phi)}, \R^m) $ and
		\item\label{item:ANNsum:comp:left:2}
 it holds for all  $x \in \R^{\inDimANN(\Phi)}$, $y_1, y_2, \ldots, y_n \in \R^{m}$ with $(\functionANN (\Phi ))(x) = (y_1, y_2, \ldots, y_n)$ 
		that 
		\begin{equation}
		\big( \functionANN(\compANN{\sumANN_{m, n}}{\Phi}) \big)(x) = \smallsum_{k=1}^n y_k
		\end{equation}
	\end{enumerate}
	(cf.~Definition~\ref{Definition:ANNrealization},
	Definition~\ref{Definition:ANNcomposition}, and 
	Definition~\ref{Definition:ANN:sum}).
\end{lemma}
\begin{proof}[Proof of Lemma~\ref{lem:def:ANNsum:comp:left}]
Note that Lemma~\ref{lem:def:ANNsum} ensures that for all $x_1, x_2, \ldots, x_n \in \R^{m}$  it holds that $\functionANN(\sumANN_{m, n}) \in C(\R^{nm}, \R^m)$ and
	\begin{equation}
	\begin{split}
	(\functionANN(\sumANN_{m, n})) (x_1, x_2, \ldots, x_n)  =  \smallsum_{k=1}^n x_k.
	\end{split}
	\end{equation}
	Combining this and, e.g., \cite[item~(v) in Proposition~2.6]{GrohsHornungJentzen2019} establishes items~\eqref{item:ANNsum:comp:left:1}--\eqref{item:ANNsum:comp:left:2}.
	The proof of Lemma~\ref{lem:def:ANNsum:comp:left} is thus completed.
\end{proof}

\begin{lemma}
	\label{lem:def:ANNsum:comp:right}
	Let  $n \in \N$, $a \in C(\R, \R)$, $\Phi \in \ANNs$ (cf.~Definition~\ref{Def:ANN}). Then
	\begin{enumerate}[(i)]
		\item\label{item:ANNsum:comp:right:1} it holds that $\functionANN(\compANN{\Phi}{\sumANN_{\inDimANN(\Phi), n}}) \in C(\R^{n \inDimANN(\Phi)}, \R^{\outDimANN(\Phi)}) $ and
\item\label{item:ANNsum:comp:right:2} it holds for all  $x_1, x_2, \ldots, x_n \in \R^{\inDimANN(\Phi)}$ that  
		\begin{equation}
		\big(\functionANN(\compANN{\Phi}{\sumANN_{\inDimANN(\Phi), n}}) \big)(x_1, x_2, \ldots, x_n) = (\functionANN(\Phi))(\smallsum_{k=1}^n x_k)
		\end{equation}
	\end{enumerate}
	(cf.~Definition~\ref{Definition:ANNrealization},
	Definition~\ref{Definition:ANNcomposition}, and 
	Definition~\ref{Definition:ANN:sum}).
\end{lemma}
\begin{proof}[Proof of Lemma~\ref{lem:def:ANNsum:comp:right}]
Note that Lemma~\ref{lem:def:ANNsum} demonstrates that for all $m \in \N$, $x_1, x_2, \ldots, x_n \in \R^{m}$  it holds that $\functionANN(\sumANN_{m, n}) \in C(\R^{nm}, \R^m)$ and
	\begin{equation}
	\begin{split}
	(\functionANN(\sumANN_{m, n})) (x_1, x_2, \ldots, x_n)  =  \smallsum_{k=1}^n x_k.
	\end{split}
	\end{equation}
	Combining this and, e.g., \cite[item~(v) in Proposition~2.6]{GrohsHornungJentzen2019} establishes items~\eqref{item:ANNsum:comp:right:1}--\eqref{item:ANNsum:comp:right:2}.
	The proof of Lemma~\ref{lem:def:ANNsum:comp:right} is thus completed.
\end{proof}

\begin{definition}
	\label{Def:Transpose}
	Let $m, n \in \N$, $A \in \R^{m \times n}$. Then we denote by $A^* \in \R^{n \times m}$ the transpose of A.
\end{definition}

\begin{definition}
	\label{Definition:ANN:extension}
	Let $m, n \in \N$. Then we denote by $\extensionANN_{m, n} \in (\R^{(nm) \times m} \times \R^{nm})$	the pair given by
	\begin{equation}
	\extensionANN_{m, n} = \matrixANN_{(\idMatrix_m \,\,\,  \idMatrix_m \,\,\, \ldots \,\,\, \idMatrix_m)^* }
	\end{equation}
	(cf.~Definition~\ref{Definition:identityMatrix}, Definition~\ref{Definition:ANN:matrix}, and Definition~\ref{Def:Transpose}).
\end{definition}

\begin{lemma}
	\label{lem:ANN:extension}
	Let $m, n \in \N$. Then
	\begin{enumerate}[(i)]
		\item\label{item:ANN:extension:1} it holds that $\extensionANN_{m, n} \in \ANNs$,
		\item\label{item:ANN:extension:2} it holds that $	\dims (\extensionANN_{m, n}) = (m, nm) \in \N^2$,
		\item\label{item:ANN:extension:3} it holds for all $a \in C(\R, \R)$ that $\functionANN(\extensionANN_{m, n}) \in C(\R^{m}, \R^{nm})$, and
		\item\label{item:ANN:extension:4} it holds for all $a \in C(\R, \R)$, $x \in \R^m$ that 
		\begin{equation}
		(\functionANN(\extensionANN_{m, n})) (x) = (x, x, \ldots, x)
		\end{equation}
	\end{enumerate}
	(cf.~Definition~\ref{Def:ANN}, Definition~\ref{Definition:ANNrealization}, and 
	Definition~\ref{Definition:ANN:extension}).
\end{lemma}
\begin{proof}[Proof of Lemma~\ref{lem:ANN:extension}]
Note that the fact that $\extensionANN_{m, n} \in (\R^{(nm) \times m} \times \R^{nm})$ ensures that $\extensionANN_{m, n} \in \ANNs$ and $\dims (\extensionANN_{m, n}) = (m, nm) \in \N^2$. This establishes items~\eqref{item:ANN:extension:1}--\eqref{item:ANN:extension:2}. Next observe that items~\eqref{item:ANN:matrix:3}--\eqref{item:ANN:matrix:4} in Lemma~\ref{lem:ANN:matrix} prove that for all $a \in C(\R, \R)$, $x  \in \R^m$  it holds that
	$\functionANN(\extensionANN_{m, n}) \in C(\R^{m}, \R^{nm})$ and
	\begin{equation}
    \begin{split}
	(\functionANN(\extensionANN_{m, n})) (x) &= \big(\functionANN \big(  \matrixANN_{(\idMatrix_m \,\,\,  \idMatrix_m \,\,\, \ldots \,\,\, \idMatrix_m)^* } \big) \big)(x)\\
	& = (\idMatrix_m \,\,\,  \idMatrix_m \,\,\, \ldots \,\,\, \idMatrix_m)^{*} x =
	(x, x, \ldots, x)
	\end{split}
	\end{equation}
(cf.~Definition~\ref{Definition:identityMatrix} and Definition~\ref{Definition:ANN:matrix}).
	This establishes items~\eqref{item:ANN:extension:3}--\eqref{item:ANN:extension:4}.
	The proof of Lemma~\ref{lem:ANN:extension} is thus completed.
\end{proof}

\begin{lemma}
	\label{lem:ANN:extension:comp:left}
	Let $n \in \N$, $a \in C(\R, \R)$, $\Phi \in \ANNs$ (cf.~Definition~\ref{Def:ANN}). Then
	\begin{enumerate}[(i)]
		\item\label{item:ANN:extension:comp:left:1} it holds that $\functionANN(\compANN{\extensionANN_{\outDimANN(\Phi), n}}{\Phi}) \in C(\R^{\inDimANN(\Phi)}, \R^{n \outDimANN(\Phi)}) $ and
			\item\label{item:ANN:extension:comp:left:2}
it holds for all $x \in \R^{\inDimANN(\Phi)}$
		that 
		\begin{equation}
		\big( \functionANN(\compANN{\extensionANN_{\outDimANN(\Phi), n}}{\Phi}) \big)(x) = \big((\functionANN (\Phi))(x), (\functionANN (\Phi))(x), \ldots, (\functionANN (\Phi))(x) \big)
		\end{equation}
	\end{enumerate}
	(cf.~Definition~\ref{Definition:ANNrealization},
	Definition~\ref{Definition:ANNcomposition}, and 
	Definition~\ref{Definition:ANN:extension}).
\end{lemma}
\begin{proof}[Proof of Lemma~\ref{lem:ANN:extension:comp:left}]
Note that Lemma~\ref{lem:ANN:extension} ensures that for all $m \in \N$,  $x  \in \R^m$  it holds that
	$\functionANN(\extensionANN_{m, n}) \in C(\R^{m}, \R^{nm})$ and
	\begin{equation}
	\begin{split}
	(\functionANN(\extensionANN_{m, n})) (x)  =
	(x, x, \ldots, x).
	\end{split}
	\end{equation}
	Combining this and, e.g., \cite[item~(v) in Proposition~2.6]{GrohsHornungJentzen2019} establishes items \eqref{item:ANN:extension:comp:left:1}--\eqref{item:ANN:extension:comp:left:2}.
	The proof of Lemma~\ref{lem:ANN:extension:comp:left} is thus completed.
\end{proof}

\begin{lemma}
	\label{lem:ANN:extension:comp:right}
	Let $m, n \in \N$, $a \in C(\R, \R)$, $\Phi \in \{\Psi \in \ANNs \colon \inDimANN(\Psi) = nm\}$ (cf.~Definition~\ref{Def:ANN}). Then
	\begin{enumerate}[(i)]
		\item\label{item:ANN:extension:comp:right:1} it holds that $\functionANN(\compANN{\Phi}{\extensionANN_{m, n}}) \in C(\R^{m}, \R^{\outDimANN(\Phi)}) $
		and
		\item\label{item:ANN:extension:comp:right:2} it holds for all  $x \in \R^{m}$ that  
		\begin{equation}
		\big(\functionANN(\compANN{\Phi}{\extensionANN_{m, n}}) \big)(x) = (\functionANN(\Phi))(x, x, \ldots, x)
		\end{equation}
	\end{enumerate}
	(cf.~Definition~\ref{Definition:ANNrealization},
	Definition~\ref{Definition:ANNcomposition}, and 
	Definition~\ref{Definition:ANN:extension}).
\end{lemma}
\begin{proof}[Proof of Lemma~\ref{lem:ANN:extension:comp:right}]
Observe that Lemma~\ref{lem:ANN:extension} demonstrates that for all  $x  \in \R^m$ it holds that
$\functionANN(\extensionANN_{m, n}) \in C(\R^{m}, \R^{nm})$ and
\begin{equation}
\begin{split}
(\functionANN(\extensionANN_{m, n})) (x)  =
(x, x, \ldots, x).
\end{split}
\end{equation}
	Combining this and, e.g., \cite[item~(v) in Proposition~2.6]{GrohsHornungJentzen2019} establishes items \eqref{item:ANN:extension:comp:right:1}--\eqref{item:ANN:extension:comp:right:2}.
	The proof of Lemma~\ref{lem:ANN:extension:comp:right} is thus completed.
\end{proof}

\begin{definition}[Sums of ANNs with the same length]
	\label{Definition:ANNsum:same}
	Let $n \in \N$, $\Phi_1, \Phi_2, \ldots, \Phi_n \in \ANNs$ satisfy  for all  $k \in \{1, 2, \ldots, n\}$ that
	$\lengthANN(\Phi_k) = \lengthANN(\Phi_1)$,
	$\inDimANN(\Phi_k) = \inDimANN(\Phi_1)$, and $\outDimANN(\Phi_k) = \outDimANN(\Phi_1)$.
	Then we denote by $\oplus_{k \in \{1, 2, \ldots, n\}} \Phi_k$ (we denote by $\Phi_1 \oplus \Phi_2 \oplus \ldots \oplus \Phi_n$)
the tuple given by
\begin{equation}
\oplus_{k \in \{1, 2, \ldots, n\}} \Phi_k = \big( \compANN{\sumANN_{\outDimANN(\Phi_1), n}}{{\compANN{\big[\mathbf{P}_n(\Phi_1,\Phi_2,\dots, \Phi_n)\big]}{\extensionANN_{\inDimANN(\Phi_1), n}}}} \big) \in \ANNs
\end{equation}
(cf.~Definition~\ref{Def:ANN}, Definition~\ref{Definition:ANNcomposition},
Definition~\ref{Definition:simpleParallelization},
Definition~\ref{Definition:ANN:sum}, and
Definition~\ref{Definition:ANN:extension}
).
\end{definition}

\begin{definition}[Dimensions of ANNs]
	\label{Def:ANN:dimensions}
	Let $n \in \N_0$. Then
	we denote by 	$\dimANNlevel_n \colon \ANNs \to \N_0$ the function which satisfies for all  $ L\in\N$, $l_0,l_1,\ldots, l_L \in \N$, 
	$
	\Phi 
	\in  \allowbreak
	( \times_{k = 1}^L\allowbreak(\R^{l_k \times l_{k-1}} \times \R^{l_k}))$ 
	that
	\begin{align}
	\begin{split}
	\dimANNlevel_n (\Phi) =
	\begin{cases}
	l_n &\colon n \leq L \\
	0 &\colon n > L
	\end{cases}
	\end{split}
	\end{align}
	(cf.~Definition~\ref{Def:ANN}).
\end{definition}

\begin{lemma}
	\label{lem:sum:ANNs}
Let $n \in \N$, $\Phi_1, \Phi_2, \ldots, \Phi_n \in \ANNs$ satisfy
 for all  $k \in \{1, 2, \ldots, n\}$ that
	$\lengthANN(\Phi_k) = \lengthANN(\Phi_1)$,
$\inDimANN(\Phi_k) = \inDimANN(\Phi_1)$, and $\outDimANN(\Phi_k) = \outDimANN(\Phi_1)$ (cf.~Definition~\ref{Def:ANN}).
Then
\begin{enumerate}[(i)]
	\item\label{item:sum:ANNs:1} it holds that $	\lengthANN(\oplus_{k \in \{1, 2, \ldots, n\}} \Phi_k) = \lengthANN(\Phi_1)$,
	\item\label{item:sum:ANNs:2} it holds that
	\begin{align*}
	&\dims(\oplus_{k \in \{1, 2, \ldots, n\}} \Phi_k) \numberthis \\
	&= \big(\inDimANN(\Phi_1), \smallsum_{k=1}^n \dimANNlevel_1(\Phi_k), \smallsum_{k = 1}^n \dimANNlevel_2(\Phi_k), \ldots, \smallsum_{k=1}^n \dimANNlevel_{\lengthANN(\Phi_1)-1}(\Phi_k), \outDimANN(\Phi_1)\big),
	\end{align*}
	\item\label{item:sum:ANNs:3} it holds for all $a \in C(\R, \R)$ that $\mathcal{R}_{a}(\oplus_{k \in \{1, 2, \ldots, n\}} \Phi_k) \in C(\R^{\inDimANN(\Phi_1)}, \R^{\outDimANN(\Phi_1)})$, and 
		\item\label{item:sum:ANNs:4} it holds for all $a \in C(\R, \R)$, $x \in \R^{\inDimANN(\Phi_1)}$ that
	\begin{equation}
	\big(\mathcal{R}_{a} (\oplus_{k \in \{1, 2, \ldots, n\}} \Phi_k ) \big) (x) = \sum_{k=1}^n (\mathcal{R}_a(\Phi_k))(x)
	\end{equation}
\end{enumerate}
(cf.~Definition~\ref{Definition:ANNrealization},  Definition~\ref{Definition:ANNsum:same}, and Definition~\ref{Def:ANN:dimensions}).
\end{lemma}
\begin{proof}[Proof of Lemma~\ref{lem:sum:ANNs}]
First, note that, e.g.,  \cite[Lemma~2.18]{GrohsHornungJentzen2019} proves that
\begin{equation}
\label{eq:lem:sum:par}
\begin{split}
&\dims \big( \mathbf{P}_n(\Phi_1,\Phi_2,\dots, \Phi_n) \big)\\
&= \big( \smallsum_{k=1}^n \dimANNlevel_0(\Phi_k), \smallsum_{k = 1}^n \dimANNlevel_1(\Phi_k), \ldots,
\smallsum_{k=1}^n \dimANNlevel_{\lengthANN(\Phi_1)-1}(\Phi_k), \smallsum_{k=1}^n \dimANNlevel_{\lengthANN(\Phi_1)}(\Phi_k)\big)\\
& = \big(n \inDimANN(\Phi_1), \smallsum_{k=1}^n \dimANNlevel_1(\Phi_k), \smallsum_{k = 1}^n \dimANNlevel_2(\Phi_k), \ldots, \smallsum_{k=1}^n \dimANNlevel_{\lengthANN(\Phi_1)-1}(\Phi_k), n \outDimANN(\Phi_1)\big)
\end{split}
\end{equation}
(cf.~Definition~\ref{Definition:simpleParallelization}).
Moreover, observe that item~\eqref{item:ANNsum:2} in Lemma~\ref{lem:def:ANNsum} ensures that
\begin{equation}
\label{eq:lem:sum:dims}
\dims\big(\sumANN_{\outDimANN(\Phi_1), n} \big) = (n\outDimANN(\Phi_1), \outDimANN(\Phi_1))
\end{equation}
(cf.~Definition~\ref{Definition:ANN:sum}).
This, \eqref{eq:lem:sum:par}, and, e.g.,  \cite[item~(i) in Proposition~2.6]{GrohsHornungJentzen2019}  demonstrate that
\begin{equation}
\label{eq:lem:sum:comp}
\begin{split}
&\dims \big(\compANN{\sumANN_{\outDimANN(\Phi_1), n}}{\big[\mathbf{P}_n(\Phi_1,\Phi_2,\dots, \Phi_n)\big]} \big)\\
& = \big(n \inDimANN(\Phi_1), \smallsum_{k=1}^n \dimANNlevel_1(\Phi_k), \smallsum_{k = 1}^n \dimANNlevel_2(\Phi_k), \ldots, \smallsum_{k=1}^n \dimANNlevel_{\lengthANN(\Phi_1)-1}(\Phi_k),  \outDimANN(\Phi_1)\big).
\end{split}
\end{equation}
Next note that item~\eqref{item:ANN:extension:2} in Lemma~\ref{lem:ANN:extension} assures that
\begin{equation}
\dims \big( \extensionANN_{\inDimANN(\Phi_1), n}\big) = (\inDimANN(\Phi_1), n \inDimANN(\Phi_1))
\end{equation}
(cf.~Definition~\ref{Definition:ANN:extension}).
Combining this, \eqref{eq:lem:sum:comp}, and,  e.g.,  \cite[item~(i) in Proposition~2.6]{GrohsHornungJentzen2019} proves that
\begin{align}
\begin{split}
&\dims(\oplus_{k \in \{1, 2, \ldots, n\}} \Phi_k) \\
& = \dims \big( \compANN{\sumANN_{\outDimANN(\Phi_1), n}}{{\compANN{\big[\mathbf{P}_n(\Phi_1,\Phi_2,\dots, \Phi_n)\big]}{\extensionANN_{\inDimANN(\Phi_1), n}}}}\big)\\
&= \big(\inDimANN(\Phi_1), \smallsum_{k=1}^n \dimANNlevel_1(\Phi_k), \smallsum_{k = 1}^n \dimANNlevel_2(\Phi_k), \ldots, \smallsum_{k=1}^n \dimANNlevel_{\lengthANN(\Phi_1)-1}(\Phi_k), \outDimANN(\Phi_1)\big).
\end{split}
\end{align}
This establishes items~\eqref{item:sum:ANNs:1}--\eqref{item:sum:ANNs:2}.
Next observe that  Lemma~\ref{lem:ANN:extension:comp:right} and \eqref{eq:lem:sum:par} ensure that
for all $a \in C(\R, \R)$, $x \in \R^{\inDimANN(\Phi_1)}$ it holds that  $\functionANN(\compANN{[\mathbf{P}_n(\Phi_1,\Phi_2,\dots, \Phi_n)]}{\extensionANN_{\inDimANN(\Phi_1), n}}) \in C(\R^{\inDimANN(\Phi_1)}, \R^{n \outDimANN(\Phi_1)}) $ and
\begin{equation}
\begin{split}
&\big(\functionANN\big(\compANN{[\mathbf{P}_n(\Phi_1,\Phi_2,\dots, \Phi_n)]}{\extensionANN_{\inDimANN(\Phi_1), n}}\big)\big) (x)\\
&= \big(\functionANN\big(\mathbf{P}_n(\Phi_1,\Phi_2,\dots, \Phi_n)\big)\big)(x, x, \ldots, x).
\end{split}
\end{equation}
Combining this with, e.g.,  \cite[item~(ii) in Proposition~2.19]{GrohsHornungJentzen2019} proves that
for all $a \in C(\R, \R)$, $x \in \R^{\inDimANN(\Phi_1)}$ it holds that  
\begin{equation}
\begin{split}
&\big(\functionANN\big(\compANN{[\mathbf{P}_n(\Phi_1,\Phi_2,\dots, \Phi_n)]}{\extensionANN_{\inDimANN(\Phi_1), n}}\big)\big) (x)\\
&= \big( (\functionANN (\Phi_1))(x), (\functionANN (\Phi_2))(x), \ldots, (\functionANN (\Phi_n))(x) \big) \in \R^{n \outDimANN(\Phi_1)}.
\end{split}
\end{equation}
Lemma~\ref{lem:def:ANNsum:comp:left}, \eqref{eq:lem:sum:dims}, and,  e.g.,  \cite[Lemma~2.8]{GrohsHornungJentzen2019} therefore demonstrate that for all $a \in C(\R, \R)$, $x \in \R^{\inDimANN(\Phi_1)}$ it holds that  $\mathcal{R}_{a}(\oplus_{k \in \{1, 2, \ldots, n\}} \Phi_k) \in C(\R^{\inDimANN(\Phi_1)}, \R^{\outDimANN(\Phi_1)})$ and 
\begin{equation}
\begin{split}
&\big(\mathcal{R}_{a} (\oplus_{k \in \{1, 2, \ldots, n\}} \Phi_k ) \big) (x)\\
& = \big(\functionANN\big(\compANN{\sumANN_{\outDimANN(\Phi_1), n}}{{\compANN{[\mathbf{P}_n(\Phi_1,\Phi_2,\dots, \Phi_n)]}{\extensionANN_{\inDimANN(\Phi_1), n}}}}\big)\big) (x) = \sum_{k=1}^n( \mathcal{R}_a(\Phi_k)) (x).
\end{split}
\end{equation}
This establishes items~\eqref{item:sum:ANNs:3}--\eqref{item:sum:ANNs:4}.
The proof of Lemma~\ref{lem:sum:ANNs} is thus completed.
\end{proof}

\subsection{ANN representation results}

\begin{lemma}
	\label{lem:sum:ANN}
	Let 
	$ n \in \N $, $h_1, h_2, \ldots, h_n \in \R$,
	$
	 \Phi_1, \Phi_2, \ldots, \Phi_n \in \ANNs$ satisfy  that
	$\dims(\Phi_1) = \dims(\Phi_2) = \ldots = \dims(\Phi_n)$,  let $A_k \in \R^{\inDimANN(\Phi_1) \times (n \inDimANN(\Phi_1))}$, $k \in \{1, 2, \ldots, n\}$, satisfy  for all $k \in \{1, 2, \ldots, n\}$, $x = (x_i)_{i \in \{1, 2, \ldots, n\}} \in \R^{n \inDimANN(\phi_1)}$ that $A_k x = x_k$, and let $\Psi \in \ANNs$ satisfy that
	\begin{equation}
	\Psi = 	\oplus_{k \in \{1, 2, \ldots, n\}} ( h_k  \circledast ( \compANN{\Phi_k}{\matrixANN_{A_k}}))
	\end{equation}
	(cf.~Definition~\ref{Def:ANN}, 
		Definition~\ref{Definition:ANN:matrix},
	Definition~\ref{Definition:ANNscalar}, and
	Definition~\ref{Definition:ANNsum:same}).
Then 
\begin{enumerate}[(i)]
	\item\label{item:sum:ANN:1} it holds that
	\begin{equation}
	\dims (\Psi) = (n \inDimANN(\Phi_1), n\dimANNlevel_1(\Phi_1), n\dimANNlevel_2(\Phi_1), \ldots, n\dimANNlevel_{\lengthANN(\Phi_1)-1}(\Phi_1), \outDimANN (\Phi_1)),
	\end{equation}
	\item\label{item:sum:ANN:2} it holds that $\paramANN(\Psi) \leq n^2 \paramANN(\Phi_1)$,
		\item\label{item:sum:ANN:3} it holds for all 	$ a \in C(\R, \R)$ that  $\functionANN(\Psi) \in C(\R^{n \inDimANN(\Phi_1)}, \R^{\outDimANN(\Phi_1)})$, and
	\item\label{item:sum:ANN:4} it holds for all 	$ a \in C(\R, \R)$, $x = (x_k)_{k \in \{1, 2, \ldots, n\}} \in \R^{n \inDimANN(\Phi_1)}$ that  
	\begin{equation}
	(\functionANN (\Psi))(x) = \sum_{k=1}^n h_k (\functionANN (\Phi_k))(x_k)
	\end{equation}
\end{enumerate}	
(cf.~Definition~\ref{Definition:ANNrealization} and Definition~\ref{Def:ANN:dimensions}).
\end{lemma}
\begin{proof}[Proof of Lemma~\ref{lem:sum:ANN}]
First, note that item~\eqref{item:ANN:matrix:2} in Lemma~\ref{lem:ANN:matrix} ensures for all 
$k \in \{1, 2, \ldots, n\}$ that
 \begin{equation}
\dims (\matrixANN_{A_k}) = (n \inDimANN(\Phi_1),  \inDimANN(\Phi_1)) \in \N^2.
\end{equation}
This and, e.g., \cite[item~(i) in Proposition~2.6]{GrohsHornungJentzen2019} prove for all 
$k \in \{1, 2, \ldots, n\}$ that
\begin{eqnarray}
\dims ( \compANN{\Phi_k}{\matrixANN_{A_k}}) = (n \inDimANN(\Phi_1), \dimANNlevel_1(\Phi_k), \dimANNlevel_2(\Phi_k), \ldots, \dimANNlevel_{\lengthANN(\Phi_k)}(\Phi_k)).
\end{eqnarray}
Item~\eqref{item:ANN:scalar:1} in Lemma~\ref{lem:ANNscalar} therefore demonstrates for all $k \in \{1, 2, \ldots, n\}$ that
\begin{equation}
\label{eq:lem:resultsum:dims}
\begin{split}
\dims ( h_k  \circledast ( \compANN{\Phi_k}{\matrixANN_{A_k}})) &= \dims (\compANN{\Phi_k}{\matrixANN_{A_k}}) \\
&= (n \inDimANN(\Phi_1), \dimANNlevel_1(\Phi_k), \dimANNlevel_2(\Phi_k), \ldots, \dimANNlevel_{\lengthANN(\Phi_k)-1}(\Phi_k), \outDimANN (\Phi_k))\\
& = (n \inDimANN(\Phi_1), \dimANNlevel_1(\Phi_1), \dimANNlevel_2(\Phi_1), \ldots, \dimANNlevel_{\lengthANN(\Phi_1)-1}(\Phi_1), \outDimANN (\Phi_1)).
\end{split}
\end{equation}
Combining this with 
item~\eqref{item:sum:ANNs:2} in  Lemma~\ref{lem:sum:ANNs} 
ensures that
\begin{equation}
\begin{split}
\dims (\Psi) &= \dims \big(\! \oplus_{k \in \{1, 2, \ldots, n\}} ( h_k  \circledast ( \compANN{\Phi_k}{\matrixANN_{A_k}}))\big)\\
& =  (n \inDimANN(\Phi_1), n\dimANNlevel_1(\Phi_1), n\dimANNlevel_2(\Phi_1), \ldots, n\dimANNlevel_{\lengthANN(\Phi_1)-1}(\Phi_1), \outDimANN (\Phi_1)).
\end{split}
\end{equation}
This establishes item~\eqref{item:sum:ANN:1}.
Hence, we obtain  that
\begin{equation}
\paramANN(\Psi)  \leq n^2 \paramANN (\Phi_1).
\end{equation}
This establishes item~\eqref{item:sum:ANN:2}.
Moreover, observe that items~\eqref{item:ANN:matrix:comp:3}--\eqref{item:ANN:matrix:comp:4} in Lemma~\ref{lem:ANN:matrix:comp} assure
 for all $k \in \{1, 2, \ldots, n\}$, $a \in C(\R, \R)$, $x = (x_i)_{i \in \{1, 2, \ldots, n\}} \in \R^{n \inDimANN(\Phi_1)}$  that 
 $\functionANN ( \compANN{\Phi_k}{\matrixANN_{A_k}}) \in C(\R^{n \inDimANN(\Phi_1)}, \R^{\outDimANN(\Phi_k)})$ and
 \begin{equation}
 \big(\functionANN ( \compANN{\Phi_k}{\matrixANN_{A_k}})\big) (x) = (\functionANN (\Phi))(A_k x) = (\functionANN (\Phi)) (x_k).
 \end{equation}
Combining this with items~\eqref{item:ANN:scalar:2}--\eqref{item:ANN:scalar:3} in Lemma~\ref{lem:ANNscalar} proves  for all $k \in \{1, 2, \ldots, n\}$, $a \in C(\R, \R)$, $x = (x_i)_{i \in \{1, 2, \ldots, n\}} \in \R^{n \inDimANN(\Phi_1)}$  that  
 $\functionANN ( h_k  \circledast ( \compANN{\Phi_k}{\matrixANN_{A_k}})) \in C(\R^{n \inDimANN(\Phi_1)}, \R^{\outDimANN(\Phi_1)})$ and
\begin{equation}
\big(\functionANN ( h_k  \circledast ( \compANN{\Phi_k}{\matrixANN_{A_k}}))\big) (x) = h_k (\functionANN (\Phi)) (x_k).
\end{equation}
Items~\eqref{item:sum:ANNs:3}--\eqref{item:sum:ANNs:4} in Lemma~\ref{lem:sum:ANNs} and \eqref{eq:lem:resultsum:dims} hence ensure for all  $a \in C(\R, \R)$, $x = (x_i)_{i \in \{1, 2, \ldots, n\}} \in \R^{n \inDimANN(\Phi_1)}$  that 
$\functionANN(\Psi) \in C(\R^{n \inDimANN(\Phi_1)}, \R^{\outDimANN(\Phi_1)})$  and
\begin{equation}
\label{eq:sum:ann}
\begin{split}
(\functionANN (\Psi))(x) & = \big( \functionANN \big(  \!	\oplus_{k \in \{1, 2, \ldots, n\}} ( h_k  \circledast ( \compANN{\Phi_k}{\matrixANN_{A_k}})) \big) \big)(x)\\
&=  \sum_{k = 1}^n \big(\functionANN ( h_k  \circledast ( \compANN{\Phi_k}{\matrixANN_{A_k}}))\big) (x) =  \sum_{k=1}^n h_k (\functionANN (\Phi_k))(x_k).
\end{split}
\end{equation}
This establishes items~\eqref{item:sum:ANN:3}--\eqref{item:sum:ANN:4}.
The proof of 	Lemma~\ref{lem:sum:ANN} is thus completed.
\end{proof}

\begin{lemma}
\label{lem:Composition_Sum}
Let $a\in C(\R,\R)$, $L_1, L_2\in\N$, $\mathbb{I}, \Phi_1,\Phi_2\in\ANNs$, $d,\hiddenDimId, l_{1,0},l_{1,1},\dots,\allowbreak l_{1,L_1},l_{2,0},l_{2,1},\allowbreak\dots,\allowbreak l_{2,L_2}\in\N$	
satisfy  for all $k\in\{1,2\}$, $x\in\R^{d}$ that $2\le\hiddenDimId\le 2d$,
$l_{2,L_2-1}\le l_{1,L_1-1}+\hiddenDimId$,
$\dims(\mathbb{I}) = (d,\hiddenDimId,d)$, $(\functionANN(\mathbb{I}))(x)=x$,  
$\inDimANN(\Phi_k)=\outDimANN(\Phi_k)=d$,
and
$\dims(\Phi_k)=(l_{k,0},l_{k,1},\dots, l_{k,L_k})$
(cf.\ Definition~\ref{Def:ANN} and Definition~\ref{Definition:ANNrealization}).
Then there exists $\Psi\in \ANNs$ such that
\begin{enumerate}[(i)]
\item \label{Composition_Sum:Realization}
it holds that $\functionANN(\Psi)\in C(\R^d,\R^d)$,
\item \label{Composition_Sum:Function}
it holds for all $x\in\R^d$  that
\begin{equation}
(\functionANN(\Psi))(x)=(\functionANN(\Phi_2))(x)+\big((\functionANN(\Phi_1))\circ (\functionANN(\Phi_2))\big)(x),
\end{equation}
\item \label{Composition_Sum:Dims}
it holds that
\begin{equation}
\dimANNlevel_{\lengthANN(\Psi) -1} (\Psi) \leq l_{1, L_1 -1} + \mathfrak{i},
\end{equation}
and
\item \label{Composition_Sum:Params}
it holds that $\paramANN(\Psi)
\le   
\paramANN(\Phi_2)+\big[\tfrac{1}{2}\paramANN(\mathbb{I})+\paramANN(\Phi_1)\big]^{\!2}$
\end{enumerate}
(cf.~Definition~\ref{Definition:ANNcomposition} and Definition~\ref{Def:ANN:dimensions}).
\end{lemma}
\begin{proof}[Proof of Lemma~\ref{lem:Composition_Sum}]
To prove items~\eqref{Composition_Sum:Realization}--\eqref{Composition_Sum:Params} we distinguish between the case $L_1=1$ and the case $L_1 \in \N \cap [2, \infty)$. We first prove items~\eqref{Composition_Sum:Realization}--\eqref{Composition_Sum:Params} in the case $L_1=1$.  Note that, e.g., \cite[Proposition~2.30]{GrohsHornungJentzen2019} (with $a=a$, $d=d$, $\mathfrak{L} = L_2$, $(\ell_0, \ell_1, \ldots, \ell_{\mathfrak{L}}) = (l_{2,0},l_{2,1}, \ldots, l_{2,L_2})$, $\psi = \Phi_2$, $\phi_n = \Phi_1$ for $n \in \N_0$ in the notation of \cite[Proposition~2.30]{GrohsHornungJentzen2019}) implies that there exists $\Psi \in \ANNs$ such that
\begin{enumerate}[(I)]
	\item \label{item:sum:comp:1} it holds that $\functionANN(\Psi) \in C(\R^d, \R^d)$,
	\item \label{item:sum:comp:2} it holds for all $x \in \R^d$ that 
	\begin{equation}
	(\functionANN(\Psi))(x)=(\functionANN(\Phi_2))(x)+\big((\functionANN(\Phi_1))\circ (\functionANN(\Phi_2))\big)(x),
	\end{equation}
	 and
	\item \label{item:sum:comp:3}  it holds that $\dims(\Psi) = \dims(\Phi_2)$.
\end{enumerate}
The hypothesis that $l_{2,L_2-1}\le l_{1,L_1-1}+\hiddenDimId$ hence ensures that 
\begin{equation}
\label{eq:prop:sum:comp:par}
\dimANNlevel_{\lengthANN(\Psi) -1} (\Psi) = \dimANNlevel_{\lengthANN(\Phi_2) -1} (\Phi_2) = l_{2,L_2-1}   \leq l_{1, L_1 -1} + \mathfrak{i}.
\end{equation}
Moreover, note that \eqref{item:sum:comp:3} assures that
\begin{equation}
\paramANN(\Psi) = \paramANN(\Phi_2)
\le   
\paramANN(\Phi_2)+\big[\tfrac{1}{2}\paramANN(\mathbb{I})+\paramANN(\Phi_1)\big]^{\!2}.
\end{equation} 
Combining this with \eqref{item:sum:comp:1} and \eqref{eq:prop:sum:comp:par} establishes items~\eqref{Composition_Sum:Realization}--\eqref{Composition_Sum:Params} in the case $L_1=1$. We now prove items~\eqref{Composition_Sum:Realization}--\eqref{Composition_Sum:Params} in the case $L_1 \in \N \cap [2, \infty)$. Observe that, e.g., \cite[Proposition~2.28]{GrohsHornungJentzen2019} (with $a=a$,  $L_1 = L_1$, $L_2 = L_2$, $\mathbb{I} = \mathbb{I}$, $\Phi_1 = \Phi_1$, $\Phi_2 = \Phi_2$, 
$d=d$, $\mathfrak{i} = \mathfrak{i}$, $(l_{1, 0}, l_{1, 1}, \ldots, l_{1, L_1}) = (l_{1, 0}, l_{1, 1}, \ldots, l_{1, L_1})$, $(l_{2, 0}, l_{2, 1}, \ldots, l_{2, L_2}) = (l_{2, 0}, l_{2, 1}, \ldots, l_{2, L_2})$   in the notation of \cite[Proposition~2.28]{GrohsHornungJentzen2019}) proves that there exists $\Psi\in \ANNs$ such that
\begin{enumerate}[(a)]
\item it holds that $\functionANN(\Psi) \in C(\R^d, \R^d)$,
\item \label{CompositionSum:Realization}
it holds for all $x\in\R^d$  that
\begin{equation}
(\functionANN(\Psi))(x)=(\functionANN(\Phi_2))(x)+\big((\functionANN(\Phi_1))\circ (\functionANN(\Phi_2))\big)(x),
\end{equation}
\item \label{CompositionSum:Dims}
it holds that
\begin{equation}
\dims(\Psi)=(l_{2,0},l_{2,1},\dots, l_{2,L_2-1},l_{1,1}+\hiddenDimId,l_{1,2}+\hiddenDimId,\dots,l_{1,L_1-1}+\hiddenDimId, l_{1, L_1}),
\end{equation} 
and
\item \label{CompositionSum:Params}
it holds that $\paramANN(\Psi)
\le   
\paramANN(\Phi_2)+\big[\tfrac{1}{2}\paramANN(\mathbb{I})+\paramANN(\Phi_1)\big]^{\!2}$.
\end{enumerate}
This establishes items~\eqref{Composition_Sum:Realization}--\eqref{Composition_Sum:Params} in the case  $L_1 \in \N \cap [2, \infty)$. The proof of Lemma~\ref{lem:Composition_Sum} is thus completed.
\end{proof}

\section[Kolmogorov partial differential equations (PDEs)]{Kolmogorov partial differential equations (PDEs)}
\label{sec:Kolmogorov}

In this section we
 establish in Theorem~\ref{thm:dnn:kolmogorov} below
the existence of  DNNs which approximate  solutions of suitable Kolmogorov PDEs without the curse of dimensionality.
Moreover, in Corollary~\ref{cor:laplacian:lebesgue} below we specialize Theorem~\ref{thm:dnn:kolmogorov} to the case where 
  for every $ d \in \N $ we have that the probability measure $ \nu_d $ appearing in Theorem~\ref{thm:dnn:kolmogorov}  is the uniform distribution on the $ d $-dimensional unit cube $ [0,1]^d $. In addition,  in Corollary~\ref{cor:dnn:kolmogorov} below we specialize Theorem~\ref{thm:dnn:kolmogorov}, roughly speaking,  to the case where the constants 
  	$\kappa \in (0, \infty)$, $\ExponError, \ExponDim_1, \ExponDim_2, \ldots, \ExponDim_6 \in [0, \infty) $, which we use  to specify the regularity hypotheses in Theorem~\ref{thm:dnn:kolmogorov},   coincide. Corollary~\ref{cor:dnn:kolmogorov}  follows  immediately from Theorem~\ref{thm:dnn:kolmogorov} and  is a slight generalization of~\cite[Theorem~6.3]{JentzenSalimovaWelti2018} and \cite[Theorem~1.1]{JentzenSalimovaWelti2018}, respectively.
In our proof of Theorem~\ref{thm:dnn:kolmogorov} we 
employ the DNN representation results 
 in Lemmas~\ref{lem:sum:ANN}--\ref{lem:Composition_Sum} from Section~\ref{sec:calculus}  above
as well as
 essentially well-known error estimates for the Monte Carlo Euler method which we establish in Proposition~\ref{prop:monte_carlo_euler} in this section below. The proof Proposition~\ref{prop:monte_carlo_euler}, in turn,  employs  the elementary error estimate results  in Lemmas~\ref{lem:con_monte_carlo_euler}--\ref{lem:error:monte_carlo} below.

\subsection{Error analysis for the Monte Carlo Euler method}

\begin{lemma}[Weak perturbation error]
	\label{lem:con_monte_carlo_euler}
	Let $d, m \in \N$, 
	$\xi \in \R^d$,
	$T \in (0, \infty)$, 
	$L_0, L_1, l \in [0, \infty)$,
	$h \in (0, T]$,
	$B \in \R^{d \times m}$,
	let 
	$
	\left\| \cdot \right\| \colon \R^d \to [0,\infty)
	$
	be the $ d $-dimensional Euclidean norm, let $(\Omega, \mathcal{F}, \P)$ be a probability space, let $W \colon [0, T] \times \Omega \to \R^m$ be a standard Brownian motion, let $f_0 \colon \R^d \to \R$ and $f_1 \colon \R^d \to \R^d$ be  functions, let $\chi \colon [0, T] \to [0, T]$ be a function, assume for all $t  \in [0, T]$, $x, y \in \R^d$ that
	\begin{gather}
	\label{eq:lem:f_0:semi}
	|f_0(x) - f_0(y)| \leq L_0 \!\left( 1+ \int_0^1 [r \|x\| + (1-r) \|y\| ]^l \, dr \right)\! \|x-y\|,\\
	\label{eq:lem:f_1_lipschitz}
	\|f_1(x) - f_1(y) \| \leq L_1 \|x-y\|, 
	\end{gather}
	and $\chi(t) = \max(\{0, h, 2h, \ldots\} \cap [0, t])$, 
	and 
	let $ X, Y \colon [0,T] \times \Omega \to \R^d $ be stochastic processes with 
	continuous sample paths which satisfy
	for all  $ t \in [0,T] $ that 
	$
	Y_t = \xi + \int_0^t f_1\big( Y_{ \chi( s ) } \big) \, ds + B W_t
	$
	and
	\begin{equation}
	X_t =  \xi + \int_0^t f_1( X_s ) \, ds + B W_t
	.
	\end{equation}
	Then it holds that
	\begin{align*}
	&\big|\E[f_0(X_T)] - \E[f_0(Y_T)] \big| \leq (h/T)^{\nicefrac{1}{2}}  e^{(l+3+2L_1+[l L_1 + 2 L_1 +2]T)}  \max\{1, L_0\} \numberthis \\
	& \cdot \Big[ \|\xi\| + 2  + \max\{1, \|f_1(0)\|\} \max\{1, T\} + \sqrt{(2\max\{l, 1\} -1) \operatorname{Trace}(B^*B) T} \Big]^{1+l} .
	\end{align*}
\end{lemma}
\begin{proof}[Proof of Lemma~\ref{lem:con_monte_carlo_euler}]
	
	First, note that \eqref{eq:lem:f_1_lipschitz} proves that for all $x \in \R^d$ it holds that
	\begin{equation}
	\|f_1(x) \| \leq \|f_1(x) - f_1(0) \| + \|f_1(0)\| \leq L_1 \|x\| + \|f_1(0)\|.
	\end{equation}
	This, \eqref{eq:lem:f_0:semi}, \eqref{eq:lem:f_1_lipschitz}, and, e.g.,  \cite[Proposition~4.6]{JentzenSalimovaWelti2018} (with $d = d$, $m = m$, $\xi = \xi$, $T = T$, $c  = L_1$, $C = \|f_1(0)\|$, $\varepsilon_0 = 0$, $\varepsilon_1 = 0$, $\varepsilon_2 = 0$,  $\varsigma_0 = 0$, $\varsigma_1 = 0$, $\varsigma_2 = 0$, $L_0 = L_0$, $L_1 = L_1$, $l = l$, $h = h$, $B = B$, $p = 2$, $q = 2$, $\left\| \cdot \right\| = \left\| \cdot \right\|$, $(\Omega, \mathcal{F}, \P) = (\Omega, \mathcal{F}, \P)$, $W = W$, $\phi_0 = f_0$, $f_1 = f_1$, $\phi_2 = (\R^d \ni x \mapsto x \in \R^d)$, $\chi = \chi$, $f_0 = f_0$, $\phi_1 = \phi_1$, $\varpi_r = 
	(
	\E[ 
	\| B W_T \|^r])^{ \nicefrac{1}{r} }$, $X = X$, $Y = Y$ for $r \in (0, \infty)$ in the notation of \cite[Proposition~4.6]{JentzenSalimovaWelti2018}) establish that
	\begin{equation}
	\begin{split}
	&\big|\E[f_0(X_T)] - \E[f_0(Y_T)] \big| \leq (h/T)^{\nicefrac{1}{2}}  e^{(l+3+2L_1+[l L_1 + L_1 + L_1 +2]T)} \max\{1, L_0\} \\
	& \cdot \Big[ \|\xi\| + 2 + \max\{1, \|f_1(0)\|\} \max\{1, T\} + \big(
	\E\big[ 
	\| B W_T \|^{\max\{2, 2l\}}
	\big]
	\big)^{ \nicefrac{1}{\max\{2, 2l\}} }    \Big]^{1+l}.
	\end{split}
	\end{equation}
	Combining this with, e.g., \cite[Lemma~4.2]{JentzenSalimovaWelti2018} (with $d = d$, $m = m$, $T = T$, $p = \max\{2, 2l\}$, $B = B$, $\left\| \cdot \right\| = \left\| \cdot \right\|$, $(\Omega, \mathcal{F}, \P) = (\Omega, \mathcal{F}, \P)$, $W = W$ in the notation of \cite[Lemma~4.2]{JentzenSalimovaWelti2018}) ensures that 
	\begin{align*}
	&\big|\E[f_0(X_T)] - \E[f_0(Y_T)] \big| \leq (h/T)^{\nicefrac{1}{2}}  e^{(l+3+2L_1+[l L_1 + 2 L_1 +2]T)}  \max\{1, L_0\} \numberthis \\
	& \cdot \Big[ \|\xi\| + 2  + \max\{1, \|f_1(0)\|\} \max\{1, T\} + \sqrt{(2\max\{l, 1\} -1) \operatorname{Trace}(B^*B) T} \Big]^{1+l} .
	\end{align*}
	The proof of Lemma~\ref{lem:con_monte_carlo_euler} is thus completed.
\end{proof}

\begin{lemma}
	\label{lem:integral:weak:euler}
	Let $d, m \in \N$, $T, \kappa \in (0, \infty)$,
	$\theta, \ExponDim_0, \ExponDim_1 \in [0, \infty)$,	
	$h \in (0, T]$, $B \in \R^{d \times m}$, $p \in [1, \infty)$,
	let $\nu  \colon \mathcal{B}(\R^d) \to [0,1] $ be a probability measure on $\R^d$, 
	let 
	$
	\left\| \cdot \right\| \colon \R^d \to [0,\infty)
	$
	be the $ d $-dimensional Euclidean norm, let $(\Omega, \mathcal{F}, \P)$ be a probability space, let $W \colon [0, T] \times \Omega \to \R^m$ be a standard Brownian motion, let $f_0 \colon \R^d \to \R$ and $f_1 \colon \R^d \to \R^d$ be  functions, let $\chi \colon [0, T] \to [0, T]$ be a function, assume for all $t  \in [0, T]$, $x, y \in \R^d$ that
	\setlength{\jot}{8pt}
	\begin{gather}
	\label{eq:lem:ass:f_0}
	|f_0(x) - f_0(y)| \leq \kappa d^{\ExponDim_0} (1 + \|x\|^{\theta} + \|y\|^{\theta}) \|x-y\|,\\ 
	\label{eq:lem:ass:f}
	\|f_1(x) - f_1(y)\| \leq \kappa \|x -y\|,
	 \qquad
	\operatorname{Trace}(B^* B)  \leq \kappa d^{2 \ExponDim_1},\\
	\label{eq:lem:ass:linear}
	\|f_1(0)\| \leq \kappa d^{\ExponDim_1}, \qquad 	\left[	\int_{\R^d} \|z\|^{p(1+\theta)} \, \nu (dz) \right]^{\nicefrac{1}{(p (1+\theta))}}  \leq \kappa d^{\ExponDim_1}, 
	\end{gather}
	and $\chi(t) = \max(\{0, h, 2h, \ldots\} \cap [0, t])$, 
	and 
	let $ X^x\colon [0,T] \times \Omega \to \R^d $, $x \in \R^d$, and $Y^x \colon [0,T] \times \Omega \to \R^d $, $x \in \R^d$, be stochastic processes with 
	continuous sample paths which satisfy
	for all $x \in \R^d$,  $ t \in [0,T] $ that 
	$
	Y_t^x = x + \int_0^t f_1\big( Y^x_{ \chi( s ) } \big) \, ds + B W_t
	$
	and
	\begin{equation}
	X^x_t =  x + \int_0^t f_1( X^x_s ) \, ds + B W_t
	.
	\end{equation}
	Then it holds that
	\begin{align}
	\begin{split}
	&\left[ \int_{\R^d} \big|\E[f_0(X^x_T)] - \E[f_0(Y^x_T)]  \big|^p \, \nu(dx) \right]^{\nicefrac{1}{p}} \leq 2^{4\theta +5} | \!\max\{1, T\} |^{\theta +1}\\
	& \qquad \cdot |\!\max\{ \kappa, \theta, 1 \}|^{\theta +3} e^{(6\max\{ \kappa, \theta, 1 \}+5|\!\max\{ \kappa, \theta, 1 \}|^2 T)}  d^{\ExponDim_0 + \ExponDim_1(\theta+1)} (h/T)^{\nicefrac{1}{2}}.
	\end{split}
	\end{align}
\end{lemma}
\begin{proof}[Proof of Lemma~\ref{lem:integral:weak:euler}]
	Throughout this proof let 
	$ \iota = \max\{ \kappa, \theta, 1 \} $.
	Note that \eqref{eq:lem:ass:f_0} proves that for all  $x, y \in \R^d$ it holds that
	\begin{equation}
	\begin{split}
	&|f_0(x) - f_0(y)|  \leq \kappa d^{\ExponDim_0} (1 +  \|x\|^{\theta} + \|y \|^{\theta} )\|x-y\| \\
	& \leq \kappa d^{\ExponDim_0} \! \left( 1 + 2 (\theta +1) \int_0^1 \big[r \|x\|+ (1-r) \|y\|\big]^{\theta} \, dr \right) \! \|x-y\|\\
	& \leq 2 \kappa (\theta +1 ) d^{\ExponDim_0} \! \left( 1 +  \int_0^1 \big[r \|x\| + (1-r) \|y\|\big]^{\theta} \, dr \right) \! \|x-y\|.
	\end{split}
	\end{equation}
	Lemma~\ref{lem:con_monte_carlo_euler} (with $d = d$, $m = m$, $\xi = x$, $T =T$,  $L_0 = 2\kappa (\theta +1) d^{\ExponDim_0}$, $L_1 = \kappa$, $l = \theta$, $h = h$, 
	$B = B$, $\left\| \cdot \right\| = \left\| \cdot \right\|$, $(\Omega, \mathcal{F}, \P) = (\Omega, \mathcal{F}, \P)$, $W = W$, $f_0 = f_0$, $f_1 = f_1$, $\chi = \chi$, $X = X^x$, $Y = Y^x$ for $x \in \R^d$ in the notation of Lemma~\ref{lem:con_monte_carlo_euler}), \eqref{eq:lem:ass:linear},
	and \eqref{eq:lem:ass:f} hence ensure that
	for all $x \in \R^d$ it holds that
	\begin{align*}
	&\big|\E[f_0(X^x_T)] - \E[f_0(Y^x_T)] \big| \leq (h/T)^{\nicefrac{1}{2}}  e^{(\theta+3+2\kappa+[\theta\kappa + 2 \kappa +2]T)}  \max\{1, 2 \kappa (\theta +1) d^{\ExponDim_0}\}\\ & \cdot \Big[ \|x\| + 2  
	+ \max\{1, \|f_1(0)\|\} \max\{1, T\} + \sqrt{(2\max\{\theta, 1\} -1) \operatorname{Trace}(B^*B) T} \Big]^{1+\theta}\\
	&  \leq (h/T)^{\nicefrac{1}{2}}  e^{(\theta+3+2\kappa+[\theta\kappa + 2 \kappa +2]T)}  \max\{1, 2 \kappa (\theta +1) d^{\ExponDim_0}\} \\
	& \cdot \Big[ \|x\| + 2 + \max\{1,  \kappa d^{\ExponDim_1} \} \max\{1, T\}  + \sqrt{(2\max\{\theta, 1\} -1) \kappa d^{2 \ExponDim_1} T} \Big]^{1+\theta}. \numberthis
	\end{align*}
	Therefore, we obtain that for all $x \in \R^d$ it holds that
	\begin{align*}
	&\big|\E[f_0(X^x_T)] - \E[f_0(Y^x_T)] \big|\\
	&  \leq 2 \iota (\iota +1)  d^{\ExponDim_0} (h/T)^{\nicefrac{1}{2}}  e^{(6\iota+5\iota^2 T)} \big[ \|x\| + 2 + \iota d^{\ExponDim_1}\max\{1, T\}   + \sqrt{(2\iota -1) \kappa d^{2 \ExponDim_1} T} \big]^{1+\theta} \\
	&  \leq  4 \iota^2  d^{\ExponDim_0} (h/T)^{\nicefrac{1}{2}} e^{(6\iota+5\iota^2 T)} \big[ \|x\| + 2 + \iota d^{\ExponDim_1}\max\{1, T\}   + \sqrt{2\iota \kappa d^{2 \ExponDim_1} T} \big]^{1+\theta}\\
	& \leq 4 \iota^2  d^{\ExponDim_0} (h/T)^{\nicefrac{1}{2}} e^{(6\iota+5\iota^2 T)}  \big[ \|x\| + 2 + 3 \iota d^{\ExponDim_1}\max\{1, T\} \big]^{1+\theta}\\
	& \leq 4 \iota^2  d^{\ExponDim_0} (h/T)^{\nicefrac{1}{2}} e^{(6\iota+5\iota^2 T)}  \big[ \|x\| + 5 \iota d^{\ExponDim_1}\max\{1, T\} \big]^{1+\theta}. \numberthis
	\end{align*}
	This establishes that
	\begin{align*}
	&\left[ \int_{\R^d} \big|\E[f_0(X^x_T)] - \E[f_0(Y^x_T)]  \big|^p \, \nu(dx) \right]^{\nicefrac{1}{p}} \numberthis\\
	& \leq 4 \iota^2  d^{\ExponDim_0} (h/T)^{\nicefrac{1}{2}} e^{(6\iota+5\iota^2 T)} 
	\left[ \int_{\R^d}
	\big[ \|x\| + 5 \iota d^{\ExponDim_1}\max\{1, T\} \big]^{p(1+\theta)} \, \nu(dx) \right]^{\nicefrac{1}{p}}\\
	&\leq 4 \iota^2  d^{\ExponDim_0} (h/T)^{\nicefrac{1}{2}} e^{(6\iota+5\iota^2 T)} 
	\left[ \int_{\R^d}
	\big[ 2^{\theta} \|x\|^{1+\theta} + 2^{\theta}  ( 5 \iota d^{\ExponDim_1} \max\{1, T\} )^{1+\theta} \big]^{p} \, \nu(dx) \right]^{\nicefrac{1}{p}}\\
	&\leq 2^{\theta+2} \iota^2  d^{\ExponDim_0} (h/T)^{\nicefrac{1}{2}} e^{(6\iota+5\iota^2 T)} \left[
	\left[ \int_{\R^d}
	\|x\|^{p(1+\theta)} \, \nu(dx) \right]^{\nicefrac{1}{p}} +   ( 5 \iota d^{\ExponDim_1} \max\{1, T\} )^{1+\theta} \right].
	\end{align*}
	Combining this and \eqref{eq:lem:ass:linear} assures that 
	\begin{align}
	\begin{split}
	&\left[ \int_{\R^d} \big|\E[f_0(X^x_T)] - \E[f_0(Y^x_T)]  \big|^p \, \nu(dx) \right]^{\nicefrac{1}{p}} \\
	& \leq 2^{\theta+2} \iota^2  d^{\ExponDim_0} (h/T)^{\nicefrac{1}{2}} e^{(6\iota+5\iota^2 T)} \big[ \kappa^{1+\theta} d^{\ExponDim_1(1+\theta)} +  ( 5 \iota d^{\ExponDim_1} \max\{1, T\} )^{1+\theta}  \big]\\
	& \leq 2^{\theta+2} ( 5 \iota \max\{1, T\} )^{\theta +1} \iota^2  d^{\ExponDim_0 + \ExponDim_1(\theta+1)} (h/T)^{\nicefrac{1}{2}} e^{(6\iota+5\iota^2 T)} \\
	& \leq 2^{\theta+2 + 3(\theta+1)} |\! \max\{1, T\} |^{\theta +1} \iota^{2 +\theta +1} d^{\ExponDim_0 + \ExponDim_1(\theta+1)} (h/T)^{\nicefrac{1}{2}} e^{(6\iota+5\iota^2 T)} \\
	& \leq 2^{4\theta+5} |\! \max\{1, T\} |^{\theta +1} \iota^{\theta+3} e^{(6\iota+5\iota^2 T)} d^{\ExponDim_0 + \ExponDim_1(\theta+1)} (h/T)^{\nicefrac{1}{2}}.
	\end{split}
	\end{align}
	The proof of Lemma~\ref{lem:integral:weak:euler} is thus completed.
\end{proof}

\begin{lemma}[Monte Carlo error]
	\label{lem:error:monte_carlo}
	Let $d, M, n \in \N$, $T, \kappa, \theta  \in (0, \infty)$, 
	$\ExponDim_0, \ExponDim_1 \in [0, \infty)$,	
	$B \in \R^{d \times n}$, $p \in [2, \infty)$,
	let $\nu  \colon \mathcal{B}(\R^d) \to [0,1] $ be a probability measure on $\R^d$, 
	let 
	$
	\left\| \cdot \right\| \colon \R^d \to [0,\infty)
	$
	be the $ d $-dimensional Euclidean norm, let $(\Omega, \mathcal{F}, \P)$ be a probability space, let $W^m \colon [0, T] \time \Omega \to \R^n$, $m \in \{1, 2, \ldots, M\}$, be independent standard Brownian motions, let $f_0 \colon \R^d \to \R$ be  $\mathcal{B}(\R^d) \slash \mathcal{B}(\R)$-measurable,
	let $f_1 \colon \R^d \to \R^d$ be  $\mathcal{B}(\R^d) \slash \mathcal{B}(\R^d)$-measurable, let $\chi \colon [0, T] \to [0, T]$ be  $\mathcal{B}([0,T]) \slash \mathcal{B}([0, T])$-measurable, assume for all $t  \in [0, T]$, $x \in \R^d$ that
	\begin{gather}
	\label{eq:monte_carlo:ass:f_1}
	|f_0(x)|  \leq \kappa d^{\ExponDim_0} (d^{\ExponDim_1 \theta} + \|x\|^{\theta}), \qquad 	\|f_1(x)\| \leq \kappa (d^{\ExponDim_1} + \|x\|), \\
	\label{eq:lem:ass:trace}
	\operatorname{Trace}(B^* B) \leq \kappa d^{2 \ExponDim_1}, \qquad 	\left[	\int_{\R^d} \|z\|^{p\theta} \, \nu (dz) \right]^{\nicefrac{1}{(p \theta)}}  \leq \kappa d^{\ExponDim_1},
	\end{gather}
	and $\chi(t) \leq t$, 
	and 
	let $ Y^{m, x} \colon [0,T] \times \Omega \to \R^d $, $m \in \{1, 2, \ldots, M\}$, $x \in \R^d$, be stochastic processes with 
	continuous sample paths which satisfy
	for all $x \in \R^d$, $m \in \{1, 2, \ldots, M\}$,  $ t \in [0,T] $ that 
	\begin{equation}
	Y_t^{m, x} = x + \int_0^t f_1\big( Y^{m, x}_{ \chi( s ) } \big) \, ds + B W_t^m,
	\end{equation}
	Then it holds that
	\begin{align}
	\begin{split}
	&\left( \E \!\left[ \int_{\R^d} \Big| \E[f_0(Y^{1, x}_T)] - \tfrac{1}{M} \Big[ \textstyle \sum\nolimits_{m=1}^M \displaystyle f_0(Y^{m, x}_T) \Big]  \Big|^p \, \nu(dx) \right] \right)^{\!\nicefrac{1}{p}}\\
	& \leq  2^{\theta+2}  p  \kappa (p \theta+p +1)^{\theta} (\kappa T +1)^{\theta}  e^{\kappa \theta T} (\kappa^{\theta} +1) d^{\ExponDim_0 + \ExponDim_1 \theta} M^{-\nicefrac{1}{2}}.
	\end{split}
	\end{align}
\end{lemma}
\begin{proof}[Proof of Lemma~\ref{lem:error:monte_carlo}]
	Throughout this proof let $ \iota = \max\{ \theta, 1 \} $.
	Note that \eqref{eq:monte_carlo:ass:f_1} and, e.g., \cite[Lemma~4.1]{JentzenSalimovaWelti2018} (with $d =d$, $m =n$, $\xi = \xi$, $p =q$, $c = \kappa$, $C= \kappa d^{\ExponDim_1}$, $T =T$, $B = B$, $\left\| \cdot \right\| = \left\| \cdot \right\|$, $(\Omega, \mathcal{F}, \P) = (\Omega, \mathcal{F}, \P)$, $W =  W^1$, $\mu = f_1$, $\chi = \chi$, $X = Y^{1, x}$ for $q \in [1, \infty)$, $x \in \R^d$ in the notation of \cite[Lemma~4.1]{JentzenSalimovaWelti2018}) prove that for all $q \in [1, \infty)$, $x \in \R^d$ it holds that
	\begin{align}
	\begin{split}
	\big(\E \big[ \| Y^{1, x}_T \|^q  \big]\big)^{\nicefrac{1}{q}} \leq \Big( \|x\| + \kappa d^{\ExponDim_1}T + \big(\E \big[ \| B W^1_T \|^q  \big]\big)^{\nicefrac{1}{q}} \Big) e^{\kappa T}.
	\end{split}
	\end{align}
	This, \eqref{eq:lem:ass:trace}, and, e.g., \cite[Lemma~4.2]{JentzenSalimovaWelti2018} (with $d = d$, $m = n$, $T = T$, $p = q$, $B = B$, $\left\| \cdot \right\| = \left\| \cdot \right\|$, $(\Omega, \mathcal{F}, \P) = (\Omega, \mathcal{F}, \P)$, $W = W^1$ for $q \in [1, \infty)$ in the notation of \cite[Lemma~4.2]{JentzenSalimovaWelti2018})  ensure that for all $q \in [1, \infty)$, $x \in \R^d$ it holds that
	\begin{align}
	\begin{split}
	\Big(\E \big[ \| Y^{1, x}_T \|^q  \big]\Big)^{\!\nicefrac{1}{q}} &\leq \Big( \|x\| + \kappa d^{\ExponDim_1}T +\sqrt{\max\{1, q-1\} \operatorname{Trace}(B^* B)T} \Big) e^{\kappa T}\\
	& \leq  \Big( \|x\| + \kappa d^{\ExponDim_1}T +\sqrt{\max\{1, q-1\} \kappa d^{2 \ExponDim_1} T} \Big) e^{\kappa T}\\
	&  \leq  \big( \|x\| + \kappa d^{\ExponDim_1}T +q \max\{\kappa T, 1\} d^{\ExponDim_1}  \big) e^{\kappa T}\\
	&  \leq  \big( \|x\| + (q+1) \max\{\kappa T, 1\} d^{\ExponDim_1}  \big) e^{\kappa T}.
	\end{split}
	\end{align}
	Combining this with \eqref{eq:monte_carlo:ass:f_1} and H\"older's inequality establishes for all $x \in \R^d$ that
	\begin{align}
	\begin{split}
	\Big(\E \big[ |f_{0} (Y^{1, x}_T ) |^p \big]\Big)^{\!\nicefrac{1}{p}} & \leq \kappa d^{\ExponDim_0 + \ExponDim_1 \theta}  + \kappa d^{\ExponDim_0}  \Big(\E \big[ \| Y^{1, x}_T \|^{p \theta} \big]\Big)^{\!\nicefrac{1}{p}}\\
	& \leq \kappa d^{\ExponDim_0 + \ExponDim_1 \theta} + \kappa d^{\ExponDim_0} \Big(\E \big[ \| Y^{1, x}_T \|^{p \iota} \big]\Big)^{\!\nicefrac{\theta}{(p \iota)}}\\
	& \leq \kappa  d^{\ExponDim_0 + \ExponDim_1 \theta} + \kappa d^{\ExponDim_0} \big( \|x\| + (p \iota+1) \max\{\kappa T, 1\} d^{\ExponDim_1}  \big)^{\theta} e^{\kappa \theta T}\\
	& \leq \kappa d^{\ExponDim_0 + \ExponDim_1 \theta} + \kappa d^{\ExponDim_0} \big( \|x\| + (p \iota+1) (\kappa T +1) d^{\ExponDim_1}  \big)^{\theta} e^{\kappa \theta T}.
	\end{split}
	\end{align}
	The fact that  $ \forall \, y, z \in \R, \alpha \in [0, \infty) \colon |y + z|^{\alpha}  \leq 2^{\alpha }(|y|^{\alpha} + |z|^{\alpha})$ hence proves  for all $x \in \R^d$ that
	\begin{align}
	\label{eq:apriori:f_0}
	\begin{split}
	\Big(\E \big[ |f_{0} (Y^{1, x}_T ) |^p \big]\Big)^{\!\nicefrac{1}{p}} & \leq  \kappa  d^{\ExponDim_0 + \ExponDim_1 \theta} + 2^{\theta} \kappa d^{\ExponDim_0} \big( \|x\|^{\theta} + (p \iota+1)^{\theta} (\kappa T +1)^{\theta} d^{\ExponDim_1 \theta}  \big) e^{\kappa \theta T}\\
	& \leq 2^{\theta} \kappa (p \iota+1)^{\theta} (\kappa T +1)^{\theta} d^{\ExponDim_0}  e^{\kappa \theta T} \big( \|x\|^{\theta} + 2 d^{\ExponDim_1 \theta} \big)\\
	& \leq 2^{\theta+1} \kappa (p \iota+1)^{\theta} (\kappa T +1)^{\theta} d^{\ExponDim_0}  e^{\kappa \theta T} \big( \|x\|^{\theta} +  d^{\ExponDim_1 \theta} \big).
	\end{split}
	\end{align}
	This implies that for all $x \in \R^d$ it holds that
	\begin{align}
	\begin{split}
	\E \big[ |f_{0} (Y^{1, x}_T ) | \big] \leq \Big(\E \big[ |f_{0} (Y^{1, x}_T ) |^p \big]\Big)^{\!\nicefrac{1}{p}} < \infty.
	\end{split}
	\end{align}
	Combining this with, e.g.,  \cite[Corollary~2.5]{GrohsWurstemberger2018} (with $p=p$, $d=1$, $n=M$, $\left\| \cdot \right\| = \left\| \cdot \right\|$, $(\Omega, \mathcal{F}, \P) = (\Omega, \mathcal{F}, \P)$, $X_i = f_0(Y^{i, x})$ for $i \in \{1, 2, \ldots, M\}$, $x \in \R^d$  in the notation of \cite[Corollary~2.5]{GrohsWurstemberger2018}) and \eqref{eq:apriori:f_0} assures for all $x \in \R^d$ that
	\begin{align}
	\begin{split}
	& \left( \E \! \left[ \Big| \E[f_0(Y^{1, x}_T)] - \tfrac{1}{M} \Big[ \textstyle \sum\nolimits_{m=1}^M \displaystyle f_0(Y^{m, x}_T) \Big]  \Big|^p  \right] \right)^{\!\nicefrac{1}{p}}\\
	& \leq 2 M^{-\nicefrac{1}{2}} \sqrt{(p-1)}  \Big(\E \Big[\big|f_{0} (Y^{1, x}_T) - \E[ f_{0} (Y^{1, x}_T)]\big|^p\Big] \Big)^{\!\nicefrac{1}{p}} \\
	& \leq 4 M^{-\nicefrac{1}{2}} \sqrt{(p-1)} \Big(\E \big[ |f_{0} (Y^{1, x}_T ) |^p \big]\Big)^{\!\nicefrac{1}{p}} \\
	& \leq 2^{\theta+3} M^{-\nicefrac{1}{2}} \sqrt{(p-1)}  \kappa (p \iota+1)^{\theta} (\kappa T +1)^{\theta} d^{\ExponDim_0}  e^{\kappa \theta T} \big( \|x\|^{\theta} +  d^{\ExponDim_1 \theta} \big) .
	\end{split}
	\end{align}
	This and the fact that $\sqrt{p -1} \leq \nicefrac{p}{2}$ establish that
	\begin{align*}
	& \left( \E \! \left[ \int_{\R^d} \Big| \E[f_0(Y^{1, x}_T)] - \tfrac{1}{M} \Big[ \textstyle \sum\nolimits_{m=1}^M  \displaystyle f_0(Y^{m, x}_T) \Big]  \Big|^p \, \nu(dx) \right] \right)^{\!\nicefrac{1}{p}} \numberthis\\
	& \leq 2^{\theta+2} M^{-\nicefrac{1}{2}} p \kappa (p \iota+1)^{\theta} (\kappa T +1)^{\theta} d^{\ExponDim_0}  e^{\kappa \theta T}  \left( \int_{ \R^d } \big( \|x\|^{\theta} +  d^{\ExponDim_1 \theta} \big)^p \, \nu(dx) \right)^{\!\nicefrac{1}{p}}\\
	& \leq   2^{\theta+2} M^{-\nicefrac{1}{2}} p  \kappa (p \iota+1)^{\theta} (\kappa T +1)^{\theta} d^{\ExponDim_0}  e^{\kappa \theta T} \left( d^{\ExponDim_1 \theta} + \left[ \int_{\R^d} \|x\|^{p \theta} \, \nu(dx) \right]^{\nicefrac{1}{p}} \right).
	\end{align*}
	Combining this and \eqref{eq:lem:ass:trace} demonstrates that
	\begin{align}
	\begin{split}
	& \left( \E \! \left[ \int_{\R^d} \Big| \E[f_0(Y^{1, x}_T)] - \tfrac{1}{M} \Big[ \textstyle \sum\nolimits_{m=1}^M \displaystyle f_0(Y^{m, x}_T) \Big]  \Big|^p \, \nu(dx) \right] \right)^{\!\nicefrac{1}{p}}\\& \leq  2^{\theta+2} M^{-\nicefrac{1}{2}} p  \kappa (p \iota+1)^{\theta} (\kappa T +1)^{\theta} d^{\ExponDim_0}  e^{\kappa \theta T}  \left[ d^{\ExponDim_1 \theta} + \kappa^{\theta} d^{\ExponDim_1 \theta} \right]\\
	& \leq 2^{\theta+2}  p  \kappa (p \iota+1)^{\theta} (\kappa T +1)^{\theta}  e^{\kappa \theta T} (\kappa^{\theta} +1) d^{\ExponDim_0 + \ExponDim_1 \theta} M^{-\nicefrac{1}{2}}\\
	& \leq 2^{\theta+2}  p  \kappa (p \theta+p +1)^{\theta} (\kappa T +1)^{\theta}  e^{\kappa \theta T} (\kappa^{\theta} +1) d^{\ExponDim_0 + \ExponDim_1 \theta} M^{-\nicefrac{1}{2}}.
	\end{split}
	\end{align}
	The proof of  Lemma~\ref{lem:error:monte_carlo} is thus completed.
\end{proof}

\begin{prop}
	\label{prop:monte_carlo_euler}
	Let $d, M, n \in \N$, $T, \kappa, \theta \in (0, \infty)$,
	$\ExponDim_0, \ExponDim_1 \in [0, \infty)$, $h \in (0, T]$, $B \in \R^{d \times n}$, $p \in [2, \infty)$,
	let $\nu  \colon \mathcal{B}(\R^d) \to [0,1] $ be a probability measure on $\R^d$, 
	let 
	$
	\left\| \cdot \right\| \colon \R^d \to [0,\infty)
	$
	be the $ d $-dimensional Euclidean norm, let $(\Omega, \mathcal{F}, \P)$ be a probability space, let $W^m \colon [0, T] \times \Omega \to \R^n$, $m \in \{1, 2, \ldots, M\}$, be independent standard Brownian motions, let $f_0 \colon \R^d \to \R$ and $f_1 \colon \R^d \to \R^d$ be  functions, let $\chi \colon [0, T] \to [0, T]$ be a function, assume for all $t  \in [0, T]$, $x, y \in \R^d$ that
		\setlength{\jot}{8pt}
	\begin{gather}
	\label{eq:cor:ass:f_0}
	|f_0(x) - f_0(y)| \leq \kappa d^{\ExponDim_0} (1 + \|x\|^{\theta} + \|y\|^{\theta}) \|x-y\|, \\
	\label{eq:cor:ass:f}
	|f_0(x)|  \leq \kappa d^{\ExponDim_0} ( d^{\ExponDim_1 \theta} + \|x\|^{\theta}), \qquad \operatorname{Trace}(B^* B) \leq \kappa d^{2\ExponDim_1},\\
	\label{eq:cor:ass:linear}
	\|f_1(x) - f_1(y)\| \leq \kappa \|x -y\|, \qquad \|f_1(x)\| \leq \kappa (d^{\ExponDim_1} + \|x\|),\\
	\label{eq:cor:ass:integral}
	\left[	\int_{\R^d} \|z\|^{p(1+\theta)} \, \nu (dz) \right]^{\nicefrac{1}{(p(1+\theta))}}  \leq \kappa d^{\ExponDim_1},
	\end{gather}
	and $\chi(t) = \max(\{0, h, 2h, \ldots\} \cap [0, t])$, 
	and 
	let $ X^x \colon [0,T] \times \Omega \to \R^d $, $x \in \R^d$, and  $ Y^{m, x} \colon [0,T] \times \Omega \to \R^d $, $m \in \{1, 2, \ldots, M\}$, $x \in \R^d$, be stochastic processes with 
	continuous sample paths which satisfy
	for all $x \in \R^d$, $m \in \{1, 2, \ldots, M\}$, $ t \in [0,T] $ that 
	$
	X^x_t =  x + \int_0^t f_1( X^x_s ) \, ds + B W_t^1
	$
	and
	\begin{equation}
	Y_t^{m, x} = x + \int_0^t f_1\big( Y^{m, x}_{ \chi( s ) } \big) \, ds + B W_t^m
	.
	\end{equation}
	Then it holds that
	\begin{align}
	\begin{split}
	&\left( \E \! \left[ \int_{\R^d} \Big| \E[f_0(X^{x}_T)] - \tfrac{1}{M} \Big[ \textstyle \sum\nolimits_{m=1}^M \displaystyle f_0(Y^{m, x}_T) \Big]  \Big|^p \, \nu(dx) \right] \right)^{\!\nicefrac{1}{p}}\\
	& \leq   2^{4\theta +5} |\! \max\{1, T\} |^{\theta +1} |\!\max\{ \kappa, \theta, 1 \}|^{2\theta +3} e^{(6\max\{ \kappa, \theta, 1 \}+5|\!\max\{ \kappa, \theta, 1 \}|^2 T)}  \\
	& \quad \cdot  p (p \theta + p +1)^{\theta} d^{\ExponDim_0 + \ExponDim_1(\theta+1)} \big((h/T)^{\nicefrac{1}{2}} +  M^{-\nicefrac{1}{2}}\big).
	\end{split}
	\end{align}
\end{prop}
\begin{proof}[Proof of Proposition~\ref{prop:monte_carlo_euler}]
	Throughout this proof let 
	$ \iota = \max\{ \kappa, \theta, 1 \}$.
	Note that the triangle inequality proves that
	\begin{align}
	\label{eq:cor:triangle}
	\begin{split}
	&\left( \E \! \left[ \int_{\R^d} \Big| \E[f_0(X^{x}_T)] - \tfrac{1}{M} \Big[ \textstyle \sum\nolimits_{m=1}^M \displaystyle f_0(Y^{m, x}_T) \Big]  \Big|^p \, \nu(dx) \right] \right)^{\!\nicefrac{1}{p}}\\
	& \leq \left( \int_{\R^d} \big|\E[f_0(X^x_T)] - \E[f_0(Y^{1, x}_T)]  \big|^p \,  \nu(dx) \right)^{\!\nicefrac{1}{p}}\\
	& + \left( \E \! \left[ \int_{\R^d} \Big| \E[f_0(Y^{1, x}_T)] - \tfrac{1}{M} \Big[ \textstyle \sum\nolimits_{m=1}^M \displaystyle f_0(Y^{m, x}_T) \Big]  \Big|^p \, \nu(dx) \right] \right)^{\!\nicefrac{1}{p}}.
	\end{split}
	\end{align}
	Next note that \eqref{eq:cor:ass:f}--\eqref{eq:cor:ass:integral} and  Lemma~\ref{lem:integral:weak:euler} (with
	$d=d$, $m=n$, $T=T$, $\kappa=\kappa$, $\theta=\theta$, $\ExponDim_0 = \ExponDim_0$, $\ExponDim_1 = \ExponDim_1$, $h=h$, $B=B$, $p=p$, $\nu=\nu$, $\left\| \cdot \right\| = \left\| \cdot \right\|$, $(\Omega, \mathcal{F}, \P) = (\Omega, \mathcal{F}, \P)$, $W = W^1$, $f_0 = f_0$, $f_1 = f_1$, $\chi=\chi$, $X^x = X^x$, $Y^x= Y^{1, x}$ for $x \in \R^d$ in the notation of  Lemma~\ref{lem:integral:weak:euler}) demonstrates that 
	\begin{align*}
	\label{eq:cor:first_term}
	& \left( \int_{\R^d} \big|\E[f_0(X^x_T)] - \E[f_0(Y^{1, x}_T)]  \big|^p \, \nu(dx) \right)^{\!\nicefrac{1}{p}}  \\
	& \leq 2^{4\theta +5} |\! \max\{1, T\} |^{\theta +1} |\!\max\{ \kappa, \theta, 1 \}|^{\theta +3} e^{(6\max\{ \kappa, \theta, 1 \}+5|\!\max\{ \kappa, \theta, 1 \}|^2 T)}  d^{\ExponDim_0 + \ExponDim_1(\theta+1)} (h/T)^{\nicefrac{1}{2}}\\
	\numberthis
	& = 2^{4\theta +5} | \!\max\{1, T\} |^{\theta +1} \iota^{\theta +3} e^{(6\iota+5\iota^2 T)}  d^{\ExponDim_0 + \ExponDim_1(\theta+1)} (h/T)^{\nicefrac{1}{2}}.
	\end{align*}
	Moreover, observe that H\"older's inequality and \eqref{eq:cor:ass:integral} imply that
	\begin{align}
	\begin{split}
	\left[	\int_{\R^d} \|z\|^{(p\theta)} \, \nu (dz) \right]^{\nicefrac{1}{p \theta}}  \leq \left[	\int_{\R^d} \|z\|^{p(1+\theta)} \, \nu (dz) \right]^{\nicefrac{1}{(p(1+\theta))}}  \leq \kappa d^{\ExponDim_1}.
	\end{split}
	\end{align}
	Lemma~\ref{lem:error:monte_carlo} (with $d=d$, $M=M$, $n=n$, $T=T$, $\kappa=\kappa$, $\theta=\theta$, $\ExponDim_0 = \ExponDim_0$, $\ExponDim_1 = \ExponDim_1$,  $B=B$, $p=p$,  $\nu = \nu$, $\left\| \cdot \right\| = \left\| \cdot \right\|$, $(\Omega, \mathcal{F}, \P) = (\Omega, \mathcal{F}, \P)$, $W^m=W^m$, $f_0=f_0$, $f_1=f_1$, $\chi=\chi$, $Y^{m,x} = Y^{m,x}$ for $m \in \{1, 2, \ldots, M\}$, $x \in \R^d$ in the notation of Lemma~\ref{lem:error:monte_carlo}),  \eqref{eq:cor:ass:f}, and \eqref{eq:cor:ass:linear}  hence establish that
	\begin{align}
	\begin{split}
	&\left( \E \! \left[ \int_{\R^d} \Big| \E[f_0(Y^{1, x}_T)] - \tfrac{1}{M} \Big[ \textstyle \sum\nolimits_{m=1}^M \displaystyle f_0(Y^{m, x}_T) \Big]  \Big|^p \, \nu(dx) \right] \right)^{\!\nicefrac{1}{p}}\\ & \leq  2^{\theta+2}  p  \kappa (p \theta+p +1)^{\theta} (\kappa T +1)^{\theta}  e^{\kappa \theta T} (\kappa^{\theta} +1) d^{\ExponDim_0 + \ExponDim_1 \theta} M^{-\nicefrac{1}{2}}\\
	& \leq 2^{\theta+2}  p  \kappa (p \theta+p +1)^{\theta}
	| \!\max\{1, T\} |^{\theta} (\kappa  +1)^{\theta}  e^{\kappa \theta T} (\kappa^{\theta} +1) d^{\ExponDim_0 + \ExponDim_1 \theta} M^{-\nicefrac{1}{2}}\\
	& \leq 2^{\theta+2}  p  \iota (p \theta+p +1)^{\theta}
	| \!\max\{1, T\} |^{\theta} (2 \iota)^{\theta}  e^{\iota \theta T} ({\iota}^{\theta} +1) d^{\ExponDim_0 + \ExponDim_1 \theta} M^{-\nicefrac{1}{2}}\\
	& \leq 2^{2\theta+3}  p  \iota^{2\theta +1} (p \theta+p +1)^{\theta}
	|\! \max\{1, T\} |^{\theta}   e^{\iota \theta T} d^{\ExponDim_0 + \ExponDim_1 \theta} M^{-\nicefrac{1}{2}}.
	\end{split}
	\end{align}
	This, \eqref{eq:cor:triangle}, and \eqref{eq:cor:first_term} assure that
	\begin{align}
	\begin{split}
	&\left( \E \! \left[ \int_{\R^d} \Big| \E[f_0(X^{x}_T)] - \tfrac{1}{M} \Big[ \textstyle \sum\nolimits_{m=1}^M \displaystyle f_0(Y^{m, x}_T) \Big]  \Big|^p \, \nu(dx) \right] \right)^{\!\nicefrac{1}{p}}  \\
	&  \leq 2^{4\theta +5} | \!\max\{1, T\} |^{\theta +1} \iota^{\theta +3} e^{(6\iota+5\iota^2 T)}  d^{\ExponDim_0 + \ExponDim_1(\theta+1)} (h/T)^{\nicefrac{1}{2}} \\
	& \quad +2^{2\theta+3}  p  \iota^{2\theta +1} (p \theta+p +1)^{\theta}
	|\! \max\{1, T\} |^{\theta}   e^{\iota \theta T} d^{\ExponDim_0 + \ExponDim_1 \theta} M^{-\nicefrac{1}{2}}\\
	& \leq 2^{4\theta +5} | \!\max\{1, T\} |^{\theta +1} \iota^{2\theta +3} e^{(6\iota+5\iota^2 T)} d^{\ExponDim_0 + \ExponDim_1(\theta+1)} \\
	& \quad \cdot \Big(  (h/T)^{\nicefrac{1}{2}}+ p (p \theta + p +1)^{\theta}  M^{-\nicefrac{1}{2}}  \Big)\\
	& \leq   2^{4\theta +5} | \!\max\{1, T\} |^{\theta +1} \iota^{2\theta +3} e^{(6\iota+5\iota^2 T)} d^{\ExponDim_0 + \ExponDim_1(\theta+1)} p (p \theta + p +1)^{\theta} \\
	& \quad \cdot \big((h/T)^{\nicefrac{1}{2}} +  M^{-\nicefrac{1}{2}}\big).
	\end{split}
	\end{align}
	The proof of Proposition~\ref{prop:monte_carlo_euler} is thus completed.
\end{proof}

\subsection{DNN approximations for Kolmogorov PDEs}

\begin{theorem}
	\label{thm:dnn:kolmogorov}
	Let 
	$
	A_d = (A_{d, i, j})_{(i, j) \in \{1, \ldots, d\}^2} \in \R^{ d \times d }
	$,
	$ d \in \N $,
	be symmetric positive semidefinite matrices, 
	let $\left\| \cdot \right\| \colon (\cup_{d \in \N} \R^d) \to [0, \infty)$ satisfy for all $d \in \N$, $x = (x_1, x_2, \ldots, x_d) \in \R^d$ that $\|x\| = ( \smallsum_{i=1}^d |x_i|^2)^{\nicefrac{1}{2}}$,
	for every $ d \in \N $ 
let $ \nu_d \colon \mathcal{B}(\R^d) \to [0,1]$ be a probability measure on $\R^d$,
	let
	$ \varphi_{0,d} \colon \R^d \to \R $, $ d \in \N $,
	and
	$ \varphi_{ 1, d } \colon \R^d \to \R^d $,
	$ d \in \N $,
	be functions,
	let
	$ T, \kappa \in (0, \infty)$,
	$\ExponError, \ExponDim_1, \ExponDim_2, \ldots, \ExponDim_6  \in [0, \infty)$, $\theta \in [1, \infty)$, $p \in [2, \infty)$,	
	$
	( \phi^{ m, d }_{ \varepsilon } )_{ 
		(m, d, \varepsilon) \in \{ 0, 1 \} \times \N \times (0,1] 
	} 
	\subseteq \ANNs
	$, 	$\activation \in C(\R, \R)$ satisfy for all $x \in \R$ that
	$\activation(x) = \max\{x, 0\}$,
	assume for all
	$ d \in \N $, 
	$ \varepsilon \in (0,1] $, 
	$ m \in \{0, 1\}$,
	$ 
	x, y \in \R^d
	$
	that
	$
	\functionANN( \phi^{ 0, d }_{ \varepsilon } )
	\in 
	C( \R^d, \R )
	$,
	$
	\functionANN( \phi^{ 1, d }_{ \varepsilon } )
	\in
	C( \R^d, \R^d )
	$, 	$
	\operatorname{Trace}(A_d)
	\leq 
	\kappa d^{ 2 \ExponDim_1  }
	$,
	$[	\int_{\R^d} \|x\|^{2p \theta} \, \nu_d (dx) ]^{\nicefrac{1}{(2p \theta)}} \allowbreak \leq \kappa d^{\ExponDim_1 + \ExponDim_2}$,
	$ 
	\paramANN( \phi^{ m, d }_{ \varepsilon } ) 
	\leq \kappa d^{ 2^{(-m)} \ExponDim_3 } \varepsilon^{ - 2^{(-m)}  \ExponError }$,
 $ |( \functionANN (\phi^{ 0, d }_{ \varepsilon }) )(x) - ( \functionANN (\phi^{ 0, d }_{ \varepsilon }) )(y)| \leq \kappa d^{\ExponDim_6} (1   + \|x\|^{\theta} + \|y \|^{\theta})\|x-y\|$, 
 	$
 \|
 ( \functionANN (\phi^{ 1, d }_{ \varepsilon }) )(x)    
 \|	
 \leq 
 \kappa ( d^{ \ExponDim_1 + \ExponDim_2 } + \| x \| )
 $, $|
	\varphi_{ 0, d }( x )| \leq \kappa d^{ \ExponDim_6 }
	( d^{ \theta(\ExponDim_1 + \ExponDim_2) } + \| x \|^{ \theta } )$,
	$
	\| 
	\varphi_{ 1, d }( x ) 
	- 
	\varphi_{ 1, d }( y )
	\|
	\leq 
	\kappa 
	\| x - y \| 
	$, 	
	and
	\begin{equation}
	\label{eq:intro_hypo}
	\| 
	\varphi_{ m, d }(x) 
	- 
	( \functionANN (\phi^{ m, d }_{ \varepsilon }) )(x)
	\|
	\leq 
	\varepsilon  \kappa d^{\ExponDim_{(5 -m)}} (d^{\theta(\ExponDim_1 + \ExponDim_2)}+ \|x\|^{\theta}) 
	,
	\end{equation}
	and for every $ d \in \N $ let
	$ u_d \colon [0,T] \times \R^{d} \to \R $
	be an 
	at most polynomially growing viscosity solution of
	\begin{equation}
	\label{eq:PDE}
	\begin{split}
	( \tfrac{ \partial }{\partial t} u_d )( t, x ) 
	& = 
	( \tfrac{ \partial }{\partial x} u_d )( t, x )
	\,
	\varphi_{ 1, d }( x )
	+
	\textstyle
	\sum\limits_{ i, j = 1 }^d
	\displaystyle
	A_{ d, i, j }
	\,
	( \tfrac{ \partial^2 }{ \partial x_i \partial x_j } u_d )( t, x )
	\end{split}
	\end{equation}
	with $ u_d( 0, x ) = \varphi_{ 0, d }( x ) $
	for $ ( t, x ) \in (0,T) \times \R^d $ 	(cf.~Definition~\ref{Def:ANN} and Definition~\ref{Definition:ANNrealization}).
	Then 
	there exist
		$
	c \in \R
	$ and
	$
	( 
	\Psi_{ d, \varepsilon } 
	)_{ (d , \varepsilon)  \in \N \times (0,1] } \subseteq \ANNs
	$
	such that
	for all 
	$
	d \in \N 
	$,
	$
	\varepsilon \in (0,1] 
	$
	it holds that
	$
\mathcal{R}( \Psi_{ d, \varepsilon } )
\in C( \R^{ d }, \R )
$, $[
\int_{ \R^d }
|
u_d(T, x) - ( \mathcal{R} (\Psi_{ d, \varepsilon }) )( x )
|^p
\,
\nu_d(dx)
]^{ \nicefrac{ 1 }{ p } }
\leq
\varepsilon$, 
and
\begin{align}
\label{eq:kolmogorov:statement}
\paramANN( \Psi_{ d, \varepsilon } ) \leq c d^{6[\ExponDim_6 + (\ExponDim_1 + \ExponDim_2)(\theta+1)] + \max\{4, \ExponDim_3\}  +  \ExponError \max\{\ExponDim_5 + \theta (\ExponDim_1 + \ExponDim_2), \ExponDim_4 + \ExponDim_6 + 2\theta (\ExponDim_1 + \ExponDim_2)\}}  \varepsilon^{-(\ExponError +6)}.
\end{align}
\end{theorem}
\begin{proof}[Proof of Theorem~\ref{thm:dnn:kolmogorov}]
Throughout this proof let $ \mathcal{A}_d \in \R^{ d \times d } $, $ d \in \N $, satisfy 
for all $ d \in \N $ that
$
\mathcal{A}_d = \sqrt{ 2 A_d }
$, 
let $(\Omega, \mathcal{F}, \P)$ be a probability space, 
let $ W^{ d, m } \colon [0,T] \times \Omega \to \R^d $, $ d, m \in \N $, 
be independent standard Brownian motions, 
let  $Z^{N, d, m}_n \colon \Omega \to \R^{d} $, $n \in \{0, 1, \ldots, N-1\}$, $m \in \{1, 2, \ldots, N\}$, $d, N \in \N$, be the random variables which satisfy for all $N, d \in \N$, $m \in  \{1, 2, \ldots, N\}$,  $n \in \{0, 1, \ldots, N-1\}$ that
\begin{align}
\label{eq:increments}
Z^{N, d, m}_n = \mathcal{A}_d W^{d, m}_{\frac{(n+1)T}{N}} - \mathcal{A}_d W^{d, m}_{\frac{nT}{N}},
\end{align}
let $f_{N, d} \colon \R^{d} \times \R^{d} \to \R^{d}$,  $d, N \in \N$, satisfy for all $N, d \in \N$, $x, y \in \R^d$ that
\begin{align}
\label{eq:varphi:Euler}
f_{N, d}(x, y) = x+ y + \tfrac{T}{N} \varphi_{1, d}(y),
\end{align}
let $ X^{ d, x } \colon [0,T] \times \Omega \to \R^d $,  $ x \in \R^d $,  $ d \in \N $,
be stochastic processes with continuous sample paths 
which satisfy for all  $ d \in \N $, $ x \in \R^d $,  $ t \in [0,T] $ that
\begin{equation}
\label{eq:X_processes}
X^{ d, x }_t 
= x + \int_0^t \varphi_{ 1, d }( X^{ d, x }_s ) \, ds 
+ 
\mathcal{A}_d
W^{ d, 1 }_t 
\end{equation} 
(cf., e.g., \cite[item (i) in Theorem~3.1]{JentzenSalimovaWelti2018}
(with $(\Omega, \mathcal{F}, \P) = (\Omega, \mathcal{F}, \P)$, $T=T$, $d=d$, $m=d$, $B= \mathcal{A}_d$, 
$\mu = \varphi_{1,d}$ for $d \in \N$ in the notation of \cite[Theorem~3.1]{JentzenSalimovaWelti2018})), 
 let $Y^{N, d, x}_n = (Y^{N, d, m, x}_n)_{m \in \{1, 2, \ldots, N\}} \colon \Omega \to \R^{N d}$, $n \in \{0, 1, \ldots, N\}$, $x \in \R^d$, $d, N \in \N$,   satisfy  for all $N,  d \in \N$,  $m \in \{1, 2, \ldots, N\}$,   $x \in \R^d$, $n \in \{1, 2, \ldots, N\}$ that 	$Y^{N, d, m, x}_{0}  = x$ and
\begin{align}
\label{eq:approx:Euler}
Y^{N, d, m, x}_{n} &= f_{N, d} \big(Z^{N, d, m}_{n-1}, Y^{N, d, m, x}_{n-1}\big),
\end{align}
let $g_{N, d} \colon  \R^{Nd} \to \R$, $d, N \in \N$, satisfy for all $N, d \in \N$, $x = (x_i)_{i \in \{1, 2, \ldots, N\}} \in \R^{Nd}$ that
\begin{equation}
\label{eq:monte_carlo}
g_{N, d}(x) = \frac{1}{N} \sum_{i=1}^N \varphi_{0,d} (x_i),
\end{equation}
and
let $\SubsetANNs_{d, \varepsilon} \subseteq \ANNs$, 
$ \varepsilon \in (0, 1]$, $d \in \N$, 
 satisfy for all $d \in \N$, $\varepsilon \in (0, 1]$ that
\begin{align*}
\label{eq:subset:DNN}
&\SubsetANNs_{d, \varepsilon} \numberthis\\
&= \Big\{  \Phi \in \ANNs \colon \! \big[ \big( \functionANN( \Phi ) \in C(\R^d, \R^d) \big) \wedge \big(\dimANNlevel_{\lengthANN(\Phi) -1}(\Phi) \leq \dimANNlevel_{\lengthANN(\phi^{1,d}_{\varepsilon}) -1}(\phi^{1,d}_{\varepsilon}) + 2d \big) \big] \Big\}
\end{align*}
(cf.~Definition~\ref{Def:ANN:dimensions}).
Note that \eqref{eq:increments} and, e.g., \cite[Lemma~4.2]{JentzenSalimovaWelti2018} (with $d = d$, $m = d$, $T = T$, $p = 2p \theta$, $B = \mathcal{A}_d$, $(\Omega, \mathcal{F}, \P) = (\Omega, \mathcal{F}, \P)$, $W = W^{d,m}$ for $d, m \in \N$ in the notation of \cite[Lemma~4.2]{JentzenSalimovaWelti2018}) 
ensure that
for all $N, d \in \N$, $m \in  \{1, 2, \ldots, N\}$,  $n \in \{0, 1, \ldots, N-1\}$ it holds that
\begin{align}
\begin{split}
&\left( \E\! \left[ \|Z^{N, d, m}_{n} \|^{2 p \theta} \right]\right)^{\nicefrac{1}{(2 p \theta)}}  =  \left( \E\! \left[\Big\| \mathcal{A}_d W^{d, m}_{\frac{(n+1)T}{N}} - \mathcal{A}_d W^{d, m}_{\frac{nT}{N}}\Big\|^{2 p \theta} \right]\right)^{\!\nicefrac{1}{(2 p \theta)}}\\
&\leq \left( \E\! \left[\Big\| \mathcal{A}_d W^{d, m}_{\frac{(n+1)T}{N}} \Big\|^{2 p \theta} \right]\right)^{\!\nicefrac{1}{(2 p \theta)}} + \left( \E\! \left[\Big\|  \mathcal{A}_d W^{d, m}_{\frac{nT}{N}}\Big\|^{2 p \theta} \right]\right)^{\!\nicefrac{1}{(2 p \theta)}}\\
& \leq 2 \sqrt{(2p\theta-1)  \operatorname{Trace}(\mathcal{A}_d^* \mathcal{A}_d) T} =  2 \sqrt{2(2p\theta-1)  \operatorname{Trace}(A_d) T}.
\end{split}
\end{align}
This and the assumption that $ \forall \, d \in \N \colon
\operatorname{Trace}(A_d)
\leq 
\kappa d^{ 2 \ExponDim_1 }
$
assure for all $N, d \in \N$, $m \in  \{1, 2, \ldots, N\}$,  $n \in \{0, 1, \ldots, N-1\}$  that
\begin{align}
\label{eq:thm:dnn:Z:int}
\begin{split}
\left( \E\! \left[ \|Z^{N, d, m}_{n} \|^{2 p \theta} \right]\right)^{\nicefrac{1}{(2 p \theta)}}  \leq 4 p \theta \sqrt{\kappa T} d^{\ExponDim_1}.
\end{split}
\end{align}
Moreover, observe that Lemma~\ref{lem:Relu:identity} (with
$d=d$, $a=a$
 for $d \in \N$ in the notation of Lemma~\ref{lem:Relu:identity})  ensures that there exist $\idRelu_d \in \ANNs$, $d \in \N$,  such that for all $d \in \N$, $x \in \R^d$ it holds that
 $\dims (\idRelu_d) = (d, 2d, d)$,
$\functionANN( \idRelu_{d}) \in C(\R^d, \R^d)$,  and
$(\functionANN (\idRelu_d))(x) = x$. This and \eqref{eq:subset:DNN} assure for all $d \in \N$, $\varepsilon \in (0, 1]$ that $\idRelu_d \in \SubsetANNs_{d, \varepsilon}$ and
\begin{align}
\label{eq:parameters:Id}
\paramANN(\idRelu_d) = 2d(d+1) + d(2d+1) = 2d^2 + 2d+2d^2 +d = 4d^2 + 3d \leq 7d^2.
\end{align}
Next note that Lemma~\ref{lem:ANNscalar}
demonstrates
that for all $N, d \in \N$, $\varepsilon \in (0, 1]$
it holds that
$\mathcal{D}(\frac{T}{N} \circledast \phi^{1, d}_{\varepsilon}) = \mathcal{D}(\phi^{1, d}_{\varepsilon})$,
$\functionANN(\frac{T}{N} \circledast \phi^{1, d}_{\varepsilon} ) \in C(\R^d, \R^d)$,  and 
\begin{align}
\label{eq:phi:multiply}
\functionANN \big(\tfrac{T}{N} \circledast \phi^{1, d}_{\varepsilon}\big) = \tfrac{T}{N} \functionANN (\phi^{1, d}_{\varepsilon})
\end{align}
(cf.~Definition~\ref{Definition:ANNscalar}).
This, the fact that $\dims (\idRelu_d) = (d, 2d, d)$, and Lemma~\ref{lem:Composition_Sum} (with 
$a=a$,  $L_1 = \lengthANN (\tfrac{T}{N} \circledast \phi^{1, d}_{\varepsilon}) $, $L_2 = 2$, $\mathbb{I} = \idRelu_d$, $\Phi_1 =  \tfrac{T}{N} \circledast \phi^{1, d}_{\varepsilon}$, $\Phi_2 = \idRelu_d$, 
$d=d$, $\mathfrak{i} = 2d$, $(l_{1, 0}, l_{1, 1}, \ldots, l_{1, L_1})  = \dims(\tfrac{T}{N} \circledast \phi^{1, d}_{\varepsilon})$, $(l_{2, 0}, l_{2, 1},  l_{2, L_2}) = (d, 2d, d)$
for $d,   N \in \N$, $\varepsilon \in (0, 1]$ in the notation of Lemma~\ref{lem:Composition_Sum}) establish that  there exist $\mathbf{f}^{N, d}_{\varepsilon} \in \ANNs$, $\varepsilon \in (0, 1]$, $d, N \in \N$, such  that
for all  $N, d \in \N$, $\varepsilon \in (0, 1]$, $x \in \R^d$ it holds that $\functionANN (\mathbf{f}^{N, d}_{\varepsilon}) \in C(\R^d, \R^d)$ and
\begin{align}
( \functionANN (\mathbf{f}^{N, d}_{\varepsilon})) (x) = x +  \big( \functionANN \big(\tfrac{T}{N} \circledast \phi^{1, d}_{\varepsilon}\big)\big)(x) =  x + \tfrac{T}{N} ( \functionANN (\phi^{1, d}_{\varepsilon}))(x).
\end{align}
Items~\eqref{item:ANN:vector:comp:2}--\eqref{item:ANN:vector:comp:3} in Lemma~\ref{lem:ANN:vector:comp} hence ensure that  there exist $\mathbf{f}^{N, d}_{\varepsilon, z}  \in \ANNs$, $z \in \R^d$, $\varepsilon \in (0, 1]$,  $d, N \in \N$, which satisfy for all
  $N, d \in \N$, $\varepsilon \in (0, 1]$, $z, x \in \R^d$ that $\functionANN (\mathbf{f}^{N, d}_{\varepsilon, z}) \in C(\R^d, \R^d)$ and
\begin{align}
\label{eq:DNN:Euler}
( \functionANN (\mathbf{f}^{N, d}_{\varepsilon, z})) (x) = ( \functionANN (\mathbf{f}^{N, d}_{\varepsilon})) (x) + z =  z+x+  \tfrac{T}{N} ( \functionANN( \phi^{1, d}_{\varepsilon}))(x) .
\end{align}
This, \eqref{eq:varphi:Euler},  and \eqref{eq:intro_hypo} imply  for all $N, d \in \N$, $\varepsilon \in (0, 1]$,  $x, z \in \R^d$ that $(\R^d \ni \mathfrak{z} \mapsto  ( \functionANN (\mathbf{f}^{N, d}_{\varepsilon, \mathfrak{z}}))(x) \in \R^d)$ is $\mathcal{B}(\R^d) \slash \mathcal{B}(\R^d)$-measurable and
\begin{align}
\label{eq:thm:dnn:varphi:appr}
\begin{split}
\|f_{N, d}(z, x) - ( \functionANN (\mathbf{f}^{N, d}_{\varepsilon, z}))(x)  \| &= \tfrac{T}{N} \| \varphi_{1, d} (x) - ( \functionANN (\phi^{1, d}_{\varepsilon}))(x) \|\\
&\leq \tfrac{T  \varepsilon \kappa d^{\ExponDim_4}}{N} 
(
d^{\theta(\ExponDim_1 + \ExponDim_2)} + \| x \|^{ \theta }
)\\
& \leq  \varepsilon T \kappa d^{\ExponDim_4}
(
d^{\theta(\ExponDim_1 + \ExponDim_2)} + \| x \|^{ \theta }
).
\end{split}
\end{align}
Next note that \eqref{eq:DNN:Euler}  and the assumption that 	$ \forall \, \varepsilon \in (0, 1], d \in \N,  x \in \R^d \colon
\|
( \functionANN (\phi^{ 1, d }_{ \varepsilon }) )(x)    
\|	
\leq 
\kappa ( d^{ \ExponDim_1 + \ExponDim_2 } + \| x \| )
$ prove for all $N, d \in \N$, $\varepsilon \in (0, 1]$,  $x, z \in \R^d$ that
\begin{align}
\label{eq:thm:dnn:varphi:linear}
\begin{split}
\|( \functionANN (\mathbf{f}^{N, d}_{\varepsilon, z}))(x)  \| &\leq \|z\| + \|x\| + \tfrac{T}{N} \|
( \functionANN (\phi^{ 1, d }_{ \varepsilon }) )(x)
\|\\
& \leq  \|z\| + \|x\| + \tfrac{T \kappa}{N} ( d^{ \ExponDim_1 + \ExponDim_2 } + \| x \| ) \\
& = \big(1 + \tfrac{T \kappa}{N} \big) \|x\| + \tfrac{T \kappa d^{\ExponDim_1 + \ExponDim_2}}{N} + \|z\|\\
& \leq  \big(1 + \tfrac{T \kappa}{N} \big) \|x\| + (T\kappa +1) (d^{\ExponDim_1 + \ExponDim_2} + \|z\|)\\
& \leq  \big(1 + \tfrac{T \kappa}{N} \big) \|x\| + (T\kappa +1) d^{\ExponDim_2}(d^{\ExponDim_1} + \|z\|). 
\end{split}
\end{align}
In addition, observe that \eqref{eq:varphi:Euler} and the assumption that $\forall \, d \in \N, x, y \in \R^d \colon \| 
\varphi_{ 1, d }( x ) 
- 
\varphi_{ 1, d }( y )
\|
\leq 
\kappa 
\| x - y \|$ imply that for all $N, d \in \N$, $x, z \in \R^d$ it holds that
\begin{align}
\label{eq:thm:dnn:varphi:lipschitz}
\begin{split}
\|f_{N, d}(z, x) - f_{N, d}(z, y) \| & = \|x + \tfrac{T}{N} \varphi_{1,d}(x) - y - \tfrac{T}{N} \varphi_{1, d}(y) \|\\
& \leq \|x-y\| + \tfrac{T}{N} \|\varphi_{ 1, d }( x ) 
- 
\varphi_{ 1, d }( y )
\|\\
& \leq  \big(1 + \tfrac{T \kappa}{N} \big) \|x -y\| \leq (1+ T \kappa)\|x -y\|.
\end{split}
\end{align}
Moreover, note that \eqref{eq:phi:multiply}, the fact that $\mathcal{D} (\idRelu_d) = (d, 2d, d)$, and
Lemma~\ref{lem:Composition_Sum} (with 
$a=a$,  $L_1 = \lengthANN (\tfrac{T}{N} \circledast \phi^{1, d}_{\varepsilon}) $, $L_2 = \lengthANN (\Phi)$, $\mathbb{I} = \idRelu_d$, $\Phi_1 =  \tfrac{T}{N} \circledast \phi^{1, d}_{\varepsilon}$, $\Phi_2 = \Phi$, 
$d=d$, $\mathfrak{i} = 2d$, $(l_{1, 0}, l_{1, 1}, \ldots, l_{1, L_1}) = \dims(\tfrac{T}{N} \circledast \phi^{1, d}_{\varepsilon}) = \dims (\phi^{1, d}_{\varepsilon})$, $(l_{2, 0}, l_{2, 1}, \ldots, l_{2, L_2}) = \dims (\Phi)$
for $N, d \in \N$, $\varepsilon \in (0, 1]$, $\Phi \in \SubsetANNs_{d, \varepsilon}$ in the notation of Lemma~\ref{lem:Composition_Sum})
 prove that for every  $N, d \in \N$, $\varepsilon \in (0, 1]$,  $\Phi \in \SubsetANNs_{d, \varepsilon}$ there exists
	$\hat{\Phi} \in \ANNs$ such that
for all $x \in \R^d$ it holds that 
$\functionANN(\hat{\Phi}) \in C(\R^d, \R^d)$,
$\dimANNlevel_{\lengthANN(\hat{\Phi}) -1}(\hat{\Phi}) \leq \dimANNlevel_{\lengthANN(\phi^{1,d}_{\varepsilon}) -1}(\phi^{1,d}_{\varepsilon}) + 2d $,
$\paramANN(\hat{\Phi}) \leq \paramANN(\Phi) + [\frac{1}{2} \paramANN(\idRelu_d) + \paramANN(\phi^{1, d}_{\varepsilon})]^2$,
and 
\begin{align}
\begin{split}
(\functionANN (\hat{\Phi})) (x) &= (\functionANN(\Phi))(x) +
\big( \big(\functionANN\big(\tfrac{T}{N} \circledast \phi^{1, d}_{\varepsilon}\big)\big)\circ (\functionANN(\Phi))\big)(x)\\
&= (\functionANN(\Phi))(x) +
\tfrac{T}{N} \big((\functionANN(  \phi^{1, d}_{\varepsilon}))\circ (\functionANN(\Phi))\big)(x).
\end{split}
\end{align}
This, \eqref{eq:subset:DNN}, \eqref{eq:parameters:Id}, and the fact that
$\forall \, d \in \N, \varepsilon \in (0, 1] \colon 
\paramANN( \phi^{ 1, d }_{ \varepsilon } ) 
\leq \allowbreak \kappa d^{ 2^{(-1)} \ExponDim_3 } \varepsilon^{ - 2^{(-1)}  \ExponError }$ demonstrate that  for every  $N, d \in \N$, $\varepsilon \in (0, 1]$,  $\Phi \in \SubsetANNs_{d, \varepsilon}$ there exists
$\hat{\Phi} \in \SubsetANNs_{d, \varepsilon}$ such that
for all $x \in \R^d$ it holds that 
\begin{equation}
\paramANN(\hat{\Phi}) \leq \paramANN(\Phi) + (4d^2 +\kappa d^{ 2^{(-1)} \ExponDim_3 } \varepsilon^{ - 2^{(-1)}  \ExponError })^2 \leq \paramANN(\Phi) +  (\kappa+4)^2 d^{ \max\{4, \ExponDim_3\}} \varepsilon^{-\ExponError}
\end{equation}
and 
\begin{align}
\begin{split}
(\functionANN (\hat{\Phi})) (x) = (\functionANN(\Phi))(x) + 
\tfrac{T}{N} \big((\functionANN(  \phi^{1, d}_{\varepsilon}))\circ (\functionANN(\Phi))\big)(x).
\end{split}
\end{align}
Items~\eqref{item:ANN:vector:comp:1}--\eqref{item:ANN:vector:comp:3} in Lemma~\ref{lem:ANN:vector:comp} and \eqref{eq:DNN:Euler}   hence  ensure that 
for every  $N, d \in \N$, $\varepsilon \in (0, 1]$,  $\Phi \in \SubsetANNs_{d, \varepsilon}$ there exist	$(\hat{\Phi}_{z})_{z \in \R^d} \subseteq \SubsetANNs_{d, \varepsilon}$ such that
	for all $x, z, \mathfrak{z} \in \R^d$ it holds that
\begin{align}
\begin{split}
\label{eq:thm:dnn:phi:exist:1}
(\functionANN (\hat{\Phi}_z))(x) & = z+  (\functionANN(\Phi))(x)+ \tfrac{T}{N} \big((\functionANN(  \phi^{1, d}_{\varepsilon}))\circ (\functionANN(\Phi))\big)(x) \\
& = ( \functionANN (\mathbf{f}^{N, d}_{\varepsilon, z}))\big( (\functionANN (\Phi))(x)\big),
\end{split}
\end{align}
\vspace{-.3cm}
\begin{equation}
\label{eq:thm:dnn:phi:exist:2}
\paramANN(\hat{\Phi}_z) \leq \paramANN(\Phi) +  (\kappa+4)^2 d^{ \max\{4, \ExponDim_3\}} \varepsilon^{-\ExponError},
\end{equation}
and  $\mathcal{D} (\hat{\Phi}_{z}) = \mathcal{D} (\hat{\Phi}_{\mathfrak{z}})$.
In the next step we observe that Lemma~\ref{lem:sum:ANN} (with 
$n = N$,
$h_m = \nicefrac{1}{N}$,
$\phi_m =  \phi^{0, d}_{\varepsilon}$,
$\activation = \activation$
for $N, d \in \N$, $\varepsilon \in (0, 1]$, $m \in \{1, 2, \ldots, N\}$
in the notation of Lemma~\ref{lem:sum:ANN}) demonstrates that there exist 	$\mathbf{g}^{N, d}_{\varepsilon} \in \ANNs$, $\varepsilon \in (0, 1]$, $d, N \in \N$, which satisfy for all $N, d \in \N$, $\varepsilon \in (0, 1]$, $x = (x_i)_{i \in \{1, 2, \ldots, N\}} \in \R^{Nd}$ that $\functionANN(\mathbf{g}^{N, d}_{\varepsilon}) \in C(\R^{N d}, \R)$ and 
\begin{align}
\label{eq:monte_carlo:DNN}
( \functionANN (\mathbf{g}^{N, d}_{\varepsilon}))  (x) = \frac{1}{N} \sum_{i = 1}^N ( \functionANN (\phi^{0, d}_{\varepsilon}))(x_i).
\end{align}
This, \eqref{eq:monte_carlo}, and \eqref{eq:intro_hypo} 
ensure that for all $N, d \in \N$, $\varepsilon \in (0, 1]$,  $x = (x_i)_{i \in \{1, 2, \ldots, N\}} \in \R^{Nd}$
 it holds that
\begin{align}
\label{eq:thm:dnn:psi:appr}
\begin{split}
&|g_{N,d}(x) - ( \functionANN (\mathbf{g}^{N, d}_{\varepsilon}) )(x) | \leq \frac{1}{N} \sum_{i=1}^N |\varphi_{0,d}(x_i) - ( \functionANN( \phi^{0, d}_{\varepsilon}))(x_i)|\\
& \leq \frac{\varepsilon \kappa d^{\ExponDim_5}}{N} \sum_{i=1}^N (d^{\theta(\ExponDim_1 + \ExponDim_2)} + \|x_i\|^{\theta}) = \varepsilon \kappa d^{\ExponDim_5} \left[ d^{\theta(\ExponDim_1 + \ExponDim_2)} + \tfrac{1}{N} \textstyle \sum\nolimits_{i=1}^N \displaystyle \|x_i \|^{\theta} \right].
\end{split}
\end{align}
Moreover, note that \eqref{eq:monte_carlo:DNN} and the assumption that $\forall \, \varepsilon \in (0,1],   d \in \N, 
x, y \in \R^d \colon   |( \functionANN (\phi^{ 0, d }_{ \varepsilon }) )(x) - ( \functionANN (\phi^{ 0, d }_{ \varepsilon }) )(y)| \leq \kappa d^{\ExponDim_6} (1  + \|x\|^{\theta} + \|y \|^{\theta})\|x-y\|$ imply that for all $N, d \in \N$, $\varepsilon \in (0, 1]$,  $x = (x_i)_{i \in \{1, 2, \ldots, N\}} \in \R^{Nd}$, 
$y = (y_i)_{i \in \{1, 2, \ldots, N\}} \in \R^{Nd}$ it holds that
\begin{align}
\label{eq:thm:dnn:psi:Lipschitz}
\begin{split}
&|( \functionANN (\mathbf{g}^{N, d}_{\varepsilon}) )(x) - ( \functionANN (\mathbf{g}^{N, d}_{\varepsilon}) )(y)|\\
& \leq \frac{1}{N} \sum_{i=1}^N | ( \functionANN( \phi^{0, d}_{\varepsilon}))(x_i) - ( \functionANN( \phi^{0, d}_{\varepsilon}))(y_i)|\\
 &\leq \frac{\kappa  d^{\ExponDim_6}}{N} \left[ \textstyle \sum\nolimits_{i=1}^N \displaystyle (1 + \|x_i\|^{\theta} + \|y_i \|^{\theta})\|x_i- y_i\| \right].
\end{split}
\end{align}
Next observe that the fact that  $\dims (\idRelu_d) = (d, 2d, d)$ and, e.g., \cite[Proposition~2.16]{GrohsHornungJentzen2019} (with 
$\Psi = \idRelu_d$,
$\Phi_1 = \phi^{0, d}_{\varepsilon}$,
$\Phi_2 \in \{\Phi \in \ANNs \colon \inDimANN(\Phi) = \outDimANN(\Phi) =d  \}$,
$\mathfrak{i} = 2d$ 
 in the notation of \cite[Proposition~2.16]{GrohsHornungJentzen2019})   prove that for every  $N, d \in \N$, $\varepsilon \in (0, 1]$,  $\Phi_1, \Phi_2, \ldots, \Phi_{N} \in \{\Phi \in \ANNs \colon \inDimANN(\Phi) = \outDimANN(\Phi) =d  \}$ with $\mathcal{D}(\Phi_1) = \mathcal{D}(\Phi_2) = \ldots = \mathcal{D}(\Phi_{N})$ there exist
$\Psi_1, \Psi_2, \ldots, \Psi_{N} \in \ANNs$ such that for all  $i \in \{1, 2, \ldots, N\}$ it holds that
$  \functionANN (\Psi_i) \in C(\R^d, \R)$,  $\mathcal{D}(\Psi_i) = \mathcal{D}(\Psi_1)$,
$\paramANN(\Psi_i) \leq 2 (\paramANN(\phi^{0, d}_{\varepsilon}) + \paramANN(\Phi_i))$, and
\begin{align}
\functionANN (\Psi_i) = [ \functionANN (\phi^{0, d}_{\varepsilon})] \circ  [\functionANN (\idRelu_d)] \circ  [\functionANN (\Phi_i)] =  [ \functionANN (\phi^{0, d}_{\varepsilon})]  \circ  [\functionANN (\Phi_i)].
\end{align}
This, \eqref{eq:monte_carlo:DNN}, and Lemma~\ref{lem:sum:ANNs}  assure that for every $N, d \in \N$, $\varepsilon \in (0, 1]$,  $\Phi_1, \Phi_2, \ldots, \Phi_{N} \in \{\Phi \in \ANNs \colon \inDimANN(\Phi) = \outDimANN(\Phi) =d \}$ with $\mathcal{D}(\Phi_1) = \mathcal{D}(\Phi_2) = \ldots = \mathcal{D}(\Phi_{N})$ there exists 
$\Psi \in \ANNs$ such that for all $x \in \R^d$ it holds that
$  \functionANN (\Psi) \in C(\R^d, \R)$, 
$\paramANN(\Psi) \leq 2 N^2 (\paramANN(\phi^{0, d}_{\varepsilon}) + \paramANN(\Phi_1))$, and
\begin{align}
\begin{split}
(\functionANN (\Psi)) (x) &= \frac{1}{N} \sum_{i=1}^{N} ( \functionANN (\phi^{0, d}_{\varepsilon}))\big( (\functionANN (\Phi_i)) (x)\big)\\
& = ( \functionANN (\mathbf{g}^{N, d}_{\varepsilon}))  \big( (\functionANN (\Phi_1)) (x), (\functionANN (\Phi_2)) (x), \ldots, (\functionANN (\Phi_N)) (x)\big).
\end{split}
\end{align}
The assumption that
$\forall \, d \in \N, \varepsilon \in (0, 1] \colon \paramANN( \phi^{ 0, d }_{ \varepsilon } ) 
\leq \kappa d^{ \ExponDim_3 } \varepsilon^{ - \ExponError }$ hence ensures that for every $N, d \in \N$, $\varepsilon \in (0, 1]$,  $\Phi_1, \Phi_2, \ldots, \Phi_{N} \in \{\Phi \in \ANNs \colon \inDimANN(\Phi) = \outDimANN(\Phi) =d \}$ with $\mathcal{D}(\Phi_1) = \mathcal{D}(\Phi_2) = \ldots = \mathcal{D}(\Phi_{N})$ there exists 
$\Psi \in \ANNs$ such that for all $x \in \R^d$ it holds that
$  \functionANN (\Psi) \in C(\R^d, \R)$,   $(\functionANN (\Psi)) (x) =  ( \functionANN (\mathbf{g}^{N, d}_{\varepsilon}))  ( (\functionANN (\Phi_1)) (x), (\functionANN (\Phi_2)) (x),$ $ \ldots, (\functionANN (\Phi_N)) (x))$, and 
\begin{align}
\label{eq:thm:dnn:psi:exist}
\begin{split}
\paramANN(\Psi) \leq 2  N^2  (  \kappa d^{ \ExponDim_3 } \varepsilon^{ - \ExponError }  + \paramANN(\Phi_1)) \leq 2 \max\{\kappa, 1\} N^2  (   d^{ \ExponDim_3 } \varepsilon^{ - \ExponError }  + \paramANN(\Phi_1)).
\end{split}
\end{align}
Furthermore, note that   \eqref{eq:intro_hypo}   and the assumption that $\forall \, d\in \N, \varepsilon \in (0, 1], x, y \in \R^d  \colon 	
|( \functionANN (\phi^{ 0, d }_{ \varepsilon }) )(x) - ( \functionANN (\phi^{ 0, d }_{ \varepsilon }) )(y)| \leq \kappa d^{ \ExponDim_6} (1 + \|x\|^{\theta} + \|y \|^{\theta})\|x-y\|$ demonstrate for all $d \in \N$, $\varepsilon \in (0, 1]$, $x \in \R^d$ that
\begin{align*}
&|\varphi_{0,d}(x) - \varphi_{0, d}(y)| \numberthis\\
&\leq 	| 
\varphi_{ 0, d }(x) 
- 
( \functionANN (\phi^{ 0, d }_{ \varepsilon }) )(x)
|
+ 	| 
\varphi_{ 0, d }(y) 
- 
( \functionANN (\phi^{ 0, d }_{ \varepsilon }) )(y)
| \\
& \quad + |( \functionANN (\phi^{ 0, d }_{ \varepsilon }) )(x) - ( \functionANN (\phi^{ 0, d }_{ \varepsilon } ))(y)|\\
& \leq 	\varepsilon \kappa d^{ \ExponDim_5 }
(
d^{\theta(\ExponDim_1 + \ExponDim_2)} + \| x \|^{ \theta }
) + 	\varepsilon \kappa d^{ \ExponDim_5 }
(
d^{\theta(\ExponDim_1 + \ExponDim_2)} + \| y \|^{ \theta }
) \\
& \quad +  \kappa d^{ \ExponDim_6} (1 + \|x\|^{\theta} + \|y \|^{\theta})\|x-y\|.
\end{align*}
This establishes that for all  $d \in \N$, $x \in \R^d$ it holds that
\begin{align}
\label{eq:thm:dnn:f_0}
\begin{split}
|\varphi_{0,d}(x) - \varphi_{0, d}(y)| &\leq \kappa  d^{ \ExponDim_6} (1 + \|x\|^{\theta} + \|y \|^{\theta})\|x-y\|.
\end{split}
\end{align}
Next observe that the assumption that $\forall \, d\in \N, \varepsilon \in (0, 1], x \in \R^d \colon 	 \allowbreak	\|
( \functionANN (\phi^{ 1, d }_{ \varepsilon }) )(x)    
\|	
\leq 
\kappa ( d^{ \ExponDim_1 + \ExponDim_2 } + \| x \| ) $ and \eqref{eq:intro_hypo} ensure for all $d \in \N$, $\varepsilon \in (0, 1]$, $x \in \R^d$ that
\begin{align}
\begin{split}
\|\varphi_{1, d} (x) \| &\leq 	\| 
\varphi_{ 1, d }(x) 
- 
( \functionANN (\phi^{ 1, d }_{ \varepsilon }) )(x)
\| +  	\|
( \functionANN (\phi^{ 1, d }_{ \varepsilon }) )(x)    
\|\\
&	\leq 		\varepsilon  \kappa d^{\ExponDim_4} (d^{\theta(\ExponDim_1 + \ExponDim_2)}+ \|x\|^{\theta})  + \kappa ( d^{ \ExponDim_1 + \ExponDim_2 } + \| x \| ).
\end{split}
\end{align}
This proves that for all  $d \in \N$, $x \in \R^d$ it holds that
\begin{align}
\label{eq:thm:dnn:f_1}
\|\varphi_{1, d} (x) \| \leq \kappa ( d^{ \ExponDim_1 + \ExponDim_2 } + \| x \| ).
\end{align}
In the next step we note that the H\"older's inequality, the assumption that 	$\forall \, d \in \N \colon [	\int_{\R^d} \|x\|^{2p \theta} \, \nu_d (dx) ]^{\nicefrac{1}{(2p\theta)}}  \leq \kappa d^{\ExponDim_1 + \ExponDim_2}$, and the assumption that $\theta \in [1, \infty)$ assure that for all $d \in \N$ it holds that
\begin{align}
\label{eq:thm:dnn:int}
\begin{split}
\left[	\int_{\R^d} \|x\|^{p (1+\theta)} \, \nu_d (dx) \right]^{\nicefrac{1}{(p(1+\theta))}} 
&\leq \left[	\int_{\R^d} \|x\|^{2p\theta} \, \nu_d (dx) \right]^{\nicefrac{1}{(2p\theta)}}  \leq \kappa d^{\ExponDim_1 + \ExponDim_2}.
\end{split}
\end{align}
Next note that \eqref{eq:approx:Euler}, \eqref{eq:varphi:Euler}, and \eqref{eq:increments} imply that for all $N,  d \in \N$,  $m \in \{1, 2, \ldots, N\}$,   $x \in \R^d$, $n \in \{1, 2, \ldots, N\}$ it holds that 
\begin{align}
\begin{split}
Y^{N, d, m, x}_{n} & = Z^{N, d, m}_{n-1} + Y^{N, d, m, x}_{n-1} + \tfrac{T}{N} \varphi_{1,d}(Y^{N, d, m, x}_{n-1}) \\
&=  Y^{N, d, m, x}_{n-1} + \tfrac{T}{N} \varphi_{1,d}(Y^{N, d, m, x}_{n-1}) +  \mathcal{A}_d W^{d, m}_{\frac{nT}{N}} - \mathcal{A}_d W^{d, m}_{\frac{(n-1)T}{N}}.
\end{split}
\end{align}
The assumption that
	$\forall \, d \in \N, x  \in \R^d \colon  |
\varphi_{ 0, d }( x )| \leq \kappa d^{ \ExponDim_6 }
( d^{ \theta(\ExponDim_1 + \ExponDim_2) } + \| x \|^{ \theta } )$,
  the assumption that $ \forall \, d \in \N \colon
\operatorname{Trace}(A_d)
\leq 
\kappa d^{ 2 \ExponDim_1 }
$,
the assumption that
$\forall \, d \in \N, x, y \in \R^d \colon 	\| 
\varphi_{ 1, d }( x ) 
- 
\varphi_{ 1, d }( y )
\|
\leq 
\kappa 
\| x - y \| $,
 \eqref{eq:thm:dnn:f_0}, \eqref{eq:thm:dnn:f_1}, \eqref{eq:thm:dnn:int}, \eqref{eq:X_processes}, and
Proposition~\ref{prop:monte_carlo_euler} (with 
$d=d$,
$M=N$, 
$n=d$,
$T=T$,
$\kappa = \kappa$,
$\theta = \theta$,
$\ExponDim_0 = \ExponDim_6$,
$\ExponDim_1 = \ExponDim_1 + \ExponDim_2$,
$h = \nicefrac{T}{N}$,
$B = \mathcal{A}_d$,
$p=p$,
$\nu = \nu_d$,
$(\Omega, \mathcal{F}, \P) =  (\Omega, \mathcal{F}, \P)$,
$W^m = W^{d, m}$,
$f_0 = \varphi_{0,d}$,
$f_1 = \varphi_{1, d}$
for $N, d \in \N$
 in the notation of Proposition~\ref{prop:monte_carlo_euler}) hence
 establish that for all $N, d \in \N$ it holds that
\begin{align}
 \begin{split}
&\left(\E \! \left[ \int_{\R^d} \Big|\E[\varphi_{0,d} (X_T^{d, x})] - \tfrac{1}{N} \Big[ \textstyle \sum\nolimits_{i=1}^N \displaystyle \varphi_{0,d} (Y^{N, d, i, x}_{N}) \Big] \Big|^p \, \nu_d (dx) \right] \right)^{\!\nicefrac{1}{p}}\\
& \leq  2^{4\theta +5} | \!\max\{1, T\} |^{\theta +1} |\!\max\{ \kappa, \theta, 1 \}|^{2\theta +3} e^{(6\max\{ \kappa, \theta, 1 \}+5|\!\max\{ \kappa, \theta, 1 \}|^2 T)}  \\
& \quad \cdot  p (p \theta + p +1)^{\theta} d^{\ExponDim_6 + (\ExponDim_1 + \ExponDim_2)(\theta+1)} \big( N^{-\nicefrac{1}{2}} +  N^{-\nicefrac{1}{2}}\big)\\
& =  2^{4\theta +6} |\! \max\{1, T\} |^{\theta +1} |\!\max\{ \kappa, \theta\}|^{2\theta +3} e^{(6\max\{ \kappa, \theta \}+5|\!\max\{ \kappa, \theta \}|^2 T)}  \\
& \quad \cdot  p (p \theta + p +1)^{\theta} d^{\ExponDim_6 + (\ExponDim_1 + \ExponDim_2)(\theta+1)}  N^{-\nicefrac{1}{2}} .
 \end{split}
 \end{align}
This, 
the fact that $\forall \, d\in \N, x \in \R^d \colon 	|
\varphi_{ 0, d }( x )
| 
\leq 
\kappa d^{ \ExponDim_6 }
( d^{\theta (\ExponDim_1 + \ExponDim_2) } + \| x \|^{ \theta } )$,
\eqref{eq:thm:dnn:f_1}, 
\eqref{eq:monte_carlo}, 
and, e.g.,
\cite[Theorem~3.1]{JentzenSalimovaWelti2018} (with $(\Omega, \mathcal{F}, \P) = (\Omega, \mathcal{F}, \P)$, $T=T$, $d=d$, $m=d$, $B= \mathcal{A}_d$, 
$\varphi = \varphi_{0,d}$,
$\mu = \varphi_{1,d}$ for $d \in \N$ in the notation of \cite[Theorem~3.1]{JentzenSalimovaWelti2018})  prove for all $N, d \in \N$ that
\begin{align}
\label{eq:thm:dnn:error}
\begin{split}
&\left(\E \! \left[ \int_{\R^d} \big|u_d(T, x) - g_{N,d} (Y^{N, d, x}_N) \big|^p \, \nu_d (dx) \right] \right)^{\!\nicefrac{1}{p}}\\
&= \left(\E \! \left[ \int_{\R^d} \big|\E[\varphi_{0,d} (X_T^{d, x})] - g_{N,d} (Y^{N, d, x}_N) \big|^p \, \nu_d (dx) \right] \right)^{\!\nicefrac{1}{p}}\\
& = \left(\E \! \left[ \int_{\R^d} \Big|\E[\varphi_{0,d} (X_T^{d, x})] - \tfrac{1}{N} \Big[ \textstyle \sum\nolimits_{i=1}^N \displaystyle \varphi_{0,d} (Y^{N, d, i, x}_{N}) \Big] \Big|^p \, \nu_d (dx) \right] \right)^{\!\nicefrac{1}{p}}\\
& \leq   2^{4\theta +6} | \!\max\{1, T\} |^{\theta +1} |\!\max\{ \kappa, \theta\}|^{2\theta +3} e^{(6\max\{ \kappa, \theta \}+5|\!\max\{ \kappa, \theta \}|^2 T)}  \\
& \quad \cdot  p (p \theta + p +1)^{\theta} d^{\ExponDim_6 + (\ExponDim_1 + \ExponDim_2)(\theta+1)}  N^{-\nicefrac{1}{2}}.
\end{split}
\end{align}
Combining this, \eqref{eq:approx:Euler}, 
\eqref{eq:thm:dnn:Z:int},
\eqref{eq:parameters:Id},
\eqref{eq:thm:dnn:varphi:appr},
\eqref{eq:thm:dnn:varphi:linear},
\eqref{eq:thm:dnn:varphi:lipschitz},
\eqref{eq:thm:dnn:phi:exist:1},
\eqref{eq:thm:dnn:phi:exist:2},
\eqref{eq:thm:dnn:psi:appr},
\eqref{eq:thm:dnn:psi:Lipschitz},
\eqref{eq:thm:dnn:psi:exist}, and
Theorem~\ref{thm:main} (with
$(\Omega, \mathcal{F}, \P) = (\Omega, \mathcal{F}, \P)$,
$\ExponN_0= \nicefrac{1}{2}$,
$\ExponN_1 = 0$,
$\ExponN_2 =2$,
$\ExponError = \ExponError$,
$\ExponDim_0 = \ExponDim_6 + (\ExponDim_1 + \ExponDim_2)(\theta+1)$,
$\ExponDim_1 = \ExponDim_1$,
$\ExponDim_2 = \ExponDim_2$,
$\ExponDim_3 =  \max\{4, \ExponDim_3\}$,
$\ExponDim_4 = \ExponDim_4$,
$\ExponDim_5 = \ExponDim_5$,
$\ExponDim_6 = \ExponDim_6$,
$ \Constant =2^{4\theta +6} | \!\max\{1, T\} |^{\theta +1} \allowbreak |\!\max\{ \kappa, \theta\}|^{2\theta +3} e^{(6\max\{ \kappa, \theta \}+5|\!\max\{ \kappa, \theta \}|^2 T)}    p (p \theta + p +1)^{\theta}$,
$p=p$,
$\theta = \theta$,
$M_N= N$,
$Z^{N, d, m}_n = Z^{N, d, m}_n$,
$f_{N, d} = f_{N, d}$,
$Y^{N, d, x}_l = Y^{N, d, x}_l$,
$\left\| \cdot \right\| = \left\| \cdot \right\|$,
$\nu_d = \nu_d$,
$g_{ N, d } = g_{ N, d }$,
$u_d(x) = u_d(T,x)$,
$\ANNs = \ANNs$,
$\paramANN = \paramANN$,
$\dims = \dims$,
$\mathcal{R} = \functionANN$,
$\SubsetANNs_{d, \varepsilon} = \SubsetANNs_{d, \varepsilon}$,
$\mathbf{f}^{N, d}_{\varepsilon, z} = \mathbf{f}^{N, d}_{\varepsilon, z}$,
$\mathbf{g}^{N, d}_{\varepsilon} = \mathbf{g}^{N, d}_{\varepsilon}$,
$\idRelu_d = \idRelu_d$
for $N, d \in \N$, $m \in \{1, 2, \ldots, N\}$, 
$n \in \{0, 1,  \ldots, N-1\}$,  $l \in \{0, 1, \ldots, N\}$,
$\varepsilon \in (0, 1]$,
$x, z \in \R^d$
in the notation of Theorem~\ref{thm:main})
establish \eqref{eq:kolmogorov:statement}. The proof of Theorem~\ref{thm:dnn:kolmogorov} is thus completed.
\end{proof}

\begin{cor}
	\label{cor:laplacian:lebesgue}
	Let
	$ \varphi_{0,d} \colon \R^d \to \R $, $ d \in \N $,
	and
	$ \varphi_{ 1, d } \colon \R^d \to \R^d $,
	$ d \in \N $,
	be functions,
		let  $\left\| \cdot \right\| \colon (\cup_{d \in \N} \R^d) \to [0, \infty)$ satisfy for all $d \in \N$, $x = (x_1, x_2, \ldots, x_d) \in \R^d$ that $\|x\| = ( \smallsum_{i=1}^d |x_i|^2)^{\nicefrac{1}{2}}$,
	let
	$ T, \kappa \in (0, \infty)$,
	$\ExponError, \ExponDim_1, \ExponDim_2, \ldots, \ExponDim_6  \in [0, \infty)$, $\theta \in [1, \infty)$, $p \in [2, \infty)$, 
	$
	( \phi^{ m, d }_{ \varepsilon } )_{ 
		(m, d, \varepsilon) \in \{ 0, 1 \} \times \N \times (0,1] 
	} 
	\subseteq \ANNs
	$, 	$\activation \in C(\R, \R)$ satisfy for all $x \in \R$ that
	$\activation(x) = \max\{x, 0\}$,
	assume for all
	$ d \in \N $, 
	$ \varepsilon \in (0,1] $, 
	$ m \in \{0, 1\}$,
	$ 
	x, y \in \R^d
	$
	that
	$
	\functionANN( \phi^{ 0, d }_{ \varepsilon } )
	\in 
	C( \R^d, \R )
	$,
	$
	\functionANN( \phi^{ 1, d }_{ \varepsilon } )
	\in
	C( \R^d, \R^d )
	$, 
		$ 
	\paramANN( \phi^{ m, d }_{ \varepsilon } ) 
	\leq \kappa d^{ 2^{(-m)} \ExponDim_3 } \varepsilon^{ - 2^{(-m)}  \ExponError }$,
	$ |( \functionANN (\phi^{ 0, d }_{ \varepsilon }) )(x) - ( \functionANN (\phi^{ 0, d }_{ \varepsilon }) )(y)| \leq \kappa d^{\ExponDim_6} (1   + \|x\|^{\theta} + \|y \|^{\theta})\|x-y\|$, 
	$
	\|
	( \functionANN (\phi^{ 1, d }_{ \varepsilon }) )(x)    
	\|	
	\leq 
	\kappa ( d^{ \ExponDim_1 + \ExponDim_2 } + \| x \| )
	$, $|
	\varphi_{ 0, d }( x )| \leq \kappa d^{ \ExponDim_6 }
	( d^{ \theta(\ExponDim_1 + \ExponDim_2) } + \| x \|^{ \theta } )$,
	$
	\| 
	\varphi_{ 1, d }( x ) 
	- 
	\varphi_{ 1, d }( y )
	\|
	\leq 
	\kappa 
	\| x - y \| 
	$, 	
	and
	\begin{equation}
	\| 
	\varphi_{ m, d }(x) 
	- 
	( \functionANN (\phi^{ m, d }_{ \varepsilon }) )(x)
	\|
	\leq 
	\varepsilon  \kappa d^{\ExponDim_{(5 -m)}} (d^{\theta(\ExponDim_1 + \ExponDim_2)}+ \|x\|^{\theta}) 
	,
	\end{equation}
	and for every $ d \in \N $ let
	$ u_d \colon [0,T] \times \R^{d} \to \R $
	be an 
	at most polynomially growing viscosity solution of
	\begin{equation}
	\begin{split}
	( \tfrac{ \partial }{\partial t} u_d )( t, x ) 
	& = 
	( \tfrac{ \partial }{\partial x} u_d )( t, x )
	\,
	\varphi_{ 1, d }( x )
	+
	\textstyle
	\sum\limits_{ i = 1 }^d
	\displaystyle
	( \tfrac{ \partial^2 }{ \partial x_i^2  } u_d )( t, x )
	\end{split}
	\end{equation}
	with $ u_d( 0, x ) = \varphi_{ 0, d }( x ) $
	for $ ( t, x ) \in (0,T) \times \R^d $ 	(cf.~Definition~\ref{Def:ANN} and Definition~\ref{Definition:ANNrealization}).
	Then 
	there exist
	$
	c \in \R
	$ and
	$
	( 
	\Psi_{ d, \varepsilon } 
	)_{ (d , \varepsilon)  \in \N \times (0,1] } \subseteq \ANNs
	$
	such that
	for all 
	$
	d \in \N 
	$,
	$
	\varepsilon \in (0,1] 
	$
	it holds that
	$
	\mathcal{R}( \Psi_{ d, \varepsilon } )
	\in C( \R^{ d }, \R )
	$, $[
	\int_{ [0, 1]^d }
	|
	u_d(T, x) - ( \mathcal{R} (\Psi_{ d, \varepsilon }) )( x )
	|^p
	\,
dx
	]^{ \nicefrac{ 1 }{ p } }
	\leq
	\varepsilon$, 
	and
	\begin{align*}
	\label{eq:cor:statement}
	&\paramANN( \Psi_{ d, \varepsilon } ) \leq c \varepsilon^{-(\ExponError +6)} \numberthis\\
	& \cdot d^{6[\ExponDim_6 + ( \max\{\ExponDim_1, \nicefrac{1}{2}\} + \ExponDim_2)(\theta+1)] + \max\{4, \ExponDim_3\}  +  \ExponError \max\{\ExponDim_5 + \theta ( \max\{\ExponDim_1, \nicefrac{1}{2}\} + \ExponDim_2), \ExponDim_4 + \ExponDim_6 + 2\theta ( \max\{\ExponDim_1, \nicefrac{1}{2}\} + \ExponDim_2)\}}  .
	\end{align*}
\end{cor}
\begin{proof}[Proof of Corollary~\ref{cor:laplacian:lebesgue}]
Throughout this proof for every $ d \in \N $ 
let $ \lambda_d \colon \mathcal{B}(\R^d) \to [0, \infty]$ be the Lebesgue-Borel measure on $\R^d$ and let  $\nu_d \colon  \mathcal{B}(\R^d) \to [0,1]$ be the function which satisfies for all $B \in \mathcal{B}(\R^d)$ that
\begin{equation}
\label{eq:measure:def}
\nu_d(B) = \lambda_{d}(B \cap [0, 1]^d).
\end{equation}
Observe that \eqref{eq:measure:def} implies that for all $d \in \N$ it holds that $\nu_d$ is a probability measure on $\R^d$. This  and \eqref{eq:measure:def} ensure that for all $d \in \N$, $g \in C(\R^d, \R)$ it holds that
\begin{equation}
\label{eq:integral:equiv}
\int_{\R^d} |g(x)| \, \nu_d(dx) = \int_{ [0,1]^d } |g(x)| \, dx.
\end{equation}		
Combining this with,  e.g., \cite[Lemma~3.15]{GrohsWurstemberger2018} demonstrates that for all $d \in \N$ it holds that
\begin{equation}
\begin{split}
\int_{\R^d} \|x\|^{2p \theta} \, \nu_d (dx)
& = 
\int_{[0,1]^d} 
\|x\|^{2p \theta}
\, dx \leq d^{p \theta }.
\end{split}
\end{equation}
This assures for all $d \in \N$ that
\begin{equation}
\label{eq:cor:measure}
\left[\int_{\R^d} \|x\|^{2p \theta} \, \nu_d (dx)\right]^{\nicefrac{1}{(2p \theta)}} \leq d^{\nicefrac{1}{2}} \leq \max\{\kappa, 1\} d^{ \max\{\ExponDim_1, \nicefrac{1}{2}\} + \ExponDim_2}.
\end{equation} 
Moreover, note that for all $d \in \N$ it holds that
\begin{equation}
\operatorname{Trace}(\idMatrix_d)
\leq d \leq \max\{\kappa, 1\} d^{2 \max\{\ExponDim_1, \nicefrac{1}{2}\}}
\end{equation}
(cf.~Definition~\eqref{Definition:identityMatrix}).
This, \eqref{eq:cor:measure}, and Theorem~\ref{thm:dnn:kolmogorov} (with 
$A_d = \idMatrix_d$,
$\left\| \cdot \right\| = \left\| \cdot \right\|$,
$\nu_d = \nu_d$,
$\varphi_{ 0, d } = \varphi_{ 0, d }$,
$\varphi_{ 1, d } = \varphi_{ 1, d }$,
$T =T$,
$\kappa = \max\{\kappa, 1\}$,
$\ExponError = \ExponError$,
$\ExponDim_1 = \max\{\ExponDim_1, \nicefrac{1}{2}\}$,
$\ExponDim_2 = \ExponDim_2$,
$\ExponDim_3 = \ExponDim_3$,
$\ExponDim_4 = \ExponDim_4$,
$\ExponDim_5 = \ExponDim_5$,
$\ExponDim_6 = \ExponDim_6$,
$\theta = \theta$,
$p = p$,
$\phi^{0, d}_{\varepsilon} = \phi^{0, d}_{\varepsilon}$,
$\phi^{1, d}_{\varepsilon} = \phi^{1, d}_{\varepsilon}$,
$a = a$, $u_d = u_d$ for $d \in \N$
in the notation of Theorem~\ref{thm:dnn:kolmogorov}) establish \eqref{eq:cor:statement}. The proof of Corollary~\ref{cor:laplacian:lebesgue} is thus completed. 
\end{proof}

\begin{cor}
	\label{cor:dnn:kolmogorov}
	Let 
	$
	A_d = ( A_{ d, i, j } )_{ (i, j) \in \{ 1, \dots, d \}^2 } \in \R^{ d \times d }
	$,
	$ d \in \N $,
	be symmetric positive semidefinite matrices, 
let $\left\| \cdot \right\| \colon (\cup_{d \in \N} \R^d) \to [0, \infty)$ satisfy for all $d \in \N$, $x = (x_1, x_2, \ldots, x_d) \in \R^d$ that $\|x\| = ( \smallsum_{i=1}^d |x_i|^2)^{\nicefrac{1}{2}}$,
	for every $ d \in \N $ 
 let $ \nu_d \colon \mathcal{B}(\R^d) \to [0,1]$ be a probability measure on $\R^d$,
	let
	$ \varphi_{0,d} \colon \R^d \to \R $, $ d \in \N $,
	and
	$ \varphi_{ 1, d } \colon \R^d \to \R^d $,
	$ d \in \N $,
	be functions,
	let
	$ T, \kappa,  p \in (0, \infty)$, $\theta \in [1, \infty)$, 
	$
	( \phi^{ m, d }_{ \varepsilon } )_{ 
		(m, d, \varepsilon) \in \{ 0, 1 \} \times \N \times (0,1] 
	} 
	\subseteq \ANNs
	$, 	$\activation \in C(\R, \R)$ satisfy for all $x \in \R$ that
	$\activation(x) = \max\{x, 0\}$,
	assume for all
	$ d \in \N $, 
	$ \varepsilon \in (0,1] $, 
	$m \in \{0, 1\}$,
	$ 
	x, y \in \R^d
	$
	that
	$
	\functionANN( \phi^{ 0, d }_{ \varepsilon } )
	\in 
	C( \R^d, \R )
	$,
	$
	\functionANN( \phi^{ 1, d }_{ \varepsilon } )
	\in
	C( \R^d, \R^d )
	$,
	$
	|
	\varphi_{ 0, d }( x )
	| 
	+
	\operatorname{Trace}(A_d)
	\leq 
	\kappa d^{ \kappa }
	( 1 + \| x \|^{ \theta })
	$,
	$[	\int_{\R^d} \|x\|^{2 \max\{p, 2\} \theta} \, \nu_d (dx) ]^{\nicefrac{1}{(2 \max\{p, 2\} \theta)}}  \leq \kappa d^{\kappa}$,
	$ 
	\paramANN( \phi^{ m, d }_{ \varepsilon } ) 
	\leq \kappa d^{ \kappa } \varepsilon^{ - \kappa }
	$,
		 $ |( \functionANN (\phi^{ 0, d }_{ \varepsilon }) )(x) \allowbreak - ( \functionANN (\phi^{ 0, d }_{ \varepsilon }) )(y)| \leq \kappa d^{\kappa} (1 + \|x\|^{\theta} + \|y \|^{\theta})\|x-y\|$, 
	$
	\|
	( \functionANN (\phi^{ 1, d }_{ \varepsilon }) )(x)    
	\|	
	\leq 
	\kappa ( d^{ \kappa } + \| x \| )
	$,
	$
\| 
\varphi_{ 1, d }( x ) 
- 
\varphi_{ 1, d }( y )
\|
\leq 
\kappa 
\| x - y \|
$,
	and
	\begin{equation}
	\| 
	\varphi_{ m, d }(x) 
	- 
	( \functionANN (\phi^{ m, d }_{ \varepsilon }) )(x)
	\|
	\leq 
	\varepsilon \kappa d^{ \kappa }
	(
	1 + \| x \|^{ \theta }
	)
	,
	\end{equation}
	and for every $ d \in \N $ let
	$ u_d \colon [0,T] \times \R^{d} \to \R $
	be an 
	at most polynomially growing viscosity solution of
	\begin{equation}
	\begin{split}
	( \tfrac{ \partial }{\partial t} u_d )( t, x ) 
	& = 
	( \tfrac{ \partial }{\partial x} u_d )( t, x )
	\,
	\varphi_{ 1, d }( x )
	+
	\textstyle
	\sum\limits_{ i, j = 1 }^d
	\displaystyle
	A_{ d, i, j }
	\,
	( \tfrac{ \partial^2 }{ \partial x_i \partial x_j } u_d )( t, x )
	\end{split}
	\end{equation}
	with $ u_d( 0, x ) = \varphi_{ 0, d }( x ) $
	for $ ( t, x ) \in (0,T) \times \R^d $ (cf.~Definition~\ref{Def:ANN} and Definition~\ref{Definition:ANNrealization}). 
	Then 
	there exist
		$
	c \in \R
	$ and
	$
	( 
	\Psi_{ d, \varepsilon } 
	)_{ (d , \varepsilon)  \in \N \times (0,1] } \subseteq \ANNs
	$
	such that
	for all 
	$
	d \in \N 
	$,
	$
	\varepsilon \in (0,1] 
	$
	it holds that
	$
	\paramANN( \Psi_{ d, \varepsilon } ) 
	\leq
	c \, d^c \varepsilon^{ - c } 
	$,
	$
	\functionANN( \Psi_{ d, \varepsilon } )
	\in C( \R^{ d }, \R )
	$,
	and
	\begin{equation}
	\left[
	\int_{ \R^d }
	|
	u_d(T,x) - ( \functionANN (\Psi_{ d, \varepsilon }) )( x )
	|^p
	\,
	\nu_d(dx)
	\right]^{ \nicefrac{ 1 }{ p } }
	\leq
	\varepsilon 
	.
	\end{equation}
\end{cor}

\subsubsection*{Acknowledgments}
This project has been partially  supported 
through the research grant~$ 200020\_175699 $ 
funded by the Swiss National Science Foundation.

\bibliographystyle{acm}
\bibliography{../../bibfile}

\end{document}